\documentclass[12pt]{amsart}
\usepackage{color}

\usepackage{amsmath}
\usepackage{amssymb}
\usepackage{extpfeil}
\usepackage[normalem]{ulem}
\usepackage[all]{xy}
\usepackage{longtable}
\usepackage{titletoc}
\usepackage{mathrsfs}
\usepackage{mathtools}
\usepackage{amscd}
\usepackage{appendix}
\usepackage{amsthm}
\usepackage{amsfonts}
\usepackage{tikz-cd}
\usepackage{array}
\usepackage{graphicx}
\usepackage{hyperref}
\usepackage{longtable}
\usepackage{xtab}
\usepackage{soul,xcolor}
\usepackage{upgreek}
\usepackage[total={15.875truecm,22.86truecm},centering]{geometry}
\usepackage{stackrel}

\usepackage{eucal}

\newtheorem{theorem}[equation]{Theorem}
\newtheorem{lemma}[equation]{Lemma}
\newtheorem{proposition}[equation]{Proposition}
\newtheorem{corollary}[equation]{Corollary}
\newtheorem{conjecture}[equation]{Conjecture}

\theoremstyle{definition}
\newtheorem{definition}[equation]{Definition}
\newtheorem{example}[equation]{Example}

\theoremstyle{remark}
\newtheorem{remark}[equation]{Remark}

\numberwithin{equation}{subsection}

\DeclareMathAlphabet{\mathpzc}{OT1}{pzc}{m}{n}
\DeclareMathAlphabet{\matheur}{U}{eur}{m}{n}

\newenvironment{Subsubsec}[1]{\vspace{\topsep}
\noindent\refstepcounter{equation}\theequation.\ \textit{#1.}}{\vspace{\topsep}}

\definecolor{lava}{RGB}{207,16,32}
\definecolor{purple}{RGB}{148,0,211}

\newcommand{\FF}{\mathbb{F}}
\newcommand{\ZZ}{\mathbb{Z}}
\newcommand{\QQ}{\mathbb{Q}}

\newcommand{\LL}{\mathbb{L}}
\newcommand{\TT}{\mathbb{T}}
\newcommand{\GG}{\mathbb{G}}

\newcommand{\CC}{\mathbb{C}}
\newcommand{\OO}{\mathbb{O}}
\newcommand{\PP}{\mathbb{P}}
\newcommand{\KK}{\mathbb{K}}

\newcommand{\NN}{\mathbb{N}}
\newcommand{\MM}{\mathbb{M}}

\newcommand{\bm}{\mathbf{m}}

\newcommand{\bK}{\mathbf{K}}

\newcommand{\bX}{\mathbf{X}}

\newcommand{\bz}{\mathbf{z}}

\newcommand{\bI}{\mathbf{I}}
\newcommand{\bU}{\mathbf{U}}

\newcommand{\cO}{\mathcal{O}}
\newcommand{\cR}{\mathcal{R}}
\newcommand{\cT}{\mathcal{T}}

\newcommand{\eK}{\matheur{K}}

\newcommand{\eM}{\matheur{M}}
\newcommand{\eN}{\matheur{N}}
\newcommand{\eF}{\matheur{F}}

\newcommand{\eC}{\matheur{C}}

\newcommand{\eI}{\matheur{I}}

\newcommand{\ek}{\matheur{k}}
\newcommand{\eum}{\matheur{m}}
\newcommand{\en}{\matheur{n}}

\DeclareMathOperator{\Aut}{Aut}
\DeclareMathOperator{\Lie}{Lie}
\DeclareMathOperator{\Ker}{Ker}
\DeclareMathOperator{\GL}{GL}
\DeclareMathOperator{\Mat}{Mat}
\DeclareMathOperator{\Cent}{Cent}
\DeclareMathOperator{\End}{End}

\DeclareMathOperator{\Gal}{Gal}
\DeclareMathOperator{\Hom}{Hom}
\DeclareMathOperator{\Res}{Res}

\DeclareMathOperator{\trdeg}{tr.deg}
\DeclareMathOperator{\wt}{wt}

\DeclareMathOperator{\Emb}{Emb}
\DeclareMathOperator{\Inf}{Inf}
\DeclareMathOperator{\Spec}{Spec}
\DeclareMathOperator{\Div}{Div}

\DeclareMathOperator{\Supp}{Supp}
\DeclareMathOperator{\rank}{rank}
\DeclareMathOperator{\ord}{ord}

\DeclareMathOperator{\Gr}{Gr}

\newcommand{\ok}{\bar{k}}

\newcommand{\sep}{\mathrm{sep}}

\newcommand{\power}[2]{{#1 [\![ #2 ]\!]}}
\newcommand{\laurent}[2]{{#1 (\!( #2 )\!)}}

\newcommand{\bomega}{\boldsymbol{\omega}}

\newcommand{\Ifk}{\mathfrak{I}}
\newcommand{\Jfk}{\mathfrak{J}}

\newcommand{\Mfk}{\mathfrak{M}}

\newcommand{\Pfk}{\mathfrak{P}}

\newcommand{\Fcal}{\CMcal{F}}
\newcommand{\Gcal}{\CMcal{G}}
\newcommand{\Hcal}{\CMcal{H}}
\newcommand{\Lcal}{\CMcal{L}}
\newcommand{\Mcal}{\CMcal{M}}

\newcommand{\Ocal}{\CMcal{O}}
\newcommand{\Pcal}{\CMcal{P}}

\newcommand{\Tcal}{\CMcal{T}}
\newcommand{\Jcal}{\CMcal{J}}

\newcommand{\Xcal}{\CMcal{X}}

\newcommand{\Mscr}{\mathscr{M}}

\def\rank{\operatorname{rank}}

\newcommand{\id}{\operatorname{id}}

\newcommand{\divv}{\operatorname{div}}

\newcommand{\Frob}{\operatorname{Frob}}

\newcommand{\assign}{\mathrel{\vcenter{\baselineskip0.5ex \lineskiplimit0pt
                     \hbox{\scriptsize.}\hbox{\scriptsize.}}}%
                     =}
\newcommand{\rassign}{=%
                     \mathrel{\vcenter{\baselineskip0.5ex \lineskiplimit0pt
                     \hbox{\scriptsize.}\hbox{\scriptsize.}}}%
                     }

\begin{document}

\title[Shimura's conjecture on period symbols]{Function field analogue of
 Shimura's conjecture \\ on period symbols}
\author[W.\ D.\ Brownawell]{W.\ Dale Brownawell}
\address{Department of Mathematics, Penn State University, University Park, PA 16802, U.S.A.}
\email{wdb@math.psu.edu}

\author[C.-Y.\ Chang]{Chieh-Yu Chang}
\address{Department of Mathematics, National Tsing Hua University and National Center for Theoretical Sciences, Hsinchu City 30042, Taiwan
  R.O.C.}
\email{cychang@math.nthu.edu.tw}

\author[M.\ A.\ Papanikolas]{Matthew A.\ Papanikolas}
\address{Department of Mathematics, Texas A{\&}M University, College Station,
TX 77843-3368, U.S.A.} \email{papanikolas@tamu.edu}

\author[F.-T.\ Wei]{Fu-Tsun Wei}
\address{Department of Mathematics, National Tsing Hua University and National Center for Theoretical Sciences, Hsinchu City 30042, Taiwan
  R.O.C.}
\email{ftwei@math.nthu.edu.tw}

\thanks{The second author is partially supported by NSTC grant 107-2628-M-007-002-MY4 and 111-2628-M-007-002.
The fourth author is partially supported by NSTC grant 109-2115-M-007-017-MY5
and the National Center for Theoretical Sciences.
}

\subjclass[2020]{Primary 11J93; Secondary 11G09, 11G15}

\date{August 18, 2022}

\begin{abstract}
In this paper we introduce the notion of Shimura's period symbols over function fields in positive characteristic and establish their fundamental properties. We further formulate and prove a function field analogue of Shimura's conjecture on the algebraic independence of period symbols. Our results enable us to verify the algebraic independence of the coordinates of any nonzero period vector of an abelian $t$-module with complex multiplication whose CM type is non-degenerate and defined over an algebraic function field. This is an extension of Yu's work on Hilbert-Blumenthal $t$-modules.
\end{abstract}

\keywords{period symbols, CM types, $t$-modules, $t$-motives}

\maketitle
\tableofcontents
\section{Introduction}

\subsection{Motivation}\label{Sec:Motivation}
The aim of this paper is to formulate and prove a function field analogue of Shimura's conjecture on the algebraic independence of period symbols. To state Shimura's conjecture, we first review the theory of Shimura's period symbols developed in~\cite{Sh98}.

Let $K/\QQ$ be a CM field of degree $2n$, and denote by $J_{K}$ the set of all embeddings of $K$ into $\CC$ and by $I_{K}$ the free abelian group generated by $J_{K}$.
Let $\Phi \subset J_K$ be a CM type of $K$, i.e., $\Phi \lor \Phi
\rho=J_{K}$, where $\rho$ is complex conjugation. For each $\sigma\in \Phi$, Shimura constructed a constant $p_{K}(\sigma,\Phi)\in \CC^{\times}$ via integration of a holomorphic differential form on an abelian variety defined over $\bar{\QQ}$ with CM type $(K,\Phi)$. He showed that it is unique up to algebraic multiples, and so we view the special value $p_{K}(\sigma,\Phi)$ as an element of $\CC^{\times}/\bar{\QQ}^{\times}$. For $\tau\in \Phi\rho$, Shimura defined
\[
p_{K}(\tau,\Phi)\assign p_{K}(\tau \rho, \Phi)^{-1}\in \CC^{\times}/\bar{\QQ}^{\times},
\]
and further showed that $p_{K}$ can be extended to a bilinear map
\[
p_{K}:I_{K}\times I_{K}\rightarrow \CC^{\times}/\bar{\QQ}^{\times}
\]
that has the following properties:
\begin{enumerate}
    \item[(i)] $p_{K}(\xi \rho,\eta)=p_{K}(\xi,\eta \rho)=p_{K}(\xi,\eta)^{-1}$ for every $\xi,\eta\in I_{K}$;
    \item[(ii)] $p_{K}(\Res_{L/K}(\xi), \eta)=p_{L}(\xi,\Inf_{L/K}(\eta))$ for $\xi\in I_{L}$ and $\eta\in I_{K}$, where $L$ is a CM field containing $K$;
\item[(iii)] $p_L(\Inf_{L/K}(\xi),\eta)=p_K(\xi, \Res_{L/K}(\eta))$ for $\xi\in I_{K}$ and $\eta\in I_{L}$, where $L$ is a CM field containing $K$;
\item[(iv)] $p_{K'}(\gamma \xi,\gamma \eta)=p_{K}(\xi,\eta)$ if $\gamma:K' \stackrel{\sim}{\to} K$ is an isomorphism.
\end{enumerate}
We note that property~(i) shows that if $\{\sigma_i\}_{i=1}^n$ is a CM type for $K$, then
\[E_{K}\assign
\overline{\QQ}\left(p_K(\iota,\id)\mid \iota\in J_K\right) = \overline{\QQ}\left(p_K(\sigma_1,\id)),\ldots,p_K(\sigma_n,\id)\right),
\]
and so
\[
\trdeg_{\overline{\QQ}}\overline{\QQ}\left(p_K(\iota,\id)\mid \iota\in J_K \right) = \trdeg_{\overline{\QQ}}\overline{\QQ}\left(p_K(\sigma_1,\id)),\ldots,p_K(\sigma_n,\id)\right) \le n.
\]
In~\cite[p.~319]{Sh80}, Shimura proposed the following conjecture/question
asking whether we actually have equality here.

\begin{conjecture}\label{C:Shimura}
Let $K$ be a CM field of degree $2n$. If $\{ \sigma_{i}\}_{i=1}^n$ is a CM type of $K$, then
\[p_{K}(\sigma_{1},\id),\ldots,p_{K}(\sigma_{n}, \id) \]
are algebraically independent over $\overline{\QQ}$.
\end{conjecture}

In other words, Shimura's conjecture asserts that, no matter the CM type $\Phi$ of $K$,  the set
\[
\left\{ p_{K}(\sigma_{1},\id),\ldots,p_{K}(\sigma_{n}, \id) \mid \sigma_{i}\in \Phi \right\}
\]
is a generating transcendence basis of $E_{K}$ over $\overline{\QQ}$. Thus property (i) above generates all $\overline{\QQ}$-algebraic relations on the period symbols $\{p_K(\iota,\id) \mid {\iota\in J_K}\}$. In~\cite[p.~270]{Yo03}, Yoshida described the following conjecture, which also includes $\pi$, and  derived it from Deligne's conjecture~\cite{De82}, which is a special case of Grothendieck's period conjecture (see~\cite{Ay14} and prologue of \cite{HuberWustholz}).

\begin{conjecture}\label{C:Yoshida}
Let $K$ be a CM field which is Galois over $\QQ$, and let $\Phi$ be a CM type of $K$. Then we have
\[
\trdeg_{\overline{\QQ}} \overline{\QQ}\left(\pi,\, p_{K}(\id,\sigma)\mid \sigma\in \Phi \right)=1+ [K:\QQ]/2.
\]
\end{conjecture}

In what follows, we mention that Conjecture~\ref{C:Yoshida} implies the Lang-Rohrlich conjecture on special $\Gamma$-values.  By {\it{special $\Gamma$-values}} we mean the values of Euler's $\Gamma$-function at proper fractions. The  Lang-Rohrlich conjecture~\cite{Lang} asserts that all $\overline{\QQ}$-algebraic relations among  special $\Gamma$-values and $2\pi\sqrt{-1}$ are explained by the standard functional equations of $\Gamma$-function. An equivalent formulation of the Lang-Rohrlich conjecture is that for any integer $n>2$, we have
\[
\trdeg_{\overline{\QQ}} \overline{\QQ}\left(\pi, \Gamma(\frac{1}{n}), \Gamma(\frac{2}{n}),\ldots,\Gamma(\frac{n-1}{n}) \right)=1+ \frac{\phi(n)}{2}.
\]
Note that if we let $K_n$ be the $n$th cyclotomic extension over $\QQ$, then $K_{n}/ \QQ$ is an abelian CM field with $[K_n:\QQ]=\phi(n)$. Using Deligne's theorem~\cite[Thm.~4.7]{A82} and Anderson's work~\cite{A82}, one can show that
\[
\trdeg_{\overline{\QQ}} \overline{\QQ}\left(\pi, p_{K_n}(\id,\sigma)\mid \sigma\in \Phi \right)= \trdeg_{\overline{\QQ}} \overline{\QQ}\left(\pi, \Gamma(\frac{1}{n}), \Gamma(\frac{2}{n}),\ldots,\Gamma(\frac{n-1}{n}) \right)
\] for any CM type $\Phi$ of $K_{n}$. So Conjecture~\ref{C:Yoshida} implies the Lang-Rohrlich conjecture. The goal of this paper is to prove function field analogues of Conjecture~\ref{C:Shimura} and Conjecture~\ref{C:Yoshida}.

\subsection{\texorpdfstring{CM types and CM dual $t$-motives}{CM types and CM dual t-motives}}

In the theory of global function fields, Anderson~\cite{Anderson86} introduced abelian $t$-modules (and their $t$-motives) to serve as the analogue of commutative algebraic groups in the classical theory. As with classical motives, $t$-motives possess  linear algebra data and they capture Betti, de~Rham, and $v$-adic Tate module realizations. The primary algebro-geometric objects here will be \emph{uniformizable} abelian $t$-modules with complex multiplication that are analogous to CM abelian varieties, and our period symbols arise from the periods of these CM $t$-modules as in the case of Shimura's period symbols. Uniformizability is automatic in the classical case for abelian varieties over $\CC$.

Let $\FF_q$ be a finite field with $q$ elements, and fix a variable $t$ over $\FF_q$.
By a \emph{CM field} (over the rational function field $\FF_q(t)$), we mean a finite separable extension $\eK$ over $\FF_q(t)$ containing a ``totally real'' subfield $\eK^+$ (i.e., the infinite place $\infty$ of $\FF_q(t)$ splits completely in $\eK^+$) so that every place of $\eK^+$ lying above $\infty$ is non-split in $\eK$ (cf.\ Definition~\ref{Def: CM fields}),  namely it has a unique extension to $\eK$.
We call $\eK^+$ the \textit{maximal totally real subfield of $\eK$}.
Unlike the classical case, the field extension $\eK/\eK^+$ may not be Galois, and the degree of $\eK$ over $\eK^+$ can be arbitrarily large.

As $\eK$ (resp.\ $\eK^+$) corresponds to a smooth projective curve $X$ (resp.\ $X^+$) over the constant field $\FF_\eK$ of $\eK$ (resp.\ $\FF_q$),
we shall introduce our CM types of $\eK$ in a geometric way as follows.
Fix another variable $\theta$ over $\FF_q$. Let $k \assign \FF_q(\theta)$, $k_\infty \assign \FF_q(\!(1/\theta)\!)$, and $\CC_\infty \assign\widehat{\bar{k}}_\infty$, the completion of a chosen algebraic closure of $k_\infty$.
Let $\bar{k}$ be the algebraic closure of $k$ in $\CC_\infty$.
Viewing $X$ as an $\FF_q$-scheme,
the embedding $\FF_q(t)\hookrightarrow \eK$ corresponds to a finite $\FF_q$-morphism $\pi_{X/\PP^1}: X\rightarrow \PP^1$.
Let $\bX$ (resp.\ $\pi_{\bX/\PP^1}$) be the base change of $X$ (resp.\ $\pi_{X/\PP^1}$) from $\FF_q$ to $\CC_\infty$.
We set $J_\eK \assign \pi_{\bX/\PP^1}^{-1}(\theta)$ (i.e.\ the preimage of the point $t=\theta$ on the projective $t$-line $\PP^1$ over $\CC_\infty$).
Let $\pi_{X/X^+}:X\rightarrow X^+$ be the $\FF_q$-morphism corresponding to the embedding $\eK^+\hookrightarrow \eK$, and $\pi_{\bX/\bX^+}$ be the base change of $\pi_{X/X^+}$ to $\CC_\infty$.

\begin{definition}
\text{\rm (cf.\ Definition~\ref{defn: CM type})} Let $\eK$ be a CM field over $\FF_q(t)$.
A \emph{generalized CM type of $\eK$} is an effective divisor $\Xi$ of $\bX$ of the following form
\[
\Xi = \sum_{\xi \in J_\eK} m_\xi \xi,
\]
such that there exists a positive integer ${\wt}(\Xi)$ satisfying
\begin{equation}\label{eqn: wt cond}
{\wt}(\Xi) = \sum_{\substack{\xi\in J_{\eK} \\ \pi_{\bX/\bX^+}(\xi) = \xi^+}} m_\xi, \quad \forall \xi^+ \in J_{\eK^+}.
\end{equation}
We call $\Xi$ a \emph{CM type of $\eK$} if ${\wt}(\Xi) = 1$.
\end{definition}

\begin{remark}
Let $\Emb(\eK,\CC_\infty)$ be the set of $\FF_q$-linear embeddings from $\eK$ to $\CC_\infty$ sending $t$ to $\theta$. ``Evaluating at points in $J_\eK$'' gives a one-to-one correspondence $(\xi \leftrightarrow \nu_\xi)$ between $J_\eK$ and $\Emb(\eK,\CC_\infty)$ (see Section~\ref{sec: CM types}). In particular, for each CM type $\Xi$ of $\eK$, we may write $\Xi = \sum_{i=1}^d \xi_i$ where $d = [\eK^+:\FF_q(t)]$, and obtain
\[
\Emb(\eK^+,\CC_\infty) = \bigl\{\nu_{\xi_1}\big|_{\eK^+},\ldots,\nu_{\xi_d}\big|_{\eK^+} \bigr\}.
\]
Under this identification between $J_\eK$ and $\Emb(\eK,\CC_\infty)$, our CM types are analogous to classical ones.
\end{remark}

Just as in the classical case, our period symbols are closely connected with periods of abelian $t$-modules with CM, but we make this precise through the theory of \emph{dual $t$-motives}, which were introduced in~\cite{ABP} and are dual to the $t$-motives in~\cite{Anderson86}. The category of uniformizable abelian $t$-modules is equivalent to the category of uniformizable dual $t$-motives (see~\cite{Mau21}). Thus the investigation of CM uniformizable abelian $t$-modules is transformed to the study of CM uniformizable dual $t$-motives (see Proposition~\ref{prop: comp-dR-pairing}).
Recall that in the classical case, for any CM type $(K,\Phi)$, Shimura was able to contruct a CM abelian variety with CM type $(K,\Phi)$. Our initial step is to establish the analogue of Shimura's construction following~\cite{Sinha97, BP02, ABP}.

Let $\eK$ be a CM field and $\Xi$ be a (generalized) CM type of $\eK$.
We use indifferently the terminology that ``$\Xi$ is a (generalized) CM type of $\eK$'' and that ``$(\eK,\Xi)$ is a (generalized) CM type.''
In Section~\ref{sec: CM dual t-motives}, we construct a CM dual $t$-motive with given generalized CM type $(\eK,\Xi)$. However, when the CM field in question is not geometric, there are complexities in proving that the constructed underlying space is a dual $t$-motive, and we work out all of the details in the appendix. For simplicity,
here we omit the precise definitions and refer readers to Definition~\ref{defn: CM dual t-motives} and \eqref{E:M(W,h} when $\FF_\eK = \FF_q$ (resp.\ Definition~\ref{defnB: CM dual t-motives} and \eqref{e:Appendix M(W,h)} when $\FF_\eK \supsetneq \FF_q$).

\begin{remark}
Hartl and Singh \cite{HS20, HS21}, introduced the notion of ``CM'' Anderson $t$-motives for semi-simple algebras $K$ and define ``CM-type'' by using their Hodge-Pink weights.
Through the comparison in Proposition~\ref{prop: HPW-CM},
our CM types seem analogous to Hartl and Singh's,
although Theorem \ref{thm: Uniform} and Proposition~\ref{prop: being CM} show that ours involve the following three properties instead:
\begin{itemize}
    \item[(1)] $K$ is a CM field;
    \item[(2)] $\Xi$ satisfies \eqref{eqn: wt cond};
    \item[(3)] $\eM$ is pure.
\end{itemize}
These properties guarantee in our setting that the following essential result holds (see Proposition~\ref{L:IsogW1W2} and Theorem \ref{thm: Uniform}):
\end{remark}

\begin{theorem}\label{thm: CM prop}
Let $\KK$ be an algebraically closed field with $\bar{k} \subset \KK \subset \CC_\infty$.
\begin{itemize}
    \item[(1)] Any two CM dual $t$-motives over $\KK$ with the same (generalized) CM type must be isogeneous over $\KK$.
    \item[(2)] Every CM dual $t$-motive over $\KK$ is uniformizable (and pure).
\end{itemize}
\end{theorem}

\begin{remark}
The existence of a CM dual $t$-motive defined over $\ok$ with any given generalized CM type follows
directly from the geometric construction in \eqref{E:M(W,h} and \eqref{e:Appendix M(W,h)}.
Theorem~\ref{thm: CM prop}~(1) then ensures in particular that every CM dual $t$-motive is isomorphic to one defined over $\ok$.
\end{remark}

We note that uniformization of $t$-modules,  $t$-motives or dual $t$-motives is a non-trivial problem, and the situation in this paper is entirely different from classical approaches. When the CM field is a Carlitz cyclotomic extension arising from the study of special geometric $\Gamma$-values, Sinha~\cite{Sinha97} first constructed a CM $t$-module associated with the CM field in question, and successfully proved the uniformization using shtuka functions. Sinha's approaches were later on applied in \cite{BP02}. In \cite{ABP}, the framework above was transferred to the study of dual $t$-motives, and the uniformization issue was proved by giving explicit rigid analytic trivializations using generalized Coleman's functions playing the role of shtuka functions.

However, there is an assumption on the convergence of the infinite product of the twists of shtuka functions in Sinha's approach for arbitrary CM fields, and hence difficulty and complexity arise once the study of uniformization is approached through shtuka functions. As the CM fields we consider are arbitrary, constructing a rigid analytic trivialization explicitly as in \cite{ABP} for all the dual $t$-motives in question seems impractical.

To circumvent the aforementioned difficulty, we first consider the case of CM fields with CM types and  adopt the correspondence between certain lattices and Hilbert-Blumenthal $t$-modules established by Anderson~\cite{Anderson86}. This enables us to show that the constructed CM dual $t$-motive is uniformizable, and the case of generalized CM types can be achieved by the fundamental property of CM dual $t$-motives  in Proposition~\ref{prop: tensor CM type} and the isogeny property in Theorem~\ref{thm: CM prop}~(1). We refer the reader to Section~\ref{AppendixSubsec: Unif. of CM dual t-motives} for the details.

\subsection{Analogue of Shimura’s period symbol} \label{sec:introAnalogue}

To motivate the definition of our period symbols, we first recall that in Shimura's theory his period symbols are defined via period integrals on the CM abelian variety in question. As mentioned in the previous subsection, the Betti and de~Rham realizations of CM uniformizable abelian $t$-modules factor through their uniformizable dual $t$-motives.
So we can use the comparison isomorphism of the Betti and de~Rham modules of the CM uniformizable dual $t$-motive in question together with Theorem~\ref{thm: CM prop} to define our period symbols.

Let $\eM$ be a uniformizable dual $t$-motive over $\bar{k}$.
Let $H_{\mathrm{Betti}}(\eM)$ be the Betti module of $\eM$ (given in Remark~\ref{rem: HB(M)}), and
$H_{\mathrm{dR}}(\eM,\bar{k})$ be the de Rham module of $\eM$ over $\bar{k}$ (following \cite[Section 4.5]{HJ20}, see Section~\ref{sec: dR module}).
We introduce the \emph{de Rham pairing} in \eqref{E:deRham pairing}:
\begin{center}
\begin{tabular}{rclcc}
$H_{\mathrm{dR}}(\eM,\bar{k})$ &$\times$ & $H_{\mathrm{Betti}}(\eM)$ & $\longrightarrow$ & $\CC_\infty$, \\
$(\omega$ & $,$ & $\gamma)$ & $\longmapsto$ & $\displaystyle \int_\gamma \omega$,
\end{tabular}
\end{center}
which is non-degenerate (see Lemma~\ref{lem: dR}).
The result of a nonzero ``cycle integration'' $\int_\gamma \omega$ is called a \emph{period of $\eM$}, which matches with the terminology in \cite{Papanikolas} by Proposition~\ref{P:fields equality}.

Suppose $\eM$ is a CM dual $t$-motive with generalized CM type $(\eK,\Xi)$ over $\bar{k}$.
For each $\xi \in J_\eK$, one may associate a nonzero differential $\omega_{\eM,\xi}$, unique up to $\bar{k}^\times$-multiples, so that (see Proposition~\ref{Prop: delta_M})
\[
\alpha^*\omega_{\eM,\xi}  = \nu_\xi(\alpha) \cdot \omega_{\eM,\xi}
\]
for every $\alpha$ in the integral closure $O_\eK$ of $\FF_q[t]$ in $\eK$.
Here we view $\alpha$ as an endomorphism of $\eM$, and $\alpha^* \omega$ is the \emph{pull-back of $\omega$ by $\alpha$} for $\omega \in H_{\mathrm{dR}}(\eM,\bar{k})$ (see Section~\ref{sec: Pf-Pb}).
From Theorem~\ref{thm: CM prop} (2), we derive that up to $\bar{k}^\times$-multiples, the quantity
$\int_\gamma \omega_{\eM,\xi}$
depends only on the chosen $\xi$ and the CM type $\Xi$ of $\eK$.
Our \emph{period symbol} is the unique element $\Pcal_\eK(\xi,\Xi)$ in $\CC_\infty^\times/\bar{k}^\times$ represented by $\int_\gamma \omega_{\eM,\xi}$ for any nonzero $\gamma \in H_{\mathrm{Betti}}(\eM)$ (see Definition~\ref{defn: period quantity}), which is analogous to the ``motivic version'' in \cite[p.~28]{De80} of Shimura's classical one.

Throughout this paper, we fix a fundamental period $\tilde{\pi}$ of the Carlitz $\FF_{q}[t]$-module (see \cite{Goss, Thakur}). Let $I_\eK$ be the free abelian group generated by $J_\eK$, and $I_\eK^0$ be the subgroup generated by (generalized) CM types of $\eK$. We prove the following (see Theorem~\ref{thm: Period Symbol}).

\begin{theorem}\label{thm: PS}
For each CM field $\eK/\FF_{q}(t)$, there exists a unique bi-additive pairing
\[
\Pcal_\eK: I_\eK \times I_\eK^0 \rightarrow \CC_\infty^\times/\bar{k}^\times
\]
satisfying:
\begin{itemize}
\item[(1)] For each point $\xi \in J_\eK$ and each generalized CM type $\Xi$ of $\eK$, $\Pcal_\eK(\xi,\Xi)$ is the period symbol given above (see Definition~\ref{defn: period quantity}).
\item[(2)] Let $\eK'/\FF_{q}(t)$ be a CM field containing $\eK$. For $\Phi' \in I_{\eK'}$ and $\Phi^0 \in I_\eK^0$,
\[
    \Pcal_{\eK}\bigl({\Res}_{\eK'/\eK}(\Phi'),\Phi^0\bigr) = \Pcal_{\eK'}\bigl(\Phi',{\Inf}_{\eK'/\eK}(\Phi^0)\bigr),
\]
where ${\Res}_{\eK'/\eK}$ and ${\Inf}_{\eK'/\eK}$ are the restriction and inflation, respectively (see Section~\ref{sub:Res and Inf}).
\item[(3)]
Let $\eK'/\FF_{q}(t)$ be a CM field containing $\eK$ and let $\eK'^{+}$ be the maximal totally real subfield of $\eK'$.
For $\Phi \in I_\eK$ and $\Phi^{\prime,0} \in I_{\eK'}^0$,  we have
\[
    \Pcal_{\eK'}\bigl({\Inf}_{\eK'/\eK}(\Phi),\Phi^{\prime,0} \bigr) = \Pcal_{\eK}\bigl(\Phi,{\Res}_{\eK'/\eK}(\Phi^{\prime,0})\bigr).
\]
\item[(4)] Let $\eK'/\FF_{q}(t)$ be a CM field together with an isomorphism $\varrho: \eK'\rightarrow \eK$ over $\FF_q(t)$, and let $\pi_\varrho : X \rightarrow X'$ be the isomorphism corresponding to $\varrho$.
Then
\[
    \Pcal_{\eK'}\bigl((\pi_\varrho)_*(\Phi),(\pi_\varrho)_*(\Phi^0)\bigr) = \Pcal_\eK(\Phi,\Phi^0), \quad \forall (\Phi,\Phi^0) \in I_\eK\times I_\eK^0.
\]
\item[(5)] Let $\eK^+$ be the maximal totally real subfield of $\eK$.
Given $\xi^+ \in J_{\eK^+}$ and $\Phi^0 \in I_\eK^0$, one has that
\[
   \tilde{\pi}^{\text{\rm wt}(\Phi^0)} \cdot \bar{k}^\times = \prod_{\substack{\xi \in J_\eK \\ \pi_{\bX/\bX^+}(\xi) = \xi^+}}\Pcal_\eK(\xi,\Phi^0)\quad \in \CC_{\infty}^{\times}/ {\ok}^{\times}.
\]
\end{itemize}
\end{theorem}

\begin{remark}\label{rem: ps}
Following Shimura, we write $x\sim y$ for any $x$, $y \in \CC_\infty^\times$ with $x/y \in \bar{k}^\times$.
    Given $\Phi \in I_\eK$ and $\Phi^0 \in I_\eK^0$, we use the notation $p_\eK(\Phi,\Phi^0)$ for an arbitrary representative of the coset $\Pcal_\eK(\Phi,\Phi^0) \in \CC_\infty^\times/\bar{k}^\times$.
    We also call $p_\eK(\Phi,\Phi^0)$ a period symbol if no confusion will arise.
    The property in Theorem~\ref{thm: PS} (5) is then equivalent to:
\[
    \tilde{\pi}^{\text{\rm wt}(\Phi^0)} \sim \prod_{\substack{\xi \in J_\eK \\ \pi_{\bX/\bX^+}(\xi) = \xi^+}}p_\eK(\xi,\Phi^0),
    \quad \forall \xi^+ \in J_{\eK^+}.
\]
\end{remark}

In Section~\ref{sec: EPS}, we extend $\Pcal_{\eK}$ bi-additively  to a pairing (still denoted by)
\[
\Pcal_\eK:I_\eK \times I_\eK \rightarrow \CC_{\infty}^{\times}/\ok^{\times},
\]
and keep using the notation $p_\eK(\Phi_1,\Phi_2)$ for an arbitrary representative of $\Pcal_\eK(\Phi_1,\Phi_2)$ for every $\Phi_1,\Phi_2 \in I_\eK$.
We now present our main theorem.

\begin{theorem}[{See Theorem~\ref{T:Shimura's conjecture}}] \label{Main Theorem}
Let $\eK$ be a CM field over $\FF_q(t)$  and $\eK^+$ its maximal totally real subfield.
For fixed $\xi_0 \in J_\eK$, all algebraic relations among the period symbols $p_\eK(\xi,\xi_0)$ varying over $\xi \in J_\eK$ are generated by the ``Legendre relation'' (see Lemma~\ref{L:Algebraic relation}):
\[
\tilde{\pi}^{1/[\eK^+:\FF_q(t)]}  \sim  \prod_{\xi \in  \pi_{\bX/\bX^+}^{-1}(\xi^+)}p_\eK(\xi,\xi_0), \quad \text{$\forall\, \xi_0 \in J_\eK$ and $\xi^+ \in J_{\eK^+}$.}
\]
Equivalently, we have
\[
\trdeg_{\ok}\, \ok(p_\eK(\xi,\xi_0)\mid \xi \in J_\eK) = 1 + \frac{([\eK:\eK^+]-1)}{[\eK:\eK^+]} \cdot [\eK:\FF_q(t)].
\]
\end{theorem}

\begin{remark}
As one purpose of this paper is to study the transcendence theory of ``CM periods,'' our extension of $\Pcal_\eK$ in Section~\ref{sec: EPS} is slightly different from the one in Shimura's setting.
In particular,
the Legendre relation given in Theorem~\ref{Main Theorem} shows that the Carlitz period $\tilde{\pi}$ is always contained in $\ok(p_\eK(\xi,\xi_0)\mid \xi_1 \in J_\eK)$.
\end{remark}

\begin{remark}
In the classical case, nonzero values of arithmetic (quasi) modular forms and their Serre derivatives at  CM points can be interpreted using Shimura's period symbols (see~\cite{Sh80}, and~\cite{C12-2} for the case of arithmetic quasi-modular forms). In the function field setting, we have the same result for any arithmetic Drinfeld modular form~$f$ of rank two at a (separable) CM point $\alpha$. More precisely, if $f(\alpha)\neq 0$, it is shown in~\cite[Thm.2.2.1]{C12} that $f(\alpha)\sim \left(\omega_{\alpha}/\tilde{\pi} \right)^{\ell}$, where $\omega_{\alpha}$ is a nonzero period of a CM Drinfeld $\FF_{q}[t]$-module $E_{\rho}$ of rank two defined over $\ok$ with CM field $\eK$, which is imaginary quadratic over $\FF_{q}(t)$, and $\ell$ is the weight of $f$. Note that $E_{\rho}$ is a Hilbert-Blumenthal $t$-module of dimension one. Therefore, from the proof of Theorem~\ref{thm: CM HB t-mod} one sees that $\omega_{\alpha}\sim p_{\eK}(\xi,\Xi)$ for the CM type $(\eK,\Xi)$, where $\xi$ is any element of $J_{\eK}$ and $\Xi=\xi$. It follows that
\[
f(\alpha) \sim \frac{p_{\eK}(\xi,\Xi)^{\ell}}{{\tilde{\pi}}^{\ell}}.
\]
In particular, when $f$ has full-level, its \emph{Serre derivative} $\partial_\ell (f)$ (introduced by \cite[(8.5)]{Ge88})
becomes an arithmetic Drinfeld modular form of weight $\ell+2$.
Hence
\[
\partial_\ell(f)(\alpha) \sim \frac{p_{\eK}(\xi,\Xi)^{\ell+2}}{{\tilde{\pi}}^{\ell+2}}\quad \textup{if $\partial_\ell(f)(\alpha) \neq 0$.}
\]

Furthermore, for any \emph{arithmetic quasi-modular form} $\tilde{f}$ of weight $\ell$ (see \cite[p.~542--543]{C12-2}), by~\cite[Theorem~3.3~(c)]{C12-2} we may express $\tilde{f}(\alpha)$ in the following form:
\[
\tilde{f}(\alpha) = \sum_{0\leq a \leq \ell/2}\beta_a \cdot \frac{p_\eK(\xi,\Xi)^{\ell-2a}}{\tilde{\pi}^{\ell-a}}, \quad \text{ where }\beta_a \in \ok \text{ for $0\leq a \leq \ell/2$}.
\]
Theorem~\ref{Main Theorem} shows that if $f(\alpha)$ (resp.\ $\partial_\ell(f)(\alpha)$, $\tilde{f}(\alpha)$) is nonzero, then it must be transcendental over $\ok$, which agrees with the results in \cite[Theorem~1.2.1]{C12} and \cite[Theorem~1.2]{C12-2}.
Moreover, we would naturally expect to determine the algebraic independence of the CM values in question as well from Theorem~\ref{Main Theorem}.
\end{remark}

As an application, Theorem~\ref{Main Theorem} leads to the following result that matches the spirit of Conjectures~\ref{C:Shimura} and~\ref{C:Yoshida}.

\begin{corollary}\label{col: SC}
Let $\eK$ be a CM field over $\FF_q(t)$ and $\eK^+$ its maximal totally real subfield.
\begin{itemize}
    \item[(1)] Suppose that $\eK \neq \eK^+$, and let
    $\Xi=\xi_1+\cdots + \xi_r$ be a  CM type of $\eK$.
    Given $\xi_0 \in J_\eK$, the period symbols
    $p_\eK(\xi_1,\xi_0),\ldots ,p_\eK(\xi_r,\xi_0)$
    are algebraically independent over $\ok$.
    \item[(2)] Suppose that $\eK$ is Galois over $\FF_{q}(t)$.
    Then for each $\xi_0 \in J_\eK$ we have
\[
\trdeg_{\ok}\, \ok(\tilde{\pi}, p_\eK(\xi_0,\xi)\mid \xi \in J_\eK) = 1+ \frac{([\eK:\eK^+]-1)}{[\eK:\eK^+]}\cdot [\eK:\FF_{q}(t)].
\]
\end{itemize}
\end{corollary}

\subsection{Strategy of proof of the main theorem}

Before describing our strategy and ideas about proofs of the main results above, we first say a few words about the classical case and illustrate the essential differences between it and our situation. As mentioned in Sec.~\ref{Sec:Motivation}, the period conjecture of Deligne/Grothendieck implies Conjecture~\ref{C:Yoshida}, and in this case it suffices to compute the dimension of the Mumford-Tate group in question. As pointed out by Yoshida, from the description of the Mumford-Tate group~\cite[p.~270,~(4)]{Yo03} in that case there is no difficulty to see that the dimension in question is given by the rank of the subgroup generated by the reflex of the given CM type.

In the function field setting, it is shown in~\cite{Papanikolas} that there is an affine algebraic group $\Gamma_{\eM}$ associated to a uniformizable dual $t$-motive from Tannakian duality, and one views $\Gamma_{\eM}$ as playing a role analogous to that of the Mumford-Tate group in the classical theory. It is further established in~\cite{Papanikolas} that the dimension of $\Gamma_{\eM}$ is equal to the transcendence degree of the field generated by the periods of $\eM$ over $\ok$. However, it can be difficult to compute the dimension of $\Gamma_{\eM}$ in general, as the defining equations of $\Gamma_{\eM}$ arising from the rigid analytic trivialization of $\eM$ can be complicated to determine. Unlike the classical case, in our case of CM dual $t$-motives it is not obvious and direct to compute $\dim \Gamma_{\eM}$ in terms of the given generalized CM type. In fact, we appeal to Hodge-Pink's theory to generate a subtorus of the $\Gamma_{\eM}$ in question. A detailed outline is given as follows.

The overall strategy of proving Theorem~\ref{Main Theorem} is to give upper and lower bounds for the transcendence degree in question. We first observe that
the Legendre relations among period symbols provides the inequality
\[
\trdeg_{\ok}\, \ok(p_\eK(\xi,\xi_0)\mid \xi \in J_\eK) \leq 1 + \frac{([\eK:\eK^+]-1)}{[\eK:\eK^+]} \cdot [\eK:\FF_q(t)].
\]
We derive the opposite inequality through the following steps:
\begin{itemize}
    \item[(1)] Take a generalized CM type $\Xi_0$ of $\eK$ so that $\bar{k}\bigl(p_\eK(\xi,\xi_0)\mid \xi \in J_\eK\bigr)$ is algebraic over $\bar{k}\bigl(p_\eK(\xi,\Xi_0)\mid \xi \in J_\eK\bigr)$ (see Lemma~\ref{lem: non-degenerate}).
    \item[(2)] Let $\eM$ be a CM dual $t$-motive with  generalized CM type $(\eK,\Xi_0)$ over $\bar{k}$.
    Apply the function field analogue of Grothendieck's period conjecture (from~\cite{Papanikolas}, see Theorem~\ref{T:dim=trdeg}) to obtain
\[
    \trdeg_{\ok}\, \ok\bigl(p_\eK(\xi,\Xi_0)\mid \xi \in J_\eK\bigr) = \dim \Gamma_\eM,
\]
    where $\Gamma_\eM$ is the \emph{$t$-motivic Galois group of $\eM$} (see Proposition~\ref{P:fields equality} and Theorem~\ref{T:dim=trdeg}).
    \item[(3)] Use the \emph{Hodge-Pink filtration} of $\eM$ (see Definition~\ref{defn: HPF}) to construct a cocharacter $\chi_\eM$ of $\Gamma_\eM$ over $k^{\rm sep}$ (see \eqref{E:chiMvee}).
    As $\Gamma_\eM$ is defined over $\FF_q(t)$,
    the image of all Galois conjugates of $\chi_\eM$ produces a subtorus $T_\eM$ inside $\Gamma_{\eM,\CC_\infty}$, the base change of $\Gamma_\eM$ to $\CC_\infty$.
    \item[(4)] From the connection between the ``Hodge-Pink weights'' of $\eM$ and the given CM type $\Xi_0$ in \eqref{eqn: HPc-CMt}, we show that $\dim T_\eM$ is equal to the rank of the ``Galois submodule'' $I_{\Xi_0}^0$ generated by $\Xi_0$ in $I_\eK^0$ (defined in Remark~\ref{rem: TM} (2)).
    \item[(5)] Finally, from the \emph{non-degeneracy} of the chosen $\Xi_0$ proved in Lemma~\ref{lem: non-degenerate}, i.e.,
    \[
    \rank_\ZZ I_{\Xi_0}^0 = \rank_\ZZ I_\eK^0,
    \]
    we obtain that
\begin{align*}
    \trdeg_{\bar{k}}\bar{k}\bigl(p_\eK(\xi,\xi_0)\mid \xi \in J_\eK\bigr) &= \trdeg_{\bar{k}}\bar{k}\bigl(p_\eK(\xi,\Xi_0)\mid \xi \in J_\eK\bigr) \\
    &= \begin{aligned}[t]
    \dim \Gamma_\eM
    &\geq  \dim T_\eM = \rank_\ZZ I_{\Xi_0}^0 = \rank_\ZZ I_\eK^0 \\
    &= 1 + \frac{([\eK:\eK^+]-1)}{[\eK:\eK^+]} \cdot [\eK:\FF_q(t)],
    \end{aligned}
\end{align*}
    as desired.
\end{itemize}

\begin{remark}
(1) For a CM dual $t$-motive $\eM$ with generalized CM type $(\eK,\Xi)$ over $\bar{k}$, the field generated by the periods of $\eM$ is (see \eqref{E:field equality})
\[
\bar{k}\Bigl(\int_\gamma \omega \Bigm| \gamma \in H_{\mathrm{Betti}}(\eM),\ \omega \in H_{\mathrm{dR}}(\eM,\bar{k})\Bigr) = \bar{k}(p_\eK(\xi,\Xi)\mid \xi \in J_\eK\bigr).
\]
When $\Xi$ is non-degenerate, we then have that
\begin{multline*}
\trdeg_{\bar{k}}\, \bar{k}\Bigl(\int_\gamma \omega \Bigm| \gamma \in H_{\mathrm{Betti}}(\eM),\ \omega \in H_{\mathrm{dR}}(\eM,\bar{k})\Bigr)
=
1 + \frac{([\eK:\eK^+]-1)}{[\eK:\eK^+]} \cdot [\eK:\FF_q(t)].
\end{multline*}
(2) As $\eM$ is a CM dual $t$-motive, the $t$-motivic Galois group $\Gamma_\eM$ is commutative (see the discussion after Proposition~\ref{prop: HPW-CM}).
If $\Gamma_{\eM,\CC_\infty}$ is generated by the images of all ``Hodge-Pink cocharacters'' of $\eM$ (cf.\ \cite[Theorem 7.11]{Pink}), i.e.\ $\Gamma_{\eM,\CC_\infty} = T_\eM$,
then we actually obtain
\[
\trdeg_{\ok}\, \bar{k}\bigl(p_\eK(\xi,\Xi)\mid \xi \in J_\eK\bigr) = \dim \Gamma_\eM = \dim T_\eM = \rank_{\ZZ} I_\Xi^0,
\]
which says that the transcendence degree of the field generated by all period symbols $p_\eK(\xi,\Xi)$ for $\xi \in J_\eK$ is determined by finding the rank of the corresponding Galois submodule $I_\Xi^0$ in $I_\eK^0$.
\end{remark}

\subsection{Applications}

Let $E_\rho$ be an abelian $t$-module over $\bar{k}$ (recalled in Definition~\ref{defn: t-module}).
Put $\Mscr(\rho)\assign \Hom_{\FF_q}(E_\rho,\GG_a)$ (introduced by Anderson \cite{Anderson86}, see Definition~\ref{defn: t-module} (2)).
Let $\eM(\rho)\assign \Hom_{\bar{k}[t]}(\tau\Mscr(\rho),\bar{k}[t]dt)$ (constructed by Hartl and Juschka \cite{HJ20}, see Section~\ref{sec: M(rho)}),
where $\tau$ is identified with the  $q$-power Frobenius endomorphism on $\GG_a$.
We know that $\eM(\rho)$ is a dual $t$-motive (and $E_\rho$ is always ``$A$-finite,'' see Theorem~\ref{thm: M(rho)}), called the \emph{Hartl-Juschka dual $t$-motive associated with $E_\rho$}.
It is worth clarifying that the traditional dual $t$-motive associated with a given $t$-module $E_\rho$ is usually regarded as $\eN(\rho)\assign \Hom_{\FF_q}(\GG_a,E_\rho)$ (also introduced by Anderson, see Definition~\ref{defn: t-module} (2)), and there is an isomorphism between $\eM(\rho)$ and $\eN(\rho)$ by \cite[Theorem 2.5.13]{HJ20} (see Theorem~\ref{thm: M(rho)}).
For our purposes, we shall utilize $\eM(\rho)$ instead of $\eN(\rho)$ throughout this paper.
Suppose $E_\rho$ is uniformizable (for which so is $\eM(\rho)$, see Theorem~\ref{thm: Uni-eqv}).
Let $\Lambda_\rho$ be the period lattice of $E_\rho$ and let $H_{\mathrm{dR}}(E_\rho,\bar{k})$ be the de Rham module of $E_\rho$ (see \eqref{E:Isom HdR and QtauM}, following \cite[\S 3.1]{BP02} and \cite[Proposition 4.1.3]{NP21}).
Let $[\cdot,\cdot]:H_{\mathrm{dR}}(E_\rho,\bar{k}) \times \Lambda_\rho \rightarrow \CC_\infty$ be the de Rham pairing introduced in \cite[(4.3.3) and (4.3.4)]{NP21} (see \eqref{eqn: dR-P t-module}).
In order to apply our results to abelian $t$-modules, we first naturally identify $H_{\mathrm{dR}}(E_\rho,\bar{k})$ (resp.\ $\Lambda_\rho$) with $H_{\mathrm{dR}}(\eM(\rho),\bar{k})$ (resp.\ $H_{\mathrm{Betti}}(\eM(\rho))$), and derive the following comparison between the associated de Rham pairings (cf.\ Proposition~\ref{prop: comp-dR-pairing}):
$$
\SelectTips{cm}{}
\xymatrixrowsep{0.4cm}
\xymatrixcolsep{0.05cm}
\xymatrix{
[\ ,\ ]: & H_{\mathrm{dR}}(E_\rho,\bar{k})\ar@{<->}[dd]^{\rotatebox{90}{\scalebox{1}[1]{$\sim$}}} & \times & \Lambda_\rho\ar@{<->}[dd]^{\rotatebox{90}{\scalebox{1}[1]{$\sim$}}}  \ar[rrr] & & & \CC_\infty \ar@{<->}[dd]^{\rotatebox{90}{\scalebox{1}[1]{$\sim$}}} & z \ar@{<->}[dd] \\
&&&& \\
\int : & H_{\mathrm{dR}}(\eM(\rho),\bar{k}) & \times & H_{\mathrm{Betti}}(\eM(\rho)) \ar[rrr] & & & \CC_\infty & -z\\
}
$$

We say that $E_\rho$ is a \emph{CM abelian $t$-module with generalized CM type $(\eK,\Xi)$} if the corresponding dual $t$-motive $\eM(\rho)$ has generalized CM type $(\eK,\Xi)$.
In this case, the above comparison of de Rham pairings leads us to the following (cf.\ Theorem~\ref{thm: CM t-mod}):

\begin{theorem}\label{thm: app CM t-mod}
Let $E_\rho$ be a CM abelian $t$-module with generalized CM type $(\eK,\Xi)$ over $\bar{k}$.
Let $\eK^+$ be the maximal totally real subfield of $\eK$.
Then
\begin{align*}
1+\frac{[\eK:\eK^+]-1}{[\eK:\eK^+]}\cdot [\eK:\FF_q(t)] &\geq
\trdeg_{\bar{k}} \bar{k}\bigl([\delta,\lambda] \bigm|  \delta \in H_{\mathrm{dR}}(E_\rho,\bar{k}),\ \lambda \in \Lambda_\rho \bigr) \\
&\geq \rank_\ZZ I_\Xi^0.
\end{align*}
In particular, if the generalized CM type $\Xi$ of $\eK$ is non-degenerate, then
\[
\trdeg_{\bar{k}} \bar{k}\bigl([\delta,\lambda] \bigm|  \delta \in H_{\mathrm{dR}}(E_\rho,\bar{k}),\ \lambda \in \Lambda_\rho \bigr)
= 1+\frac{[\eK:\eK^+]-1}{[\eK:\eK^+]}\cdot [\eK:\FF_q(t)].
\]
\end{theorem}

Furthermore, suppose $\Xi$ is a CM type (i.e.\ ${\rm wt}(\Xi) = 1$) for $\eK$.
This is equivalent to saying that $E_\rho$ is a ``Hilbert-Blumenthal $O_{\eK^+}$-module''  with $\rank_A(\Lambda_\rho) = [\eK:\FF_q(t)]$, and the structure homomorphism $\rho$ can be extended to a homomorphism (still denoted by $\rho$) $O_\eK \rightarrow \End_{\text{\rm ab.\ $t$-mod}}(E_\rho)$ (cf. Section~\ref{sec: CM HB t-mod}).
Recall in \cite{Yu89} that the coordinates of any nonzero period vector of a Hilbert-Blumenthal $t$-module $E_\rho$ over $\bar{k}$ must be transcendental over $k$.
Here we extend Yu's result to the case when $E_\rho$ is CM with a non-degenerate CM type (see Theorem~\ref{thm: CM HB t-mod}):

\begin{theorem}\label{thm: app CM HB t-mod}
Let $E_\rho$ be a CM abelian $t$-module of dimension $d$ with non-degenerate CM type $(\eK,\Xi)$ over $\ok$.
Let $\eK^+$ be the maximal totally real subfield of $\eK$.
Fix an identification $\Lie(E_\rho) \cong \bar{k}^d$ where $d = [\eK^+:\FF_q(t)]$ and extend to $\Lie(E_\rho)(\CC_\infty)\cong \CC_\infty^d$.
Then every nonzero period vector $\lambda \in \Lambda_\rho \subseteq \Lie(E_\rho)(\CC_\infty)$ has coordinate vector $(\lambda_1,\ldots,\lambda_d) \in \CC^d$ with $\lambda_1,\ldots,\lambda_d$ which are algebraically independent over $\ok$.
\end{theorem}

\begin{Subsubsec}{Lang-Rohrlich conjecture over function fields}
As special values of various transcendental functions appear as CM periods, and hence period symbols, in certain circumstances we may apply our theorems to study transcendence problems and algebraic relations among special values.
For instance, let $A_+$ be the set of monic polynomials in $A$. Given $f \in A_+$ with $\deg f >0$,
suppose the CM field $\eK = \eK_f$, the \emph{$f$-th Carlitz cyclotomic function field over $\FF_q(t)$}.
Using the ``generalized soliton dual $t$-motives'' constructed in \cite{ABP}, we are able to show that $\bar{k}\bigl(p_{\eK_f}(\xi,\xi_0)\mid \xi \in J_{\eK_f}\bigr)$
is a finite extension over
\[
\bar{k}\bigl(\tilde{\pi},\Gamma(x)\mid x \in \frac{1}{f} A \setminus (\{0\} \cup -A_+)\bigr),
\]
where $\Gamma(x)$ is the \emph{geometric gamma function} defined by Thakur~\cite{Thakur91}:
\[
\Gamma(x)\assign x^{-1} \prod_{a \in A_+}(1+\frac{x}{a})^{-1}.
\]
Therefore Theorem~\ref{Main Theorem} implies
\begin{multline*}
\trdeg_{\bar{k}} \bar{k}\bigl(\tilde{\pi},\Gamma(x)\mid x \in \frac{1}{f} A\setminus(\{0\} \cup -A_+)\bigr) \\
= 1 + \frac{[\eK_f:\eK_f^+]-1}{[\eK_f:\eK_f^+]}\cdot [\eK_f:\FF_q(t)] \ = \ 1 + \frac{q-2}{q-1} \cdot \#(A/f)^\times,
\end{multline*}
which coincides with \cite[Corollary 1.2.]{ABP}. In other words, our main theorem provides an alternative approach to the Lang-Rohrlich conjecture for the geometric Gamma values.

\end{Subsubsec}

\begin{Subsubsec}{Remark}
We note that classical CM abelian varieties  correspond in our setting to  those Hilbert-Blumenthal $t$-modules arising from CM dual $t$-motives with CM type ($\eK,\Xi$), namely a generalized CM type with $\wt(\Xi)=1$. It seems that there are no appropriate objects in the classical theory that correspond to the CM $t$-modules associated to CM dual $t$-motives of generalized CM type $(\eK,\Xi)$ with $\wt(\Xi)>1$, and so we could expect that applying our results above  to  CM $t$-modules of generalized CM types gives rise to new kinds of results.

For instance, when the given CM field $\eK$ is contained in a constant field extension of a Carlitz cyclotomic function field, it is shown in current work~\cite{Wei22} that any quasi-period of the $t$-module associated to the CM dual $t$-motive with generalized CM type $(\eK,\Xi)$ is related to  \emph{special two-variable gamma values} (introduced by Goss in \cite{Goss88}, see also \cite{Thakur91}), which previously have not been investigated from a transcendence point of view in this generality, and whose algebraic relations are determined based on the results of this paper.  Moreover, in~\cite{Wei22} the recipe/conjecture on the \emph{Chowla-Selberg phenomenon} proposed by Thakur \cite{Thakur91} is completely proved, and the intrinsic relations between period symbols and two-variable gamma values are found by developing an analogue of Anderson's \emph{period distribution}.
\end{Subsubsec}

\subsection{Contents}
This paper is organized as follows.
We first set up basic notation in the beginning of Section~\ref{sec: pre}, and review the equivalence between abelian $t$-modules and dual $t$-motives in Section~\ref{sec: Abelian t-mod}--\ref{sec: M(rho)}.
After introducing the Betti module of a dual $t$-motive in Section~\ref{sec: U d t-mod}, we  review the theory of the function field analogue of Grothendieck's period conjecture established by the third author in \cite{Papanikolas}.

In Section~\ref{sec: CM setting}, we introduce the CM function fields and the CM types (when the CM field is geometric over $\FF_q(t)$), including the inflation and restriction maps in Section~\ref{sub:Res and Inf} and the ``reduction of CM types at infinity'' in Section~\ref{sec: R-pt-inf}.

In Section~\ref{sec: CM dual t-motives}, we utilize the ``shtuka'' functions associated to CM types (given in Section~\ref{sec: shtuka}) to construct CM dual $t$-motives in a geometric way (when the CM field is geometric over $\FF_q(t)$) in Section~\ref{sec: construction of CM} (cf.~\cite{Sinha97, BP02, ABP}).
We also conclude at the end of Section~\ref{sec: construction of CM} that two CM dual $t$-motives with the same (generalized) CM type must be isogeneous.
The uniformizablility of CM dual $t$-motives is proved
in Section~\ref{subsec: U-CM}.

After introducing the de Rham pairing of a dual $t$-motive and showing the needed properties in Section~\ref{sec: dR module}, we introduce an analogue of Shimura's period symbol in Section~\ref{sec: defn PS}.
The fundamental properties of the period symbols in Theorem~\ref{thm: PS} are proved in Section~\ref{sec: FP PS}.

We state our main theorem and prove Corollary~\ref{col: SC} in Section~\ref{sec: A-SC}.
Some preparation for proving the main theorem is given in Section~\ref{sec: Proof MT}.

In Section~\ref{sec: H-P theory}, we recall the needed results from the Hodge-Pink theory of the category introduced by Hartl and Juschka \cite{HJ20}.
After indicating the isomorphism between the $t$-motivic Galois groups coming from the category of pre-$t$-motives and the Hartl-Juschka category in Section~\ref{sec: t-motivic}, we recall the Hodge-Pink filtration and the Hodge-Pink cocharacter in Section~\ref{sec: HP-fil} and Section~\ref{sec: HP coc}, respectively.
In Section~\ref{Sub:The case of CM dual t-motives}, we connect the weight of the Hodge-Pink cocharacter of a given CM dual $t$-motive $\eM$ with its associated CM type, which leads to a lower bound of the dimension of the $t$-motivic Galois group in question.
Finally, we complete the proof of Theorem~\ref{Main Theorem}  at the end of Section~\ref{Sub:The case of CM dual t-motives}.

Section~\ref{sec: App} contains applications of this theorem to the periods of CM abelian $t$-modules.
We first make natural comparisons on the Lie spaces, the Betti modules, the de Rham modules, and the de Rham pairings between the given abelian $t$-modules
and the associated dual $t$-motives in Section~\ref{sec: Com Lie}--\ref{sec: comparison DR}. As further applications, we derive Theorem~\ref{thm: app CM t-mod} and Theorem~\ref{thm: app CM HB t-mod} in Section~\ref{sec: app CM t-mod} and Section~\ref{sec: CM HB t-mod}, respectively.

Finally, we work out the generalization of the definition of CM types and CM dual $t$-motives for those  CM fields which are not geometric over $\FF_{q}(t)$ in Appendix~\ref{secA: CM} and \ref{secB: CM motive}. Using the characterization in Proposition~\ref{prop: being CM} of those dual $t$-motives which are CM, we complete the proof of the uniformizability of CM dual $t$-motives at the end of Section~\ref{AppendixSubsec: Unif. of CM dual t-motives}.

\subsection*{Acknowledgements}
We are grateful to J. Yu for providing many helpful suggestions and comments during the preparation of this paper. We particularly thank National Center for Theoretical Sciences for financial support and hospitality over the past years.

\section{Preliminaries}\label{sec: pre}
Let $A\assign \FF_{q}[\theta]$ be the polynomial ring in the variable $\theta$ over a finite field of $q$ elements denoted by $\FF_{q}$, and let $k\assign \FF_{q}(\theta)$ be the fraction field of $A$. We fix a normalized absolute value $|\cdot|_{\infty}$ on $k$ associated to the place at infinity of $k$, and we let $k_{\infty}\assign \laurent{\FF_{q}}{1/\theta}$ be the completion of $k$ with respect to $|\cdot|_{\infty}$. Let $\CC_{\infty}$  be the completion of a fixed algebraic closure of $k_{\infty}$, completed with respect to the absolute value  that canonically extends the one on $k_{\infty}$. Let $\ok$ be a fixed algebraic closure of $k$ embedded inside $\CC_{\infty}$.

\subsection{\texorpdfstring{Abelian $t$-modules}{Abelian t-modules}}\label{sec: Abelian t-mod}
Let $\KK$ be an algebraically closed field with $k \subset \KK\subset \CC_{\infty}$, and let $t$ be a new variable independent from $\CC_{\infty}$. We recall the notion of an abelian $t$-module introduced by Anderson~\cite{Anderson86}, which plays the analogous role of a commutative algebraic group in the classical setting.
\begin{definition}\label{defn: t-module}
Let ${\GG_{a}}_{/\KK}$ be the additive group over $\KK$.
\begin{enumerate}
\item A \emph{$t$-module of dimension} $d>0$ over $\KK$ is a pair $(E,\rho)$ with the following properties.
\begin{itemize}
\item $E$ is a $d$-dimensional additive group ${\GG_{a}^{d}}_{/\KK}$ defined over $\KK$,
\item $\rho:\FF_{q}[t]\rightarrow \End_{\FF_{q}} ({\GG_{a}^{d}}_{/\KK})$ is an $\FF_{q}$-linear ring homomorphism so that $ \partial\rho_{t}-\theta I_{d}$ acts nilpotently on $\Lie ({\GG_{a}^{d}}_{/\KK})$, where $\partial \rho_{t}$ denotes the  differential of the homomorphism $\rho_{t}:{\GG_{a}^{d}}_{/\KK}\rightarrow {\GG_{a}^{d}}_{/\KK}$.
\end{itemize}
We denote by $E_{\rho}$ the underlying space of ${\GG_{a}^{d}}_{/\KK}$ equipped with the $\FF_{q}[t]$-module structure via $\rho$.
\item $E_{\rho}$ is called an \emph{abelian} (resp.\ \emph{$A$-finite}) \emph{$t$-module} if there exists a finite-dimensional $\KK$-vector subspace $V$ of $\Mscr(\rho) \assign \Hom_{\FF_{q}}(E_{\rho},\GG_{a})$ (resp.\ $\eN(\rho)\assign  \Hom_{\FF_q}(\GG_a, E_{\rho})$) of $\FF_{q}$-linear morphisms of algebraic groups over $\KK$ so that
\[
\Mscr(\rho)=\sum_{j=0}^{\infty}V\circ \rho_{t^{j}} \quad (\text{resp.} \quad \eN(\rho)=\sum_{j=0}^{\infty}\rho_{t^{j}} \circ V).
\]
\end{enumerate}
\end{definition}

Let $\KK[t,\tau]$ (resp.\ $\KK[t,\sigma]$) be the non-commutative polynomial ring over $\KK$ subject to the relations,
\[
tc = ct,\ t\tau = \tau t,\  \tau c=c^{q} \tau \quad (\textup{resp.}\  t \sigma = \sigma t,\ \sigma c = c^{1/q}\sigma), \quad \forall\, c\in \KK.
\]
Identifying $\tau$ with the Frobenius $q$-th power endomorphism $(x\mapsto x^{q}) \in \End_{\FF_q}(\GG_{a/\KK})$, we note that $\Mscr(\rho)$ (resp.\ $\eN(\rho)$)
has a left $\KK[t,\tau]$-module (resp.\ $\KK[t,\sigma]$-module) structure given as follows: for $c\in \KK$, $a\in \FF_{q}[t]$, $m \in \Mscr(\rho)$ (resp.\ $\en \in \eN(\rho)$),
\begin{enumerate}
\item $c \cdot m \assign  cm$ \quad (resp.\ $c\cdot \en \assign  \en c$).
\item $\tau \cdot m\assign \tau \circ m$ \quad (resp.\ $\sigma \cdot \en \assign  \en \circ \tau$).
\item $a\cdot m\assign m\circ \rho_{a}$ \quad (resp.\ $a \cdot \en \assign  \rho_a \circ \en$).
\end{enumerate}
We note that $E_{\rho}$ being an abelian (resp.\ $A$-finite) $t$-module is equivalent to $\Mscr(\rho)$ (resp.\ $\eN(\rho)$) being finitely generated over $\KK[t]$.

Morphisms of abelian (resp.\ $A$-finite) $t$-modules are defined as morphisms of the underlying algebraic groups that are also left $\FF_{q}[t]$-module homomorphisms. A surjective morphism of abelian (resp.\ $A$-finite) $t$-modules is called an \emph{isogeny} if it has finite kernel.

The basic examples of abelian and $A$-finite $t$-modules are Drinfeld $\FF_{q}[t]$-modules over~$\KK$. They are defined as $E_{\rho}\assign ({\GG_{a}}_{/\KK}, \rho)$, where $\rho_{t}$ is a morphism of ${\GG_{a}}_{/\KK}$ of the form
\[
\rho_{t}:X\mapsto \theta X+ a_{1}X^{q}+\cdots+a_{r}X^{q^{r}},
\]
where $r$ is a positive integer and $a_{1},\ldots,a_{r}\in \KK$ with $a_{r}\neq 0$. Here $r$ is called the \emph{rank} of the Drinfeld module $E_{\rho}$.

Given an abelian $t$-module $E_{\rho}$ over $\KK$, Anderson~\cite{Anderson86} showed that there is an associated exponential function
\[
\exp_{\rho}:\Lie(E_{\rho})(\CC_{\infty})\rightarrow
E_{\rho}(\CC_{\infty}),
\]
given by the unique $\FF_{q}$-linear power series satisfying the relations,
\begin{itemize}
\item $\exp_{\rho}( \partial \rho_{t}(\bz)) = \rho_{t}( \exp_{\rho}(\bz))$,
\item $\exp_{\rho}(\bz)\equiv \bz\  \textup{(mod deg $q$)}$.
\end{itemize}

\begin{definition}
An abelian $t$-module $E_{\rho}$ is called \emph{uniformizable} if its exponential map $\exp_{\rho}$ is surjective. In this case, $\Lambda_{\rho}\assign \Ker \exp_{\rho}$ is called the \emph{period    lattice} of $E_{\rho}$, and $E_{\rho}(\CC_{\infty})$ is isomorphic to $\CC_{\infty}^{\dim E_{\rho}}/\Lambda_{\rho}$ as an $\FF_{q}[t]$-module.
\end{definition}

\subsection{\texorpdfstring{Anderson dual $t$-motives}{Anderson dual t-motives}}\label{sec: dual t-mot}
For an integer $n$, we define \emph{$n$-fold twisting} on the Laurent series field $\laurent{\CC_{\infty}}{t}$ as
\[
\biggl( f\assign \sum a_{i}t^{i}\ \mapsto\ f^{(n)}\assign \sum a_{i}^{q^{n}}t^{i}
\biggr):\laurent{\CC_{\infty}}{t}\rightarrow \laurent{\CC_{\infty}}{t}.
\]
We note that $n$-fold twisting is an $\laurent{\FF_{q^{n}}}{t}$-linear automorphism of the field $\laurent{\CC_{\infty}}{t}$. We may regard $\KK[t,\sigma] $ as the non-commutative polynomial ring obtained by adjoining the variable $\sigma$ to $\KK[t]$ and multiplying by the rule
\[
\sigma f=f^{(-1)}\sigma \textnormal{ for }f\in \KK[t].
\]
\begin{definition}
An (Anderson) \emph{dual $t$-motive} over $\KK$ is a left $\KK[t,\sigma]$-module $\eM$ so that
\begin{itemize}
\item $\eM$ is a free left $\KK[t]$-module of finite rank,
\item $\eM$ is a free left $\KK[\sigma]$-module of finite rank,
\item $(t-\theta)^{N} \eM/\sigma \eM= (0)$ for a sufficiently large integer $N$.
\end{itemize}
\emph{Morphisms} of  dual $t$-motives over $\KK$ are defined to be left $\KK[t,\sigma]$-module homomorphisms.
\end{definition}

\begin{example}
The basic example of a dual $t$-motive is the Carlitz motive $\eC$ over $\KK$. The underlying space of $\eC$ is $\KK[t]$, on which the $\sigma$-action is given by
\[
\sigma f\assign (t-\theta) f^{(-1)}, \quad \textup{for $f\in \eC$.}
\]
\end{example}

Let $E_\rho = (\GG_{a/\KK}^d,\rho)$ be an $A$-finite $t$-module over $\KK$. Then $\eN(\rho)$ must be free of finite rank over $\KK[t]$. In other words, $\eN(\rho)$ is a dual $t$-motive. Moreover, given a morphism $f : E_{\rho_1} \rightarrow E_{\rho_2}$ between $A$-finite $t$-modules, there is an induced $\KK[t,\sigma]$-module homomorphism
\[
\Theta_f : \eN(\rho_1)\rightarrow \eN(\rho_2), \quad \Theta_f(\en_1) \assign  f\circ \en_1.
\]

\begin{theorem}[{Anderson, see~\cite[Prop.~2.5.8 and Thm.~2.5.11]{HJ20}}] \label{thm: Anderson Eq}
The functor
\[
(E_{\rho} \mapsto \eN(\rho),\quad f \mapsto \Theta_f)
\]
from the category of $A$-finite $t$-modules to the category of  dual $t$-motives is an equivalence. Moreover, ${\eN}(\rho)/(\sigma -1) {\eN}(\rho)$ and ${\eN(\rho)}/\sigma {\eN(\rho)}$ are naturally identified with $E_\rho(\KK)$ and  $\Lie({E_\rho})(\KK)$ respectively.
\end{theorem}

\begin{remark}
(1) An inverse functor of this equivalence  from the category of  dual $t$-motives to the category of $A$-finite $t$-modules is given as follows.
Let $\eM$ be a dual $t$-motive of rank $d$ over $\KK[\sigma]$. We note that the action of  $\FF_{q}[t]$ on $\eM$ commutes with the operator $\sigma$ and, so the quotient
\[
\eM/(\sigma-1)\eM
\]
is a left $\FF_{q}[t]$-module.
As $\rank_{\KK[\sigma]} \eM = d$,  we can view  $\eM/(\sigma-1)\eM$ as the $\KK$-valued points of ${\GG_{a}^{d}}_{/\KK}$, and its $\FF_q[t]$-module structure induces an $\FF_q$-algebra homomorphism $\rho^\eM : \FF_q[t] \rightarrow \End_{\FF_q}(\GG_{a/\KK}^d)$. It is known that $E_\eM \assign  (\GG_{a/\KK}^d,\rho^\eM)$ is an $A$-finite $t$-module over $\KK$. Moreover, let $\Theta:\eM_{1}\rightarrow \eM_{2}$ be a morphism of  dual $t$-motives. the corresponding $\FF_{q}[t]$-module homomorphism
\[
\overline{\Theta}:\eM_{1}/(\sigma-1)\eM_{1}\rightarrow \eM_{2}/(\sigma-1)\eM_{2},
\]
induces a morphism of $A$-finite $t$-modules $\overline{\Theta}: E_{\eM_1} \rightarrow E_{\eM_2}$.

(2) We refer the readers to \cite{CP11}, \cite{CP12} for explicit examples of how to associate dual $t$-motives to Drinfeld modules. For general dual $t$-motives and $A$-finite $t$-modules, see also \cite[\S 1.5.4]{BPrapid}, \cite[\S 3.1]{NP21}.
\end{remark}

\begin{definition}
An injective morphism $\Theta:\eM_{1}\rightarrow \eM_{2}$ of  dual $t$-motives with cokernel finite-dimensional over $\KK$ is called an \emph{isogeny}. We say that the  dual $t$-motives $\eM_{1}$ and $\eM_{2}$ are \emph{isogenous} if there exists an isogeny $\Theta:\eM_{1}\rightarrow \eM_{2}$.
\end{definition}

Let $\Theta:\eM_{1}\rightarrow \eM_{2}$ be an isogeny of  dual $t$-motives. We note that the induced morphism $\overline{\Theta}$ of abelian $t$-modules is surjective with finite kernel (cf.~\cite[Lem.~4.1.4]{ABP}), whence an isogeny of abelian $t$-modules. We further note that the isogeny relation of dual $t$-motives is an equivalence relation (see~\cite[p.~266]{ABP}).

\subsection{\texorpdfstring{Hartl-Juschka dual $t$-motives associated with abelian $t$-modules}{Hartl-Juschka dual t-motives associated with abelian t-modules}} \label{sec: M(rho)}
Let $\Omega_{\KK[t]/\KK}$ $\assign  \KK[t]\,dt$ be the module of K\"{a}hler differentials of $\KK[t]$ over $\KK$. Given an abelian $t$-module $E_{\rho}\assign (\GG_{a/\KK}^d, \rho)$ over $\KK$, following Hartl and Juschka~\cite[Proposition 2.4.3]{HJ20}  we define
\begin{equation}\label{eqn: M(rho)}
\eM(\rho)\assign  \Hom_{\KK[t]}(\tau \Mscr(\rho), \Omega_{\KK[t]/\KK}),
\end{equation}
together with the following $\sigma$-action: for $\eum \in \eM(\rho)$,
\[
(\sigma \cdot \eum)(\tau \cdot m) \assign  \eum(\tau^2 \cdot m)^{(-1)}, \quad \forall \tau m \in \tau \Mscr(\rho).
\]
This gives a $\KK[t,\sigma]$-module structure on $\eM(\rho)$, and $\eM(\rho)$ is free of finite rank over $\KK[t]$.  See also~\cite[\S~4.4]{NP21}. Moreover, the following holds.

\begin{theorem}\label{thm: M(rho)}
Let $E_\rho = (\GG_{a/\KK}^d,\rho)$ be an abelian $t$-module over $\KK$. Then $\eM(\rho)$ is a dual $t$-motive over $\KK$ isomorphic to $\eN(\rho)$ as $\KK[t,\sigma]$-modules. Consequently, $E_\rho$ is $A$-finite, and $E_\rho \cong E_{\eM(\rho)}$.
\end{theorem}

\begin{proof}
From \cite[Theorem 2.5.13]{HJ20}, there exists an injective $\KK[t,\sigma]$-module homomorphism $\Theta: \eM(\rho)\hookrightarrow \eN(\rho)$.
As $\eN(\rho)$ is free of finite rank over $\KK[\sigma]$, so is $\eM(\rho)$.
Moreover, we have that
\begin{align*}
\rank_{\KK[\sigma]} \eM(\rho) = \dim_{\KK}\left(\frac{\eM(\rho)}{\sigma \eM(\rho)}\right) &= \dim_\KK\left(\frac{\Mscr(\rho)}{\tau \Mscr(\rho)}\right) \\
&=
\rank_{\KK[\tau]} \Mscr(\rho) = \dim E_\rho = \rank_{\KK[\sigma]} \eN(\rho),
\end{align*}
where the second equality follows from \cite[Remark 2.4.4 (c)]{HJ20}.
This implies that the cokernel of $\Theta$ must be of finite dimension over $\KK$.
Since $\eM(\rho)$ is free of finite rank over $\KK[t]$, we obtain that $\eN(\rho)$ is a finitely generated $\KK[t]$-module.
Therefore $E_\rho$ is $A$-finite, and by \cite[Theorem 2.5.13]{HJ20} again $\Theta: \eM(\rho)\rightarrow \eN(\rho)$ is an isomorphism.
Finally, the induced isomorphism $\overline{\Theta} : E_{\eM(\rho)} \cong E_{\eN(\rho)}$
and the identification $E_{\eN(\rho)} = E_\rho$ completes the proof.
\end{proof}

\begin{remark}\label{rem: M(rho)}
For each abelian $t$-module $E_\rho$, we call $\eM(\rho)$ the \emph{Hartl-Juschka dual $t$-motive associated with $E_\rho$}.
For our purpose, we utilize $\eM(\rho)$ here (instead of $\eN(\rho)$) in order to derive the natural compatibility between the de Rham pairings of $E_\rho$ and $\eM(\rho)$, see Section~\ref{sec: comparison DR}.
\end{remark}

\subsection{\texorpdfstring{Uniformizability of  dual $t$-motives}{Uniformizability of  dual t-motives}}\label{sec: U d t-mod}
Given a dual $t$-motive $\eM$ of rank $r$ over $\KK[t]$, we fix $\bm\in \Mat_{r\times 1}(\eM)$ so that the entries of $\bm$ comprise a $\KK[t]$-basis of $\eM$. Then the action of $\sigma$ on $\bm$ is presented as
\[
\sigma \bm=\Upphi \bm
\]
for some $\Upphi\in \Mat_{r}(\KK[t])$. Since $(t-\theta)^{N} \eM/\sigma \eM =(0)$, $\det \Phi$ is of the form $c(t-\theta)^{n}$ for some $c\in \KK^{\times}$ and an integer $n$. Let $\TT$ be the Tate algebra in the variable $t$ over $\CC_{\infty}$, which consists of all formal power series in $\power{\CC_{\infty}}{t}$ that are convergent on the closed unit disk.

Let $\sigma$ act on $\TT\otimes_{\KK[t]}\eM$ via $f\otimes m\mapsto f^{(-1)}\otimes\sigma m$. Then it is not hard to see that the natural map
\begin{equation}\label{E:Inj MtensorTsigma}
  \TT\otimes_{\FF_{q}[t]} \left(\TT\otimes_{\KK[t]} \eM \right)^{\sigma}  \rightarrow \TT\otimes_{\KK[t]} \eM
\end{equation}
is injective, where $(\TT\otimes_{\KK[t]} \eM)^{\sigma}$ denotes the part fixed by $\sigma$.

\begin{definition}
Let $\eM$ be a dual $t$-motive. $\eM$ is said to be \emph{uniformizable}  (or rigid analytically trivial) if the injective map \eqref{E:Inj MtensorTsigma} is also surjective.
\end{definition}

In fact, the equivalence in Theorem~\ref{thm: Anderson Eq} preserves uniformizability.

\begin{theorem}[{See~\cite[Prop.~2.4.17]{HJ20}}]\label{thm: Uni-eqv}
The category of unifomizable $A$-finite $t$-modules is equivalent to the category of uniformizable dual $t$-motives.
\end{theorem}

\begin{remark}\label{rem: HB(M)}
For a uniformizable dual $t$-motive $\eM$ over $\ok$, we define the \textit{Betti} module of $\eM$ by
\[
H_{\mathrm{Betti}}(\eM)\assign \bigl(\TT\otimes_{\KK[t]} \eM \bigr)^{\sigma},
\]
which is free of rank $r\assign \rank_{\ok[t]}\eM$ over $\FF_{q}[t]$. If we define  $\TT^\dagger\subset \TT$ to be the ring of rigid analytic functions on $\CC_\infty\setminus \{\theta^{q^i}\mid i \in \NN\}$, then the result of Hartl and Juschka~\cite[Prop.~2.3.30]{HJ20} shows that
\[
H_{\textup{Betti}}(\eM)=\left\{ x\in \TT^{\dagger}\otimes_{\ok[t]}\eM \mid  \sigma x=x \right\}.
\]
\end{remark}

\begin{remark}
We note that the map \eqref{E:Inj MtensorTsigma} is surjective if and only if there exists a matrix $\Uppsi\in \GL_{r}(\TT)$ satisfying the difference equation
\[
\Uppsi^{(-1)}=\Upphi\Uppsi.
\]
\end{remark}

\subsection{\texorpdfstring{Tannakian formalism for $t$-motives}{Tannakian formalism for t-motives}}\label{Seubset: Tannakian category of t-motives}
For later use in transcendence problems, in this subsection we fix $\KK$ to be $\ok$ and review the theory developed by the third author of the present paper~\cite{Papanikolas}.

We let $\ok(t)[\sigma,\sigma^{-1}]$ be the Laurent polynomial ring in $\sigma$ over the field $\ok(t)$ subject to the relations
\[
\sigma^{n}f=f^{(-n)}\sigma^{n}, \quad \forall n \in \ZZ.
\]
A left $\ok(t)[\sigma,\sigma^{-1}]$-module is called a \emph{pre-$t$-motive} if it is finite-dimensional over $\ok(t)$. Morphisms of pre-$t$-motives are left $\ok(t)[\sigma,\sigma^{-1}]$-module homomorphisms.

It is natural to pass from a dual $t$-motive to a pre-$t$-motive via tensor product. More explicitly, let $\eM$ be a dual $t$-motive and consider the tensor product $M\assign \ok(t)\otimes_{\ok[t]}\eM$. We extend the action of $\sigma$ on $\eM$ to this tensor product via $f\otimes m\mapsto f^{(-1)}\otimes \sigma m$. So $M$ is a finite-dimensional $\ok(t)$-vector space on which the action of $\sigma$ is invertible, and thereby one can regard it as a left $\ok(t)[\sigma,\sigma^{-1}]$-module. It follows that $M$ is a pre-$t$-motive.

Given a pre-$t$-motive of dimension $r$ over $\ok(t)$, we take $\bm\in \Mat_{r\times 1}(M)$ so that the entries of $\bm$ comprise a $\ok(t)$-basis of $M$.  Then the action of $\sigma$ on $M$ is
represented as
\[
\sigma \bm=\Upphi \bm
\]
for some $\Upphi\in \GL_{r}(\ok(t))$. Let $\LL$ be the fraction field of $\TT$. Then $M$ is called \emph{rigid analytically trivial} if there exists $\Psi\in \GL_{r}(\LL)$ satisfying
\[
\Uppsi^{(-1)}=\Upphi\Uppsi.
\]
Extending the action of $\sigma$ on $\LL\otimes_{\ok(t)}M$ via $f\otimes m\mapsto f^{(-1)}\otimes \sigma m$ and denoting by $H_{\mathrm{Betti}}(M)\assign (\LL\otimes_{\ok(t)}M)^{\sigma}$ the elements fixed by $\sigma$, one of the main results in \cite{Papanikolas} is as follows.

\begin{theorem}[{Papanikolas~\cite[Thm.~3.3.15]{Papanikolas}}]
The category $\mathcal{R}$ of rigid analytically trivial pre-$t$-motives forms a neutral Tannakian category over $\FF_{q}(t)$ with fiber functor
\[
M\mapsto H_{\textup{Betti}}(M).
\]
\end{theorem}

\begin{definition}\label{Def: Tannakian category of t-motives}\
\begin{enumerate}
\item We define the category $\mathcal{AR}^{I}$ of uniformizable dual $t$-motives up to isogeny as follows:
\begin{itemize}
\item Objects of $\mathcal{AR}^{I}$: uniformizable dual $t$-motives.
\item Morphisms of $\mathcal{AR}^{I}$: For uniformizable dual $t$-motives $\eM,\eN$,
\[
\Hom_{\mathcal{AR}^{I}}(\eM,\eN)\assign \FF_{q}(t)\otimes_{\FF_{q}[t]}\Hom_{\ok[t,\sigma]}\left( \eM,\eN  \right).
\]
\end{itemize}
\item We define the \emph{category $\mathcal{T}$ of $t$-motives} to be the strictly full Tannakian subcategory of $\mathcal{R}$ generated by the essential image of the functor
\[
\Bigl(\eM\mapsto \ok(t)\otimes_{\ok[t]}\eM \Bigr):\mathcal{AR}^{I}\rightarrow\mathcal{R}.
\]
\end{enumerate}
\end{definition}

Given a $t$-motive $M$, we let $\mathcal{T}_{M}$ be the strictly full Tannakian subcategory of $\mathcal{T}$ generated by $M$. Then by Tannakian duality there exists an affine algebraic group scheme $\Gamma_{M}$ defined over $\FF_{q}(t)$ so that $\mathcal{T}_{M}$ is equivalent to the $\FF_{q}(t)$-finite dimensional representations of the algebraic group $\Gamma_{M}$. We call $\Gamma_{M}$ the \emph{$t$-motivic Galois group} of $M$.

In the case that $M=\ok(t)\otimes_{\ok[t]}\eM$ for a uniformizable Anderson $t$-motive $\eM$, we let $\Upphi\in \Mat_{r}(\ok[t])\cap\GL_{r}(\ok(t))$ be the matrix representing the action of $\sigma$ on a $\ok[t]$-basis of $\eM$, and let $\Uppsi\in \GL_{r}(\TT)$ satisfy $\Uppsi^{(-1)}=\Upphi \Uppsi$. We note that by \cite[Prop.~3.1.3]{ABP} the entries of $\Uppsi$ are convergent on all of $\CC_{\infty}$. Let $\ok\left(\Uppsi(\theta)\right)$ be the field over $\ok$ generated by all the entries $\Uppsi_{ij}(\theta)$ of $\Uppsi(\theta)$. The following equality is a function field analogue of Grothendieck's periods conjecture for abelian varieties.

\begin{theorem}{\textnormal{(Papanikolas~\cite[Thm.~1.1.7]{Papanikolas})}}\label{T:dim=trdeg}
Given a uniformizable dual $t$-motive $\eM$, we define $M\assign \ok(t)\otimes_{\ok[t]}\eM$ and let $(\eM,\Upphi,\Uppsi)$ be given  as above. Then
\[
\dim \Gamma_{M}=\trdeg_{\ok}\ok(\Uppsi(\theta)).
\]
\end{theorem}

\section{CM fields and CM types}
\label{sec: CM setting}
\subsection{CM fields}
A finite  extension $\eF$ over $\FF_{q}(t)$ is called \emph{totally real} if the infinite place $\infty$ of $\FF_{q}(t)$ splits completely in $\eF$. We note that $\eF$ is then necessarily separable over $\FF_{q}(t)$ and that every separable finite extension $\eK/\FF_q(t)$ contains a maximal totally real intermediate field. We define the following notion of CM fields, which differs slightly from the classical case.

\begin{definition}\label{Def: CM fields}
A finite separable extension $\eK$ over $\FF_q(t)$ is called a \emph{CM field} if every place $\infty^{+}$ of its maximal real subfield $\eK^{+}$ lying over $\infty$ in $\FF_q(t)$ is non-split in $\eK$, i.e.,  $\infty^{+}$ has a unique extension to $\eK$.
\end{definition}

\begin{remark}\label{rem: CM fields}
(1) Based on this definition, the field $\FF_{q}(t)$ is itself a CM field. This is partly for convenience so that we can incorporate the period of the Carlitz module in this context. Similarly, every totally real field over $\FF_q(t)$ is also a CM field in our setting.

(2) Let $\eK$ be an \emph{imaginary} field over $\FF_q(t)$, i.e.\ the place $\infty$ of $\FF_q(t)$ is non-split in $\eK$, and suppose $\eK/\FF_q(t)$ is separable. Then $\eK$ is a CM field over $\FF_q(t)$ with $\eK^+ = \FF_q(t)$. In particular, the constant field extension $\FF_{q^\ell}(t)$ is a CM field over $\FF_q(t)$ for every positive integer $\ell$.

(3) Other typical examples of CM fields are the \emph{Carlitz cyclotomic extensions} over $\FF_q(t)$.
Given a non-constant polynomial $f\in \FF_{q}[t]$, the $f$-th Carlitz cyclotomic extension is the field $\eK_{f}$ generated by the $f$-torsion points of the Carlitz $\FF_{q}[t]$-module over $\FF_{q}(t)$. It is known that  $\eK_{f}$ is finite abelian over $\FF_{q}(t)$ with Galois group isomorphic to $\left( \FF_{q}[t]/(f)\right)^{\times}$ under the Artin map (see~\cite[Thm.~2.3]{Hayes74} or~\cite[Thm.~12.8]{Rosen}). The maximal totally real subfield $\eK_{f}^{+}$ is the fixed field of $\FF_{q}^{\times}$, and each place of $\eK_{f}^+$ lying above $\infty$ is totally ramified in $\eK_{f}$ with ramification index $q-1$ (see~\cite[Thm.~3.2]{Hayes74} or~\cite[Thm.~12.14]{Rosen}).

(4) Unlike the classical case, the compositum of two CM fields over $\FF_q(t)$ may not be a CM field.
For instance, let $\eK_1 = \FF_2(\sqrt[3]{t})$ and $\eK_2 = \FF_2(\sqrt[3]{t+1})$. The infinite place $\infty$ of $\FF_2(t)$ is totally ramified in both $\eK_1$ and $\eK_2$, which says in particular that $\eK_1$ and $\eK_2$ are imaginary fields over $\FF_2(t)$. The compositum $\eK\assign  \eK_1\cdot \eK_2$ has degree $9$ over $\FF_2(t)$ and contains $\eK_3\assign  \FF_2(\sqrt[3]{(t+1)/t})$. As
\[
\power{\FF_2}{1/t}[x] \ \ni\   x^3-\frac{t+1}{t} \ \equiv \  x^3 -1 \quad   \bmod (1/t),
\]
it follows from Hensel's lemma that there exists exactly one root of $x^3-(t+1)/t$ in $\laurent{\FF_2}{1/t}$. Thus $\infty$  splits into two places in $\eK_3$ (with residue degree equal to $1$ and $2$, respectively), and both of these are totally ramified in $\eK$ with ramification index $3$. Hence there are exactly two places of $\eK$ lying above $\infty$, which implies that $\eK$ cannot be a CM field.

(5) Let $\eK$ be a CM field and $\eF$ be a totally real field over $\FF_q(t)$ (contained in an algebraic closure of $\eK$). We claim that the compositum $\eK\cdot \eF$ of $\eK$ and $\eF$ is also a CM field with the maximal totally real subfield $\eK^+\cdot \eF$.
To prove this, let $\tilde{\infty}^+$ be a place of $\eK^+\cdot \eF$ lying above $\infty$ and take
a place $\tilde{\infty}$ of $\eK\cdot \eF$ lying above $\tilde{\infty}^+$.
We also denote by
$\infty_K$ and $\infty^+$ the places of $\eK$ and $\eK^+$ lying below $\tilde{\infty}$, respectively.
Since $\eK^+\cdot \eF$ is totally real over $\FF_q(t)$, one has that $\infty^+$ splits completely in $\eK^+\cdot \eF$.
Thus we have the following (in-)equalities of the ramification indices and residue degrees:
\begin{equation}\label{eqn: ef-1}
\begin{cases}
    e(\tilde{\infty}^+/\infty^+) = f(\tilde{\infty}^+/\infty^+) = 1; \\
    e(\infty_K/\infty^+)\leq e(\tilde{\infty}/\infty^+) = e(\tilde{\infty}/\tilde{\infty}^+);\\
    f(\infty_K/\infty^+)\leq f(\tilde{\infty}/\infty^+) =  f(\tilde{\infty}/\tilde{\infty}^+).
\end{cases}
\end{equation}
On the other hand, as $\eK$ is a CM field, one has that
\begin{align}\label{eqn: ef-2}
e(\infty_K/\infty^+)f(\infty_K/\infty^+) = [\eK:\eK^+] & \geq [\eK\cdot \eF:\eK^+\cdot \eF] \\
& \geq e(\tilde{\infty}/\tilde{\infty}^+)f(\tilde{\infty}/\tilde{\infty}^+). \nonumber
\end{align}
By \eqref{eqn: ef-1} and \eqref{eqn: ef-2} we get
$$
[\eK:\eK^+] = e(\infty_K/\infty^+)f(\infty_K/\infty^+) = e(\tilde{\infty}/\tilde{\infty}^+)f(\tilde{\infty}/\tilde{\infty}^+) = [\eK\cdot \eF:\eK^+\cdot \eF].
$$
This says that $\tilde{\infty}$ is the only place of $\eK\cdot \eF$ lying above $\tilde{\infty}^+$, i.e.\ $\tilde{\infty}^+$ is non-split in $\eK\cdot \eF$.
Since $\tilde{\infty}^+$ is taken arbitrarily, the claim holds.
\end{remark}

\subsection{Notation on curves} \label{sec: curvenot}
Given a CM field $\eK$ over $\FF_{q}(t)$, we will assume henceforth for simplicity  that $\eK/\FF_{q}(t)$ is a geometric extension. As the case of non-geometric CM fields requires some additional care, we leave the details of the general case to
Appendix~\ref{secA: CM}.

Recall that $\eK^{+}$ denotes the maximal totally real subfield of $\eK$. Let $X$ (resp.\ $X^{+}$) be a smooth, projective, geometrically connected algebraic curve defined over $\FF_{q}$ whose function field is $\eK$ (resp.\ $\eK^{+}$).  The field embeddings
\[
\FF_{q}(t)\hookrightarrow \eK^{+}\hookrightarrow \eK,
\]
induce corresponding morphisms of algebraic curves over $\FF_{q}$,
\[
\xymatrix{
X \ar[r]^{\pi_{X/X^{+}}} & X^{+} \ar[r]^{\pi_{X^{+}/\PP^{1}}}  & \PP^{1}.
}
\]
For a chosen algebraically closed field $\KK$ with $k\subset \KK \subset \CC_{\infty}$, we denote by $\bX$ (resp.\ $\bX^{+}$) the base change of $X$ (resp.\ $X^{+}$) to $\KK$, and  $\pi_{\bX/\bX^{+}}$ and  $\pi_{\bX^{+}/\PP^{1}}$ the base change of the respective morphisms above to $\KK$. Define $\bK\assign \KK(\bX)$ (resp.~$\bK^{+}$) to be the function field of $\bX$ (resp.\ $\bX^{+}$) over $\KK$.
Let $O_{\eK}$ (resp.\ $O_{\eK^{+}}$) be the integral closure of $\FF_{q}[t]$ in $\eK$ (resp.\ $\eK^{+}$), and put $U\assign \Spec O_{\eK}$ (resp.\ $U^{+}\assign \Spec O_{\eK^{+}}$).
The base change of $U$ and $U^+$ to $\KK$ are denoted
\[
\bU \assign  \Spec (O_\bK) \quad \textup{and} \quad \bU^+ \assign  \Spec(O_{\bK^+}),
\]
where
\[
O_{\bK}\assign \KK\otimes_{\FF_{q}} O_{\eK} \cong \KK[t]\otimes_{\FF_q[t]} O_\eK
\quad \textup{and} \quad O_{\bK^{+}}\assign \KK\otimes_{\FF_{q}}O_{\eK^{+}} \cong \KK[t] \otimes_{\FF_q[t]} O_\eK.
\]
Using this framework one checks that
\[
\bK = \KK(t)\otimes_{\FF_q(t)} \eK \quad \textup{and} \quad \bK^+ = \KK(t) \otimes_{\FF_q(t)} \eK^+.
\]
Note that $O_{\bK}$ (resp.\ $O_{\bK^{+}}$) is the integral closure of $\KK[t]$ in $\bK$ (resp.\ $\bK^{+}$), and that
\[
\bU = \bX\setminus \pi_{\bX/\PP^1}^{-1}(\infty) \quad (\text{resp.} \quad \bU^{+}= \bX^{+}\setminus \pi_{\bX^+/\PP^1}^{-1}(\infty)).
\]

\subsection{CM types}\label{sec: CM types}
We continue in the setting above starting with a given geometric CM field $\eK$ over $\FF_{q}(t)$, and for now we set $\KK = \CC_{\infty}$ in \S\ref{sec: curvenot}. Let
\begin{equation}\label{E:JK}
J_\eK\assign  \pi_{\bX/\PP^{1}}^{-1}(\theta) = \{\xi \in \bX(\CC_{\infty})\mid \pi_{\bX/\PP^{1}}(\xi) = \theta \}
\end{equation}
be the set of points of $\bX$ lying above $\theta$. Denote by $\Emb(\eK,\CC_\infty)$ the set of all $\FF_q$-algebra embeddings from $\eK$ into $\CC_\infty$ which send $t$ to $\theta$. We may identify $J_\eK$ with $\Emb(\eK,\CC_\infty)$ as follows.  Each $\xi \in J_\eK$ corresponds to an $\FF_q$-morphism
\[
\Spec(\CC_\infty)\rightarrow X.
\]
As $\xi$ is not $\overline{\FF}_q$-valued (where $\overline{\FF}_q$ is the algebraic closure of $\FF_q$ in $\CC_\infty$), this morphism must factor through the generic point $\Spec(\eK)\hookrightarrow X$.
Thus we obtain an $\FF_q$-morphism $\Spec(\CC_\infty)\rightarrow \Spec(\eK)$, which corresponds to an $\FF_q$-algebra embedding
\[
\nu_\xi:\eK\hookrightarrow \CC_\infty, \quad \nu_\xi(t) = \theta,
\]
where the latter part follows from $\pi_{\bX/\PP^1}(\xi) = \theta$. Conversely, for $\nu \in \Emb(\eK,\CC_\infty)$, the induced $\FF_q$-morphism
\[
\Spec(\CC_\infty)\stackrel{\nu^*}{\longrightarrow}
\Spec(\eK)\hookrightarrow X
\]
gives a $\CC_\infty$-valued point $\xi_\nu$ of $\bX$ so that $\pi_{\bX/\PP^1}(\xi_\nu) = \theta$ (as $\nu(t) = \theta$). It is clear that
\[
\xi_{\nu_\xi} = \xi, \quad \forall \xi \in J_\eK, \quad \textup{and} \quad \nu_{\xi_\nu} = \nu, \quad \forall \nu \in \Emb(\eK,\CC_\infty).
\]

Let $I_\eK$ be the free abelian group generated by all elements of $J_\eK$. We may view $I_\eK$ as a subgroup of $\Div(\bX)$ consisting of all the divisors supported on  $\pi_{\bX/\PP^{1}}^{-1}(\theta)$.

\begin{definition}\label{defn: CM type}
Let $\eK$ be a geometric CM field over $\FF_q(t)$ with maximal totally real subfield $\eK^+$. Let $I_\eK^0$ be the subgroup of $I_\eK$ consisting of the divisors
\[
\Phi = \sum_{\xi \in J_\eK} m_\xi \xi, \quad m_\xi \in \ZZ,
\]
satisfying
\[
 \sum_{\xi\in \pi_{\bX/\bX^+}^{-1}(\xi_1^+)} m_\xi= \sum_{\xi'\in \pi_{\bX/\bX^+}^{-1}(\xi_2^+)} m_{\xi'},
\quad \forall \xi_1^+,\, \xi_2^+ \in J_{\eK^+}.
\]
For $\Phi \in I_{\eK}^0$, the \emph{weight} of $\Phi$ is then
\[
\wt(\Phi)\assign \sum_{\xi\in \pi_{\bX/\bX^+}^{-1}(\xi^+)} m_\xi \quad \text{ for any } \xi^+ \in J_{\eK^+}.
\]
We say that a \emph{generalized CM type} of $\eK$ is a nonzero effective divisor $\Xi$ in $I_\eK^0$, and that a \emph{CM type} of $\eK$ is a generalized CM type $\Xi$  for which $\text{wt}(\Xi) = 1$.
\end{definition}

\begin{remark}
Let $d = [\eK^+:\FF_q(t)]$. The correspondence between $J_{\eK^+}$ and $\Emb(\eK^+,\CC_\infty)$ implies that there are exactly $d$ points in $J_{\eK^+}$, say $\xi_1^+, \ldots, \xi_d^+$. Then a  CM type of $\eK$ must be of the form $\sum_{i=1}^{d} \xi_{i}\in I_\eK$, where $\xi_{i} \in \bX(\CC_\infty)$ is a point lying above $\xi_{i}^{+}$ for $1\leq i\leq d$.
\end{remark}

The following basic properties are straightforward from the definitions.

\begin{proposition}\label{prop: CM prop}
Let $\eK$ be a geometric CM field over $\FF_{q}(t)$.
\begin{enumerate}
\item Define the generalized CM type
\[
N_\eK \assign  \sum_{\xi \in J_\eK} \xi \in I_\eK^0.
\]
Then  $\wt(N_\eK) = [\eK:\eK^+]$, and for any $\Phi_0 \in I_\eK^0$, there exists $\mu \in \ZZ_{\geq 0}$ so that
\[
\mu N_\eK + \Phi_0
\]
becomes a generalized CM type of $\eK$.
\item Every generalized CM type $\Xi$ of $\eK$ of $\wt(\Xi) = n$ can be written as
\[
\Xi = \Xi_1+ \cdots + \Xi_n,
\]
where $\Xi_1,\ldots, \Xi_n$ are  CM types of $\eK$.
\end{enumerate}
Consequently, $I_\eK^0$ is generated by all the CM types of $\eK$.
\end{proposition}

Suppose that $\eK$ and $\eK'$ are geometric CM fields over $\FF_{q}(t)$ and that $\varrho: \eK'\rightarrow \eK$ is an $\FF_{q}(t)$-embedding. Let $\pi_\varrho: X\rightarrow X'$ be the corresponding surjective $\FF_q$-morphism of algebraic curves over $\FF_{q}$. When it is clear from the context, we  will still denote the corresponding base changes to $\KK$ by $\varrho: \bK' \rightarrow \bK$ and $\pi_{\varrho}: \bX \rightarrow \bX'$. Then $\pi_\varrho$ induces a map $J_{\eK} \rightarrow J_{\eK'}$.
From the one-to-one correspondence
\[
\begin{array}{ccc}
J_{K} & \longleftrightarrow & \Emb\left(\eK,\CC_{\infty} \right),\\
 \xi     & \longleftrightarrow  & \nu_{\xi}
\end{array}
\] we obtain that
\[
\nu_{\pi_\varrho(\xi)} = \nu_\xi \circ \varrho.
\]
In particular, for any $\varrho\in \Aut_{\FF_{q}(t)}(\eK)$, the automorphism group of $\eK$ over $\FF_{q}(t)$, the corresponding map $\pi_\varrho$ gives a permutation on $J_\eK$. Defining
\[
\xi^\varrho\assign  \pi_\varrho(\xi)
\]
for $\xi\in J_{\eK}$, one has that
\[
(\xi)^{\varrho_1 \varrho_2} = (\xi^{\varrho_1})^{\varrho_2}, \quad \forall \xi \in J_\eK\ \textup{and}\ \varrho_1,\,\varrho_2 \in \Aut_{\FF_{q}(t)}(\eK).
\]
Extending additively, we have a right action of $\Aut_{\FF_q(t)}(\eK)$ on $I_\eK$, and the following lemma holds.

\begin{lemma}\label{lem: right-action}
For a generalized CM type $\Xi = \sum_{\xi \in J_\eK} m_\xi \xi$ of $\eK$ and $\varrho \in \Aut_{\FF_q(t)}(\eK)$,
\[
 \Xi^\varrho\assign  \sum_{\xi \in J_\eK} m_\xi \cdot  \xi^\varrho
\]
is also a generalized CM type of $\eK$. In other words, $I_\eK^0$ is stable under the right action of $\Aut_{\FF_q(t)}(\eK)$.
\end{lemma}

\begin{remark}\label{rem: left-action}
Since $\eK$ is separable over $\FF_q(t)$, every point $\xi \in J_\eK$ is actually defined over $k^{\sep}$. This also implies that $\nu_\xi$ embeds $\eK$ into $k^{\sep}$. We then have a natural (left) action of $\Gal(k^{\sep}/k)$ on $J_\eK$ so that
\[
\nu_{\varsigma(\xi)} = \varsigma \circ \nu_\xi, \quad \forall \xi \in J_\eK \text{ and } \varsigma \in \Gal(k^{\sep}/k).
\]
This gives a $\Gal(k^{\rm sep}/k)$-module structure on $I_\eK$, and $I_\eK^0$ is clearly a submodule of $I_\eK$.
\end{remark}

\subsection{Restriction and inflation of CM types}\label{sub:Res and Inf}
Let $\eK_1$ and $\eK_2$ be two geometric CM fields over $\FF_{q}(t)$ with $\eK_1 \subset \eK_2$. Let $X_1$ (resp.\ $X_2$) be the projective, smooth, geometrically connected curve over $\FF_q$ associated to $\eK_1$ (resp.\ $\eK_2$), and let $\pi:X_2 \rightarrow X_1$ be the morphism over $\FF_q$ corresponding to the embedding $\eK_1\hookrightarrow \eK_2$. Define
\[
\Inf_{\eK_2/\eK_1}:I_{\eK_1}\longrightarrow I_{\eK_2} \quad \textup{and} \quad \Res_{\eK_2/\eK_1}: I_{\eK_2}\longrightarrow I_{\eK_1}
\]
by
\begin{align*}
\Inf_{\eK_2/\eK_1}(\Phi_1) &\assign  \pi^*(\Phi_1), \quad \forall \Phi_1 \in I_{\eK_1},\\
\Res_{\eK_2/\eK_1}(\Phi_2) &\assign  \pi_*(\Phi_2), \quad \forall \Phi_2 \in I_{\eK_2}.
\end{align*}
Then we have
\[
\Res_{\eK_2/\eK_1}\Bigl( \Inf_{\eK_2/\eK_1}(\Phi_1)\Bigr) = \pi_*\bigl(\pi^*(\Phi_1)\bigr) = \deg \pi \cdot \Phi_1 = [\eK_2:\eK_1]\cdot \Phi_1, \  \forall \Phi_1 \in I_{\eK_1},
\]
and one checks that
\[
\Inf_{\eK_2/\eK_1}\bigl(I_{\eK_1}^0\bigr) \subset I_{\eK_2}^0 \quad \textup{and} \quad
\Res_{\eK_2/\eK_1}\bigl(I_{\eK_2}^0\bigr) \subset I_{\eK_1}^0,
\]
which in particular implies the following basic result.

\begin{proposition}\label{prop: Inf-Res}
Let $\eK_1$ and $\eK_2$ be two geometric CM fields over $\FF_{q}(t)$ with $\eK_1 \subset \eK_2$.  For each generalized CM type $\Xi_1$ of $\eK_1$ \textup{(resp.\ $\Xi_2$ of $\eK_2$)}, the divisor $\Inf_{\eK_2/\eK_1}(\Xi_1)$ \textup{(resp.\ $\Res_{\eK_2/\eK_1}(\Xi_2)$)} is a generalized CM type of $\eK_2$ \textup{(resp.\ $\eK_1$)}.
\end{proposition}

\subsection{Reduction of points at infinity}\label{sec: R-pt-inf}
Let $\eK$ be a CM field which is geometric over $\FF_q(t)$, and recall the one-to-one correspondence between $J_\eK$ and $\Emb(\eK,\CC_\infty)$. Put
\[
O_{\CC_\infty}\assign  \{\alpha \in \CC_\infty\mid |\alpha|_\infty \leq 1\}.
\]
For each $\xi \in J_\eK$, the associated embedding $\nu_\xi \in \Emb(\eK,\CC_\infty)$ provides a valuation ring $\nu_\xi^{-1}(O_{\CC_\infty})$ of $\eK$, whose maximal ideal gives rise to a closed point of $X$, and hence a morphism
\[
\Spec \bigl(\nu_\xi^{-1}(O_{\CC_\infty})\bigr) \rightarrow X.
\]
The ring homomorphisms
\[
\nu_\xi^{-1}O_{\CC_\infty} \overset{\nu_{\xi}}{\hookrightarrow} O_{\CC_\infty} \twoheadrightarrow \overline{\FF}_{q}, \]
where the latter map is given by reduction, gives us an $\overline{\FF}_q$-valued point of $X$,
\[
\Spec(\overline{\FF}_q) \longrightarrow \Spec(O_{\CC_\infty}) \stackrel{\nu_\xi^*}{\longrightarrow}
\Spec\bigl(\nu_\xi^{-1}(O_{\CC_\infty})\bigr) \longrightarrow X,
\]
denoted by $\overline{\infty}_\xi$, so that $\pi_{\bX/\PP^1}(\overline{\infty}_\xi) = \infty$. To summarize we have the following maps
\[
\begin{tabular}{ccccc}
$J_\eK$ & $\stackrel{\sim}{\longleftrightarrow}$ & $\Emb(\eK,\CC_\infty)$ & $\xtwoheadrightarrow{}$ &  $\pi_{\bX/\PP^1}^{-1}(\infty)$ \\
$\xi$ & $\longleftrightarrow$ & $\nu_\xi$ & $\longmapsto$ & $\overline{\infty}_\xi$,
\end{tabular}
\]
where the surjectivity of  $\Emb(\eK,\CC_\infty) \twoheadrightarrow  \pi_{\bX/\PP^1}^{-1}(\infty)$ is implied by the following proposition.

\begin{proposition}\label{prop: reduction at infinity}
Let $\eK/\FF_{q}(t)$ be a geometric CM field with notation given as above. Then
\[
\pi_{\bX/\PP^{1}}^{-1}(\infty)= \bigl\{\overline{\infty}_\xi\mid \xi\in J_{\eK} \bigr\}.
\]
\end{proposition}

\begin{proof}
Given $\overline{\infty} \in \bX$ with $\pi_{\bX/\PP^1}(\overline{\infty}) = \infty$, we know that $\overline{\infty}$ must be $\overline{\FF}_q$-valued, and the valuation of $O_{\eK,\overline{\infty}}$ on $\eK$ associated to (the $\Gal(\overline{\FF}_q/\FF_q)$-orbit of) $\overline{\infty}$ is extended from the one on $\FF_q(t)$ associated to $\infty \in \PP^1$. Take an ($\FF_q$-algebra) embedding $\nu: O_{\eK,\overline{\infty}}\hookrightarrow O_{\CC_\infty}$ so that the induced morphism
\[
\Spec(\overline{\FF}_q) \longrightarrow \Spec(O_{\CC_\infty}) \overset{\nu^*}{\longrightarrow} \Spec\bigl(O_{\eK,\overline{\infty}}\bigr) \longrightarrow X
\]
corresponds to $\overline{\infty}$. Extending $\nu$ to an embedding $\eK\hookrightarrow \CC_\infty$ (still denoted by $\nu$), we obtain that $\nu \in \Emb(\eK,\CC_\infty)$. Thus the bijection between $J_\eK$ and $\Emb(\eK,\CC_\infty)$ implies that there exists $\xi \in J_\eK$ so that $\nu = \nu_\xi$. The construction of $\nu$ gives us that $\overline{\infty} = \overline{\infty}_\xi$, and the result holds.
\end{proof}

\begin{remark}\label{rem: Residue disc}
(1) Background for the correspondence between $\xi$ and $\overline{\infty}_\xi$ above is that the point $\xi$ lies in the residue disk around the infinite point $\overline{\infty}_\xi$ using Sinha's Green's function \cite[Section 2.3]{Sinha97}.

(2) When $\eK = \eK^+$ is totally real over $\FF_q(t)$, the infinite place $\infty$ of $\FF_q(t)$ splits completely in $\eK$. In this case, there are exactly $[\eK:\FF_q(t)]$ points of $X$ ($\FF_q$-valued) lying above $\infty$, and the corresponding valuations give all the embeddings from $\eK$ to $\laurent{\FF_q}{1/t}$. Together with the isomorphism $\laurent{\FF_q}{1/t} \cong k_\infty \subset \CC_\infty$ via the evaluation map sending $t$ to $\theta$, we obtain a one-to-one correspondence between $\Emb(\eK,\CC_\infty)$ and $\pi_{\bX/\PP^1}^{-1}(\infty)$, whence the above surjective map $J_\eK\twoheadrightarrow \pi_{\bX/\PP^1}^{-1}(\infty)$
is actually bijective when $\eK$ is totally real over $\FF_q(t)$.
For an arbitrary geometric CM field $\eK$ with maximal totally real subfield $\eK^+$, we then have the following commutative diagram:
\begin{equation}\label{eqn: Reduction}
\SelectTips{cm}{}
\xymatrix
{
J_\eK \ar@{<->}[r]^<<<<<\sim \ar@{->>}[d]_{\pi_{\bX/\bX^+}} &
\Emb(\eK,\CC_\infty) \ar@{->>}[r] \ar@{->>}[d]^{(\cdot)|_{\eK^+}} & \pi^{-1}_{\bX/\PP^1}(\infty) \ar@{->>}[d]^{\pi_{\bX/\bX^+}} \\
J_{\eK^+}\ar@{<->}[r]^<<<<\sim & \Emb(\eK^+,\CC_\infty)\ar@{<->}[r]^<<<<\sim & \pi_{\bX^+/\PP^1}^{-1}(\infty).
}
\end{equation}
\end{remark}

\begin{definition}\label{Def: IXi}
Given a generalized CM type $\Xi = \sum_{\xi \in J_{\eK}} m_\xi \xi$ of a geometric CM field $\eK$ over $\FF_q(t)$, we set
\[
\eI_\Xi \assign  \sum_{\xi \in J_{\eK}} m_\xi \overline{\infty}_\xi.
\]
\end{definition}

Note that $\eI_\Xi$ is an $\overline{\FF}_q$-rational divisor of $\bX$ supported on $\pi_{\bX/\PP^1}^{-1}(\infty)$, and that $\Xi - \eI_\Xi$ has degree zero.
In the next section, we shall apply Lang's isogeny theorem  to introduce the shtuka functions associated to $\Xi$ and to define CM dual $t$-motives via geometric constructions.

\begin{remark}
The above constructions and definitions for non-geometric CM fields are contained in  Appendix~\ref{secA: CM}.
\end{remark}

\section{\texorpdfstring{CM dual $t$-motives}{CM dual t-motives}}
\label{sec: CM dual t-motives}

Throughout this section, we fix $\eK$ to be a CM field that is geometric over $\FF_q(t)$ and $X$ to be the curve over $\FF_q$ associated with $\eK$. We further fix an algebraically closed subfield $\KK$ of $\CC_{\infty}$ containing $k$. Recall that the base change of $X$ to $\KK$ is denoted by $\bX$, and $ \KK(\bX)$ is the function field of  $\bX$ over $\KK$. For each integer $i$, we define the $i$-th Frobenius twist automorphism $(f\mapsto f^{(i)}):\KK(\bX)\rightarrow \KK(\bX)$, where $f^{(i)}$ is the function in $\KK(\bX)$ obtained by raising each coefficient of $f$ to the $q^{i}$-th power. For each closed point $P$ on $\bX$, we define $P^{(i)}$ to be the point on $\bX$ obtained by raising each coordinate of $P$ to the $q^{i}$-th power. This action extends linearly to the divisor group $\Div(\bX)$.

\subsection{Shtuka functions}\label{sec: shtuka}
In what follows, we will construct a dual $t$-motive from a generalized CM type along the lines of \cite{Sinha97, BP02, ABP}. We first need the following lemma that is a consequence of Lang's ``isogeny theorem.''

\begin{lemma}\label{L:Diff,W,h}
Let $k\subset \KK \subset \CC_{\infty}$ be an algebraically closed intermediate field.  Let $\eK$ be a  CM field, geometric over $\FF_{q}(t)$, and let $\Xi$ be a generalized CM type of $\eK$. Then there exists a divisor $W\in \Div(\bX)$ and a function $h\in \KK(\bX)$ so that
\begin{equation}\label{E:Shtuka}
\divv(h) = W^{(1)}-W+\Xi-\eI_\Xi.
\end{equation}
\end{lemma}

\begin{proof}
As mentioned in Definition~\ref{Def: IXi},  the divisor $\Xi-\eI_\Xi$ is of degree zero and hence the class $[\Xi-\eI_\Xi]$ can be regarded as a point in the $\KK$-valued points of the Jacobian $J_{\bX}(\KK)$ of the curve $\bX$, which is identified with the group of degree zero divisors $\Div^{0}(\bX)$ modulo principal divisors. Lang's isogeny theorem~\cite[Cor., p.~557]{Lang56} asserts that the map
\[
\left([D]\rightarrow [D-D^{(1)}]\right):J_{\bX}(\KK)\rightarrow J_{\bX}(\KK)
\]
is surjective. It follows that there is a divisor $W$ satisfying
\[
W-W^{(1)}\sim \Xi-\eI_\Xi.
\]
In other words, there exists $h \in \KK(\bX)$ with the desired divisor.
\end{proof}

\begin{definition}
The function $h$ in \eqref{E:Shtuka} is called a \emph{shtuka function} associated to the divisor $W$ for the generalized CM type $(\eK,\Xi)$.
\end{definition}

\begin{remark}
For the definition of shtuka functions in the case of non-geometric CM fields over $\FF_q(t)$, see Lemma~\ref{lem: B-shtuka function}.
\end{remark}

\subsection{\texorpdfstring{Constructions of dual $t$-motives}{Construction of dual t-motives with CM types}}
\label{sec: construction of CM}
For a geometric CM field $\eK$ over $\FF_{q}(t)$, we fix a generalized CM type $\Xi$ of $\eK$ and let $(W,h)$ be given as in Lemma~\ref{L:Diff,W,h}. Recall that $O_\eK$ is the integral closure of $\FF_q[t]$ in $\eK$ and $O_\bK = \KK \otimes_{\FF_q} O_\eK$. We also put
\[
\bU\assign \Spec O_{\bK}=\bX\setminus  \pi_{\bX/\PP^{1}}^{-1}(\infty).
\]
Define
\begin{align}\label{E:M(W,h}
\eM = \eM_{(W,h)} & \assign \Gamma\left(\bU, O_{\bX}(-W^{(1)})  \right) \\ & \bigl(=\bigl\{ f\in \KK(\bX)\bigm| \bigl(\divv(f)-W^{(1)}\bigr)|_{\bU}\geq 0\bigr\}\bigr), \nonumber
\end{align}
which is a projective $O_\bK$-module of rank one, together with the $\sigma$-action,
\[
\sigma m \assign  h \cdot m^{(-1)}, \quad \forall m \in \eM.
\]
As $\KK[t]$ is contained in $O_\bK$, the above $\sigma$-action gives $\eM$ the structure of a left $\KK[t,\sigma]$-module.
Moreover, the basic approach of Sinha~\cite{Sinha97} (cf.~\cite{ABP}, \cite{BP02}) shows the following:

\begin{theorem}\label{thm: CM dual}
Let hypotheses and notation be as in Lemma~\ref{L:Diff,W,h}. The left $\KK[t,\sigma]$-module $\eM=\Gamma(\bU,\cO_{\bX}(-W^{(1)}))$ is a dual $t$-motive.
\end{theorem}

\begin{proof}
(cf.~\cite[Thm.~3.2.3]{Sinha97} and \cite[Lemma~6.4.1]{ABP}) As $\eM$ is projective of rank one over $O_\bK$, it is free of finite rank over $\KK[t]$. Further, multiplication by $h$ induces a $\KK$-linear isomorphism of sheaves
\[
\mathcal{O}_{\bX}(-W)\simeq \mathcal{O}_{\bX}(-W^{(1)}-\Xi+\eI_\Xi),
\]
whence
\[
\sigma \eM= \Gamma\left(\bU,  \mathcal{O}_{\bX}(-W^{(1)}-\Xi+\eI_\Xi) \right)
 = \Gamma \left( \bU, \mathcal{O}_{\bX}(-W^{(1)}-\Xi) \right).
\]
It follows that if $\Xi\assign \sum_{\xi\in J_{\eK}}m_{\xi} \cdot \xi$, then
\[
(t-\theta)^{N}\eM \subset \sigma \eM
\]
for all $N\geq \max \{ m_{\xi}\}$.

It remains to show that $\eM$ is finitely generated over $\KK[\sigma]$. We let $j$ be the smallest positive integer so that $\eI_\Xi^{(j)}=\eI_\Xi$ and denote by $\mathbf{I}\assign \eI_\Xi+\eI_\Xi^{(-1)}+\cdots+\eI_\Xi^{(-j+1)}$, whence $\mathbf{I}^{(1)}=\mathbf{I}$. We claim that the support of $\mathbf{I}$ is equal to $\pi_{\bX/ \PP^{1}}^{-1}(\infty)$.

To see this, notice first that the support of $\mathbf{I}$ is contained in $\pi_{\bX/\PP^1}^{-1}(\infty)$. On the other hand, let $\eK^+$ be the maximal totally real subfield of $\eK$. Then by the definition of generalized CM type,
\[
\sum_{\xi \in \pi_{\bX/\bX^+}^{-1}(\xi^+)} m_\xi = \text{wt}(\Xi) >0, \quad \forall \xi^+ \in J_{\eK^+}.
\]
Hence the commutative diagram~\eqref{eqn: Reduction} ensures that
\[
\Supp(\bI)\cap \pi_{\bX/\bX^+}^{-1}(\infty^+) \supseteq \Supp(\eI_\Xi)\cap \pi_{\bX/\bX^+}^{-1}(\infty^+) \neq \emptyset, \quad \forall \infty^+ \in \pi_{\bX^+/\PP^1}^{-1}(\infty).
\]
Since $\eK$ is a CM field, one has that $\pi_{\bX/\bX^+}^{-1}(\infty^+)$ is actually an orbit under Frobenius twisting. As $\bI^{(1)} = \bI$, we obtain that
$\Supp(\bI)\supset \pi_{\bX/\bX^+}^{-1}(\infty^+)$ for each
$\infty^+ \in \pi_{\bX^+/\PP^1}^{-1}(\infty)$.
Therefore $\pi_{\bX/\PP^1}^{-1}(\infty)$ is contained in the support of $\bI$, and the claim follows.

We now write
\[
\eM=\bigcup_{n\in \mathbb{N}} \eM_{n}, \quad \text{ where } \quad
\eM_{n}\assign \Gamma\left(\bX,\mathcal{O}_{\bX}(-W^{(1)}+n {\bI})  \right).
\]
For each $n>1$, multiplication by $\sigma^j$ induces an isomorphism
\[
\sigma^j:\frac{\eM_{n}}{\eM_{n-1}}\rightarrow
\frac{\Gamma\left(\bX,\mathcal{O}_{\bX}(-W^{(1)}-\Xi-\Xi^{(-1)}-\dots-\Xi^{(-j+1)}+(n+1)
    \bI)
  \right)}{\Gamma\left(\bX,\mathcal{O}_{\bX}(-W^{(1)}-\Xi-\Xi^{(-1)}-\dots-\Xi^{(-j+1)}+n
    \bI) \right) }.
\]
By the Riemann-Roch theorem \cite[p.\ 80, Cor.~2]{Mumford}, for $n$ sufficiently large the natural injective map
\[
\pi_{n}: \frac{\Gamma\left(\bX,\mathcal{O}_{\bX}(-W^{(1)}-\Xi-\Xi^{(-1)}-\dots-\Xi^{(-j+1)}+(n+1)
    \bI)
  \right)}{\Gamma\left(\bX,\mathcal{O}_{\bX}(-W^{(1)}-\Xi-\Xi^{(-1)}-\dots-\Xi^{(-j+1)}+n
    \bI) \right) }
\rightarrow \frac{\eM_{n+1}}{\eM_{n}}
\]
is an isomorphism of $\KK$-vector spaces since the dimensions of both sides are equal to $\deg \bI$. So for a sufficiently large integer $n_{0}$ and for $n \geq n_{0}$, the composition gives an isomorphism
\[
\pi_{n}\circ
\sigma^j:\frac{\eM_{n}}{\eM_{n-1}}\stackrel{\sim}{\longrightarrow}
\frac{\eM_{n+1}}{\eM_{n}}.
\]
Since $\eM_{n_{0}}$ is of finite dimension over $\KK$, any $\KK$-basis of $\eM_{n_{0}}$ generates $\eM$ over $\KK[\sigma^j]$ and therefore over $\KK[\sigma]$. Therefore the proof is complete.
\end{proof}

The ranks of $\eM$ over $\KK[t]$ and $\KK[\sigma]$ are determined explicitly as follows.

\begin{proposition}\label{P:ranks}
The rank of $\eM$ as a dual $t$-motive is
\[
r(\eM)\assign \rank_{\KK[t]}\eM=[\eK:\FF_{q}(t)].
\]
The dimension of $\eM$ is
\[
d(\eM)\assign \rank_{\KK[\sigma]}\eM=\deg\Xi.
\]
\end{proposition}

\begin{proof}
As $\eM$ is projective of rank one over $O_\bK$, it must be free of rank $[\eK:\FF_q(t)]$ over $\KK[t]$.
The rank of $\eM$ over $\KK[\sigma]$ follows immediately from the standard result in the next lemma by taking $F=-W^{(1)}$ and $E=\Xi$.
\end{proof}

\begin{lemma}\label{L:degE}
Let $E$ and $F$ be fixed divisors on $\bX$ with the polar part of $E$ equal to the part of $E$ supported at infinity. Then
\[
\dim_{\KK}\Gamma\left( \bU,\cO_{\bX}(F)\right)/ \Gamma\left(\bU,\cO_{\bX}(F-E)  \right) =\deg E|_{\bU}.
\]
\end{lemma}

\begin{proof}
We note that the condition on $E$ forces $E|_{\bU}$ to be effective. It suffices to prove the case when $E = P$ is a $\KK$-valued point on $\bU$. Clearly, we have
\[
\dim_{\KK} \Gamma(\bU,\cO_\bX(F))/\Gamma(\bU,\cO_\bX(F-P))\leq 1.
\]
On the other hand, fix $\overline{\infty} \in \pi_{\bX/\PP^1}^{-1}(\infty)$. The Riemann-Roch theorem implies that for $n$ large enough, we can find
\[
0 \neq f \in \Gamma(\bX,\cO_\bX(F+n\overline{\infty})) \setminus \Gamma(\bX,\cO_\bX(F-P+n\overline{\infty})),
\]
which means in particular that
$\ord_P(f) = -\ord_P(F)$.
Hence
\[
0 \neq f \in \Gamma(\bU,\cO_\bX(F)) \setminus \Gamma(\bU,\cO_\bX(F-P))
\]
and
$\dim_{\KK} \Gamma(\bU,\cO_\bX(F))/\Gamma(\bU,\cO_\bX(F-P))\geq 1$. This completes the proof.
\end{proof}

\begin{remark}
The construction of the dual $t$-motive $\eM_{(W,h)}$ for a generalized CM type $(\eK,\Xi)$ when $\eK$ is a non-geometric CM field over $\FF_q(t)$ is given in \eqref{e:Appendix M(W,h)}.
\end{remark}

\begin{definition}\label{defn: CM dual t-motives}
Let $\eK$ be a geometric CM field over $\FF_{q}(t)$ and $\Xi$ be a generalized  CM type of $\eK$.
A \textit{CM dual $t$-motive with generalized CM type $(\eK,\Xi)$} is a dual $t$-motive isomorphic to $\eM_{(W,h)}$ for some pair $(W,h)$, where $h$ is a shtuka function associated to $W$ for $(\eK ,\Xi)$.
\end{definition}

\begin{proposition}\label{prop: tensor CM type}
Let $\eK$ be a CM field over $\FF_q(t)$.
Given CM dual $t$-motives $\eM_{i}$ with  generalized CM types $(\eK,\Xi_{i})$ for $1=1,\ldots,n$,
the tensor product
\[
\eM_1\otimes_{O_\bK} \cdots \otimes_{O_\bK} \eM_n
\]
with $\sigma$-action defined by
\[
\sigma (m_1\otimes \cdots \otimes m_n) \assign  \sigma  m_1 \otimes \cdots \otimes \sigma  m_n, \quad \forall m_i \in \eM_i,\ i = 1,\ldots, n,
\]
is a CM dual $t$-motive with generalized CM type $(\eK,\Xi_1+\cdots +\Xi_n)$.
\end{proposition}

\begin{proof}
Here we prove the case when $\eK$ is geometric over $\FF_q(t)$, and refer the reader to Proposition~\ref{propB: tensor} for the non-geometric case.

Without loss of generality, we may assume that $n=2$ and
\[
\eM_i = \eM_{(W_i,h_i)} = \Gamma\bigl(\bU,\Ocal_\bX(-W_i^{(1)})\bigr),
\]
where
$h_i$ is a shtuka function associated to $W_i$ for $(\eK,\Xi_i)$ for $i=1,2$.
Put
\[
W \assign  W_1+W_2, \quad h \assign  h_1 \cdot h_2, \quad \text{ and } \quad \Xi \assign  \Xi_1 + \Xi_2.
\]
One has that
\[
\divv(h) = W^{(1)}-W+\Xi-\eI_\Xi \quad \in \Div(\bX).
\]
Let $\eM = \eM_{(W,h)} = \Gamma\bigl(\bU, \Ocal_\bX(-W^{(1)})\bigr)$.
The $\Ocal_\bX$-module isomorphism
\[
\Ocal_\bX(-W_1^{(1)}) \otimes_{\Ocal_\bX}  \Ocal_\bX(-W_2^{(1)}) \stackrel{\sim}{\longrightarrow} \Ocal_\bX(-W^{(1)})
\]
induces the following $O_\bK$-module isomorphism
\[
\begin{tabular}{rccc}
$\Theta:$ & $\eM_1\otimes_{O_\bK}  \eM_2$ & $\stackrel{\sim}{\longrightarrow}$ & $\eM$ \\
& $m_1\otimes m_2$ & $\longmapsto$ & $m_1\cdot m_2$.
\end{tabular}
\]
Moreover, for $m_i \in \eM_i$, $i=1,2$, we have
\[
\Theta\bigl(\sigma (m_1 \otimes m_2)\bigr)
= \Theta(h_1m_1^{(-1)} \otimes  h_2 m_2^{(-1)})
=
h (m_1\cdot m_2)^{(-1)} = \sigma \Theta\big (m_1\otimes  m_2).
\]
Therefore $\eM_1\otimes_{O_\bK} \eM_2$ is isomorphic to the CM dual $t$-motive $\eM$ with the generalized CM type $(\eK,\Xi)$ as desired.
\end{proof}

To end this subsection, we derive the following.

\begin{proposition}\label{L:IsogW1W2}
CM dual $t$-motives with the same generalized CM type are isogeneous.
\end{proposition}

\begin{proof}
We prove the case of geometric CM fields over $\FF_q(t)$ here, and refer readers to Proposition~\ref{propB: isogeny} for the non-geometric case. Let $\eK$ be a geometric CM field over $\FF_q(t)$ and $X$ be the curve over $\FF_q$ associated with $\eK$.
Recall that $\bX$ is the base change of $X$ to $\KK$.
It is a theorem of Chow that the Jacobian $J_{\bX}$ of $\bX$ and the map of $\bX$ to $J_{\bX}$
are actually defined over $\FF_q$ (see \cite[Section~6]{Chow}).
By the
first half of that theorem, the $q$-power Frobenius acts on $J_{\bX}$, and $J_{\bX}(\FF_q)$ consists of the points on the Jacobian fixed by it.
Since the Brauer group of a finite field is trivial,
\cite[Remark~1.6]{Milne} shows that $J_{\bX}(\FF_q) \simeq
\text{Pic}^0_{\bX}(\FF_q)$.
By this and the second half of Chow's
theorem, points of $J_{\bX}(\FF_q)$ are precisely those represented
by the $\FF_q$-rational divisors of degree zero.

Now, let $\Xi$ be a generalized CM type of $\eK$.
Take $h_{1}$ and
$h_{2}$ to be shtuka functions associated to divisors
$W_{1}$ and $W_{2}$, respectively, for $(\eK,\Xi)$.
That is, for $i=1,2$,
  \[ W_{i}^{(1)}-W_{i}+\Xi-\eI_\Xi=\divv (h_{i}) . \]
Let $D_0$ be an $\FF_q$-rational divisor of degree one
on $\bX$ \cite[Cor.~VII.~5.5]{Weil}, and set $$W \assign  W_2-W_1 - (\deg W_2-\deg W_1)\cdot D_0.$$
Then $\deg W = 0$ and
$
W^{(1)} - W \thicksim 0
$.
Hence the divisor class $[W]$ corresponds to a point in $J_\bX(\FF_q)$.
The above discussion assures that $[W]$ is represented by an $\FF_q$-rational divisor of degree zero, which we
call $D_1$, so that there is a rational function $f_1$ on $\bX$ such that
$$W = D_1 + \divv(f_1).$$
Put $D_2 \assign  - D_1 -(\deg W_2 - \deg W_1)\cdot D_0$ and $Q \assign  \pi_{\bX/\PP^1}^*(\infty)$.
As $D_2$ and $Q$ are both $\FF_q$-rational and $\deg Q > 0$,
by the Riemann-Roch theorem there exists a sufficiently large integer $n$ and $f_2 \in \eK$ so that
$$
\divv(f_2) + D_2 + nQ \quad \text{ is effective.}
$$
Let $D_3 \assign  \divv(f_2) + D_2$
and $f \assign  f_1\cdot f_2$.
We obtain that
$D_3$ is an $\FF_q$-rational divisor of $\bX$ which is effective on $\bU$, and
\begin{align*}
W_2 - W_1 + D_3 &= W_2-W_1+ \divv(f_2)-D_1-(\deg W_2-\deg W_1)\cdot D_0 \\
&= W-D_1 + \divv(f_2) \ = \ \divv(f_1\cdot f_2) \ = \  \divv(f).
\end{align*}
In particular, we have that
$\divv(f^{(1)}/f)=\divv(h_{2}/h_{1})$.
Replacing $f$
by a suitable $\KK$-scalar multiple if necessary, we may further assume that $f^{(1)}/f = h_2/h_1$. Finally, define
\[
\Theta
:\eM_{1}\rightarrow \eM_{2} \quad m_{1}\longmapsto f^{(1)}m_{1}, \quad \forall m_1 \in \eM_1.
\]
As for each $m_1 \in \eM_1$,
\[
\divv(f^{(1)}m_{1}) - W_2^{(1)}
= (\divv(m_1) - W_1^{(1)}) + D_3
\]
is always effective on $\bU$,
$\Theta$ is a well-defined $O_\bK$-module homomorphism.
Moreover, the cokernel of $\Theta$ is finite dimensional over $\KK$ by Lemma~\ref{L:degE}, and
\[
\Theta(\sigma m_1)
= \Theta (h_1 m_1^{(-1)})
= f^{(1)}h_1 m_1^{(-1)}
= h_2 (f m_1^{(-1)})
= \sigma \Theta (m_1), \ \forall\, m_1 \in \eM_1.
\]
Therefore $\Theta: \eM_1\rightarrow \eM_2$ is actually an isogeny, and the assertion holds.
\end{proof}

\subsection{\texorpdfstring{Uniformization of CM dual $t$-motives}{Uniformization of CM dual t-motives}}\label{subsec: U-CM}

In this subsection, we shall outline the key ingredients of the argument of the uniformizability of CM dual $t$-motives, and refer the readers to Appendix~\ref{AppendixSubsec: Unif. of CM dual t-motives} for the complete details.
Without loss of generality,
we may choose $\KK = \CC_\infty$ throughout this subsection.

Let $\eK/\FF_{q}(t)$ be a CM field with maximal totally real subfield $\eK^+$.
We put $d = [\eK^+:\FF_q(t)]$ and let $\Xi = \xi_1+\cdots + \xi_d$ be a CM type of $\eK$.
Notice that
$\nu_i^+\assign  \nu_{\xi_i}\big|_{\eK^+}: \eK^+\rightarrow \CC_\infty$, $i = 1,\ldots, d$,
give all the embeddings from $\eK^+$ into $\CC_\infty$.
Recall that (see~\cite[Section 4.3]{Anderson86}) a \emph{Hilbert-Blumenthal $O_{\eK^+}$-module} over $\CC_\infty$ is a pure uniformizable abelian $t$-module $E$ of dimension $d$ equipped with a structure homomorphism $\rho_E : O_{\eK^+}\rightarrow \End_{\text{ab.}\ t-{\rm mod.}}(E)$ extending the tautological homomorphism $\FF_q[t] \rightarrow \End_{\textup{ab.\,$t$-mod.}}(E)$ such that the induced $\CC_\infty$-linear representation $\partial \rho_E : O_{\eK^+}\rightarrow \End_{\CC_\infty}(\Lie(E))$ is isomorphic to $\nu_1^+\oplus \cdots \oplus \nu_d^+$.
Set
\[
\nu_\Xi:\eK \rightarrow \CC_\infty^d, \quad \nu_\Xi(\alpha) \assign  (\nu_{\xi_1}(\alpha),\ldots, \nu_{\xi_d}(\alpha)), \quad \forall \alpha \in \eK.
\]
Let $O_\eK$ be the integral closure of $\FF_q[t]$ in $\eK$. Then $\nu_\Xi(O_\eK)$ becomes a discrete $A$-lattice in $\CC_\infty^d$.

\begin{proposition}
There exists a Hilbert-Blumenthal $O_{\eK^+}$-module $(\GG_a^d,\rho^\Xi)$ whose period lattice is $\nu_\Xi(O_\eK) \subset \CC_\infty^d = \Lie(\GG_a^d)$.
\end{proposition}

\begin{proof}
This follows directly from \cite[Theorem~7~(II)]{Anderson86}.
\end{proof}

Moreover, by \cite[Theorem~7~(I)]{Anderson86} we can actually extend $\rho^\Xi$ to an $\FF_q$-algebra homomorphism from $O_\eK$ into $\End_{\FF_q}({\GG_{a}^d}_{/ \CC_\infty})$ so that
$\partial \rho^\Xi_a = \nu_{\xi_1}(a)\oplus \cdots \oplus \nu_{\xi_d}(a)$ for every $a \in O_\eK$. Recall that in \eqref{eqn: M(rho)}, $\eM(\rho^\Xi)$ is the Hartl-Juschka dual $t$-motive associated to $\rho^{\Xi}$.

\begin{theorem}\label{thm: Mrho}
For each CM type $\Xi$ of a given CM field $\eK$ over $\FF_q(t)$,
$\eM(\rho^\Xi)$ is a CM dual $t$-motive with CM type $(\eK, \Xi)$.
\end{theorem}

\begin{remark}\label{Rmk:Compactiblity of CM type}
Note that in~\cite[Sec.~2.3]{BP02}, $\eK$ is called the CM field of the $t$-module $\rho^{\Xi}$ and $\rho^{\Xi}$ has CM type $\nu_{\Xi}$. Comparing with Theorem~\ref{thm: Mrho}, this matches with the present notion of the CM type of $\eM(\rho^\Xi)$.
\end{remark}

\subsubsection*{Outline of the proof of Theorem~\ref{thm: Mrho}}
${}$
\begin{itemize}
\item (Step 1.) As
\[
\rank_{\CC_\infty[t]} \eM(\rho^\Xi) = \rank_{\FF_q[t]} \bigl(\nu_\Xi(O_\eK)\bigr) = [\eK : \FF_q(t)]
\]
and the $\FF_q$-algebra homomorphism $\rho^\Xi: O_\eK \rightarrow \End_{\FF_q}(\GG_{a/\CC_\infty}^d)$ induces an $\FF_q[t]$-algebra embedding $O_\eK \hookrightarrow \End_{\CC_\infty[t,\sigma]}(\eM(\rho^\Xi))$, we have that
$\eM(\rho^\Xi)$ is projective of rank one over $O_\bK \assign  \CC_\infty \otimes_{\FF_q} O_\eK$ (see~Lemma~\ref{lemB: projectivity}).
\item (Step 2.) For each $\xi \in J_\eK$, let $\Pfk_{\xi}$ be the maximal ideal of $O_\bK$ corresponding to $\xi$, and
put $\Ifk_\Xi \assign  \prod_{i=1}^d\Pfk_{\xi_i}$.
Then we show in Lemma~\ref{lem: CM type of M(rho)} that
\[
\sigma \cdot \eM(\rho^\Xi) = \Ifk_\Xi \cdot \eM(\rho^\Xi).
\]
\item (Step 3.) Finally, since $\Xi$ is a CM type of $\eK$ and $\eM(\rho^\Xi)$ is pure (cf.\ Proposition~\ref{propC: PU-Mrho}), the proof rests on the following characterization assuring that $\eM(\rho^\Xi)$ is a CM dual $t$-motive with CM type $(\eK,\Xi)$.
\end{itemize}

We capture this situation in the following slightly more general result.

\begin{proposition}\label{prop: being CM-0}
{\rm (See~Proposition~\ref{prop: being CM})}
Let $\eK$ be a CM field over $\FF_q(t)$,
and $O_\eK$ be the integral closure of $\FF_q[t]$ in $\eK$.
Given a dual $t$-motive $\eM$, suppose the following conditions hold:
\begin{itemize}
    \item[(1)] $\rank_{\KK[t]}\eM = [\eK:\FF_q(t)]$;
    \item[(2)] there is an $\FF_q[t]$-algebra embedding $O_\eK \hookrightarrow \End_{\CC_\infty[t,\sigma]} \eM$, which induces an $O_\bK$-module structure on $\eM$;
    \item[(3)] there exists a (generalized) CM type $\Xi_\eM = \sum_{\xi \in J_\eK} m_\xi \xi$ of $\eK$ so that $\sigma \eM = \Ifk_{\Xi_\eM} \eM$ where $\Ifk_{\Xi_\eM} = \prod_{\xi \in J_\eK} \Pfk_\xi^{m_\xi}$;
    \item[(4)] $\eM$ is pure.
\end{itemize}
Then $\eM$ is a CM dual $t$-motive with (generalized) CM type $(\eK,\Xi_\eM)$.
\end{proposition}

\begin{remark}
We briefly sketch the proof of the above characterization property when $\eK$ is a geometric CM field over $\FF_q(t)$. Conditions (1) and (2) imply that
$\eM$ is a projective $O_{\mathbf{K}}$-module of rank one (see~Lemma~\ref{lemB: projectivity}).
We may write $\eM=\Jfk^{(1)}\cdot \bm$ for some fractional ideal $\Jfk^{(1)}$ of $O_{\mathbf{K}}$ and $\bm\in \eM$.
From condition (3), we derive that
\[
O_{\mathbf{K}}\cdot \sigma \bm=\Jfk^{(1)}\cdot \Jfk^{-1}\cdot \Ifk_{\Xi_\eM} \cdot \bm,
\]
and hence there exists a nonzero function $h\in\mathfrak{J}^{(1)}\cdot \mathfrak{J}^{-1}\cdot \Ifk_{\Xi} \subset \mathbf{K}$ for which
$\sigma \cdot \bm = h \bm$.
Therefore
\begin{equation}\label{E: (h)}
 h \cdot O_\bK = \Jfk^{(1)}\cdot \Jfk^{-1}\cdot \Ifk_{\Xi_\eM}.
\end{equation}
Let $W_{\Jfk}\assign \sum_{x\in \mathbf{U}} \ord_{\Pfk_{x}}(\Jfk)\cdot x\in \Div(\bX)$. The equation~\eqref{E: (h)} gives rise to
\[ \divv(h)=W_{\Jfk}^{(1)}-W_{\Jfk}+\Xi_\eM-\Upsilon, \]
where $\rm{Supp}(\Upsilon) \subset \bX \setminus \bU$.
One technical step is to use the purity of $\eM(\rho^{\Xi})$ to find a divisor $W_{\infty}\in \Div(\bX)$ with $\rm{Supp}(W_{\infty})\subset \bX\setminus \bU$ so that
$W_{\infty}^{(1)}-W_{\infty}=I_{\Xi_\eM}-\Upsilon $ (cf.\ Lemma~\ref{lem: CM type of M(rho)}).
Put $W\assign W_{\Jfk}+W_{\infty}$.
Then we obtain
\[ \divv(h) =W^{(1)}-W+\Xi-I_{\Xi_\eM} .\]
By definition,
\[
M_{(W,h)}\assign \Gamma\left(\bU, \mathcal{O}_{\bX}(-W^{(1)}) \right)\assign \left\{f\in \bK \mid \left(\divv(f)-W^{(1)}\right)|_{\bU} \geq 0\right\}=\mathfrak{J}^{(1)},
\]
and so the map $\Theta\assign \left(x\mapsto x\cdot \bm \right):M_{(W,h)}=\Jfk^{(1)}\rightarrow \eM =\Jfk^{(1)}\cdot \bm$ is an isomorphism of left $\KK[t,\sigma]$-modules, thus completing the proof.
\hfill $\Box$
\end{remark}

Since $(\GG_a^d,\rho^\Xi)$ is pure and uniformizable, so is $\eM(\rho^\Xi)$
(cf.\ Proposition~\ref{propC: PU-Mrho}).  We are now able to show the following.

\begin{theorem}\label{thm: Uniform}
Every CM dual $t$-motive is pure and uniformizable.
\end{theorem}

\begin{proof}
Let $\eK$ be a CM field and $\Xi$ be a generalized CM type of $\eK$.
Suppose ${\rm wt}(\Xi) = n$.
By Proposition~\ref{prop: CM prop} (2), we may write $\Xi = \Xi_1+\cdots + \Xi_n$, where $\Xi_1,\ldots, \Xi_n$ are CM types of $\eK$. For $1\leq i \leq n$, Theorem~\ref{thm: Mrho} implies that  $\eM_i\assign  \eM(\rho^{\Xi_i})$ is a CM dual $t$-motive with CM type $(\eK,\Xi_i)$.
The purity and uniformizability of $\eM(\rho^{\Xi_1}),\ldots, \eM(\rho^{\Xi_n})$ imply that (cf.\ \cite[Propsition 2.4.10 (e)]{HJ20})
\[
\eM_1 \otimes_{\CC_\infty[t]} \cdots \otimes_{\CC_\infty[t]} \eM_n
\]
is pure and uniformizable.

As $\eM_0 \assign  \eM_1 \otimes_{O_\bK} \cdots \otimes_{O_\bK} \eM_n$ is a quotient of $
\eM_1 \otimes_{\CC_\infty[t]} \cdots \otimes_{\CC_\infty[t]} \eM_n$ under the canonical homomorphism
\[
\eM_1 \otimes_{\CC_\infty[t]} \cdots \otimes_{\CC_\infty[t]} \eM_n \twoheadrightarrow \eM_1 \otimes_{O_\bK} \cdots \otimes_{O_\bK} \eM_n,
\]
it follows from~\cite[Proposition 2.4.10 (c)]{HJ20} that $\eM_0$ is also pure and uniformizable.
By Proposition~\ref{prop: tensor CM type}, $\eM_0$ is a CM dual $t$-motive with generalized CM type $(\eK,\Xi)$.
Since every CM dual $t$-motive $\eM$ with generalized CM type $(\eK,\Xi)$ is isogeneous to $\eM_0$ by Proposition~\ref{L:IsogW1W2},
we obtain that $\eM$ must be pure and uniformizable by \cite[Proposition 2.4.10 (c)]{HJ20} again.
\end{proof}

\section{Analogue of Shimura's period symbols}

In this section, we fix $\KK = \bar{k}$.

\subsection{de Rham module and pairing}\label{sec: dR module}

Let $\eM$ be a dual $t$-motive over $\bar{k}$.
For our purposes here the \emph{de Rham module of $\eM$ over $\bar{k}$} is the following (cf.\ \cite[Section 4.5]{HJ20}):
\[
H_{\mathrm{dR}}(\eM,\bar{k})\assign
\Hom_{\bar{k}}(\eM/(t-\theta)\eM, \bar{k}).
\]
Elements of $H_{\mathrm{dR}}(\eM,\bar{k})$ are called \emph{differentials of $\eM$ over $\bar{k}$}.
When it is clear from the context, we also identify $H_{\mathrm{dR}}(\eM,\bar{k})$ with the space of $\bar{k}$-linear functionals on $\eM$ which factor through $\eM/(t-\theta) \eM$.
Given $\omega \in H_{\mathrm{dR}}(\eM,\bar{k})$,
we naturally extend $\omega$ to a $\CC_{\infty}$-linear functional on $\MM^\dagger = \TT^\dagger \otimes_{\bar{k}[t]}\eM$ by setting
\[
\omega(f\otimes m)\assign  f(\theta)\cdot \omega(m), \quad \forall f\otimes m \in \MM^\dagger.
\]
It follows that $\omega((t-\theta) \MM^\dagger) = 0$, i.e.,\ $\omega$ induces a linear functional $\omega : \MM^\dagger/(t-\theta)\MM^\dagger \rightarrow \CC_\infty$.
Suppose now that $\eM$ is uniformizable, and recall from Remark~\ref{rem: HB(M)} that $H_{\mathrm{Betti}}(\eM)\subset \MM^\dagger$ is an $\FF_{q}[t]$-submodule.  The \textit{de Rham pairing} of $\eM$ is defined as follows:
\begin{equation}\label{E:deRham pairing}
H_{\mathrm{dR}}(\eM,\bar{k}) \times H_{\mathrm{Betti}}(\eM)
\longrightarrow \CC_\infty, \quad (\omega,\gamma) \longmapsto \int_{\gamma} \omega \assign  \omega(\gamma),
\end{equation}
and every period integral $\int_{\gamma} \omega$ is called a \textit{period of $\eM$}.

\begin{lemma}\label{lem: dR}
Let $\eM$ be a uniformizable dual $t$-motive over $\bar{k}$.
The de Rham pairing of $\eM$ in~\eqref{E:deRham pairing} is non-degenerate.
Consequently,
put $\eM_{\CC_\infty}\assign  \CC_\infty\otimes_{\bar{k}} \eM$ and let the de Rham module of $\eM$ over $\CC_\infty$ be
\[
H_{\mathrm{dR}}(\eM,\CC_\infty) \assign  \Hom_{\CC_\infty}(\eM_{\CC_\infty}/(t-\theta)\eM_{\CC_\infty}, \CC_\infty) \cong \CC_{\infty}\otimes_{\ok}H_{\mathrm{dR}}(\eM,\bar{k}).
\]
Then the de Rham pairing of $\eM$ induces  the following comparison isomorphism
\[
\Hom_{\CC_\infty}\bigl(\CC_{\infty}\otimes_{\FF_{q}[t]} H_{\mathrm{Betti}}(\eM), \CC_\infty\bigr) \stackrel{\sim}{\longrightarrow} H_{\mathrm{dR}}(\eM,\CC_\infty).
\]
\end{lemma}

\begin{proof}
We first embed $H_{\mathrm{dR}}(\eM,\bar{k})$ into
\[
\CC_\infty\otimes_{\bar{k}}H_{\mathrm{dR}}(\eM,\bar{k}) \cong H_{\mathrm{dR}}(\eM,\CC_\infty) \cong \Hom_{\CC_\infty}(\MM^\dagger/(t-\theta)\MM^\dagger,\CC_\infty).
\]
Let $\{\gamma_1,\ldots, \gamma_r\}$ be an $\FF_q[t]$-basis of $H_{\mathrm{Betti}}(\eM)$, and let $v_i$ be the image of $\gamma_i$ in the quotient module $\MM^\dagger/(t-\theta)\MM^\dagger$ for $i=1,\ldots, r$.
Then $\{v_1,\ldots, v_r\}$ forms a $\CC_\infty$-basis of $\MM^\dagger/(t-\theta)\MM^\dagger$ (as $\eM$ is uniformizable).

Given $\omega \in H_{\mathrm{dR}}(\eM,\bar{k})$ with $\int_\gamma \omega = 0$ for every $\gamma \in H_{\mathrm{Betti}}(\eM)$, by extending $\omega$ to a $\CC_\infty$-linear functional on $\MM^\dagger/(t-\theta)\MM^\dagger$ we have that
\[
\omega(v_i) = \int_{\gamma_i} \omega = 0, \quad \text{ for $i = 1,\ldots, r$.}
\]
It follows that $\omega$ is the zero map on $\MM^\dagger/(t-\theta)\MM^\dagger$ and so $\omega = 0$.

On the other hand, given $\gamma \in H_{\mathrm{Betti}}(\eM)$, we may choose our basis $\{\gamma_1,\ldots, \gamma_r\}$ of $H_{\mathrm{Betti}}(\eM)$ suitably so that
\[
\gamma = f(t) \cdot \gamma_1, \quad \textup{for some $f(t) \in \FF_q[t]$.}
\]
Now suppose
\[
0 = \int_\gamma \omega \quad  (= f(\theta) \cdot \omega(v_1)), \quad \forall\, \omega \in H_{\rm {dR}}(\eM,\bar{k}) .
\]
As $v_1$ is nonzero in $\MM^\dagger/(t-\theta)\MM^\dagger$ and elements in $H_{\mathrm{dR}}(\eM,\bar{k})$ span the entire $\CC_\infty$-vector space $\Hom_{\CC_\infty}(\MM^\dagger/(t-\theta)\MM^\dagger,\CC_\infty)$,
we must have that $f(\theta) = 0$, which implies $f(t) = 0$. Therefore $\gamma = 0$ as desired. In conclusion, the de Rham pairing of $\eM$ is non-degenerate.
\end{proof}

\begin{Subsubsec}{Push-forward and pull-back} \label{sec: Pf-Pb}
Let $f:\eM_1\rightarrow \eM_2$ be a morphism between dual $t$-motives. Extending $f$ to a homomorphism (still denoted by $f$)  $\MM_1^\dagger \rightarrow \MM_2^\dagger$
induces the $\FF_q(t)$-linear \emph{push-forward} homomorphism on $H_{\mathrm{Betti}}(\eM_1)$:
\[
f_* :H_{\mathrm{Betti}}(\eM_1) \rightarrow H_{\mathrm{Betti}}(\eM_2), \quad
f_* \gamma \assign  f(\gamma), \quad \forall\, \gamma \in H_{\mathrm{Betti}}(\eM_1).
\]
On the other hand $f$ also induces the $\bar{k}$-linear \emph{pull-back} homomorphism on $H_{\mathrm{dR}}(\eM_2)$:
\[
f^*: H_{\mathrm{dR}}(\eM_2,\ok) \longrightarrow H_{\mathrm{dR}}(\eM_1,\ok), \quad \omega \longmapsto f^* \omega \assign  \omega \circ f.
\]
The following lemma is straightforward.
\end{Subsubsec}

\begin{lemma}\label{lem: adjoint}
Let $f:\eM_1\rightarrow \eM_2$ be a morphism of dual $t$-motives over $\ok$. Given $\gamma_1 \in H_{\mathrm{Betti}}(\eM_1)$ and $\omega_2 \in H_{\mathrm{dR}}(\eM_2,\bar{k})$, we have
\[
\int_{f_* \gamma_1} \omega_2 = \int_{\gamma_1} f^* \omega_2.
\]
\end{lemma}

\subsection{Definition of period symbols}\label{sec: defn PS}
To define period symbols in our setting, we first need the following proposition.

\begin{proposition}\label{Prop: delta_M}
Let $\eM$ be a CM dual $t$-motive over $\bar{k}$ with generalized CM type $(\eK,\Xi)$.
For each $\xi \in J_\eK$, there exists a nonzero differential $\omega_{\eM,\xi} \in H_{\mathrm{dR}}(\eM,\bar{k})$, unique up to a $\bar{k}^{\times}$-multiple, satisfying
\[
\omega_{\eM,\xi} (\alpha \cdot m) = \nu_\xi(\alpha) \cdot \omega_{\eM,\xi}(m), \quad \text{$\forall\, m \in \eM$ and $\alpha \in O_\eK$.}
\]
\end{proposition}

We call $\omega_{\eM,\xi}$ the \emph{differential of $\eM$ associated with $\xi$}.

\begin{proof}
For each $\xi \in J_\eK$, let $\Pfk_\xi$ be the maximal ideal of $O_\bK$ associated to $\xi$,
i.e.\ $\Pfk_\xi$ is generated by $\alpha - \nu_\xi(\alpha)$ for all $\alpha \in O_\eK$. We consider the following isomorphisms
\[
\eM/(t-\theta)\eM\cong \frac{O_{\bK}}{(t-\theta)O_{\bK}}\otimes_{O_{\bK}} \eM \cong \frac{O_{\bK}}{\prod_{\xi\in J_{K}}\Pfk_{\xi}} \otimes_{O_{\bK}}\eM \cong \bigoplus_{\xi\in J_{K}} \left( \frac{O_{\bK}}{\Pfk_{\xi}} \otimes_{O_{\bK}}\eM\right),
\]
and note in particular that the composite map above sends $\alpha\cdot \bar{m}$ to $(\nu_{\xi}(\alpha) \otimes m)_{\xi\in J_{\eK}}$ for $\alpha\in O_{\eK}$ and $m\in \eM$.

Note further that for $\omega \in H_{\mathrm{dR}}(\eM,\bar{k})$ the condition
\[
\omega (\alpha \cdot m) = \nu_\xi(\alpha) \cdot \omega(m), \quad \textup{$\forall\, m \in \eM$ and $\alpha \in O_\eK$,}
\]
is equivalent to the identity
\[
\omega\bigl((\alpha-\nu_\xi (\alpha)) \cdot m\bigr) = 0, \quad \textup{$\forall\, m \in \eM$ and $\alpha \in O_\eK$,}
\]
which is equivalent to saying that $\omega$ factors through $\frac{O_{\bK}}{\Pfk_{\xi}} \otimes_{O_{\bK}}\eM \cong \eM/\Pfk_\xi \eM$.  As $\eM$ is a projective $O_{\bK}$-module of rank one, $\eM/(t-\theta)\eM$ is a free $O_{\bK}/(t-\theta)O_{\bK}$-module of rank one. Since
\[
 \frac{O_\bK}{\Pfk_\xi} \cong \bar{k}, \quad \forall\, \xi \in J_\eK,
\]
we see that
$\Hom_{\bar{k}}(\eM/\Pfk_\xi \eM, \bar{k})$ is one dimensional over $\bar{k}$, and hence the desired  result follows.
\end{proof}

Recall in Remark~\ref{rem: ps} (1) that for $x$, $y \in \CC_{\infty}^{\times}$, we let $x \sim y$ denote that $x/y \in \bar{k}^\times$.
Let $\eM$ be a CM dual $t$-motive with the generalized CM type $(\eK,\Xi)$ over $\bar{k}$ and  fix any $\xi \in J_\eK$.
As $H_{\mathrm{Betti}}(\eM)$ is a projective $O_\eK$-module of rank one, we have that
\[
\int_{\gamma_1} \omega_{\eM,\xi} \sim \int_{\gamma_2} \omega_{\eM,\xi} \quad \text{ for nonzero cycles $\gamma_1,\gamma_2 \in H_{\mathrm{Betti}}(\eM)$}.
\]
Moreover, let $\eM'$ be another CM dual $t$-motive over $\bar{k}$ with the same generalized CM type $(\eK,\Xi)$. By Lemma~\ref{lem: adjoint}, the isogeny between $\eM$ and $\eM'$ over $\bar{k}$ leads to
\[
\int_{\gamma}\omega_{\eM,\xi} \sim \int_{\gamma'} \omega_{\eM',\xi},  \quad \forall \gamma \in H_{\mathrm{Betti}}(\eM) \text{ and } \gamma' \in H_{\mathrm{Betti}}(\eM').
\]
It follows that up to algebraic multiples, the period integral $\int_\gamma \omega_{\eM,\xi}$ depends only on the generalized CM type $(\eK,\Xi)$ and the chosen $\xi \in J_\ek$.
This leads to the following definition.

\begin{definition} \label{defn: period quantity}
Let $\eK$ be a CM field over $\FF_q(t)$ and $\Xi$ be a generalized CM type of~$\eK$. Given $\xi \in J_\eK$, take a CM dual $t$-motive $\eM$ with generalized CM type $(\eK,\Xi)$ over $\bar{k}$ and the differential $\omega_{\eM,\xi}$ of $\eM$ associated with $\xi$ (given in Proposition~\ref{Prop: delta_M}).
We define the following \emph{period symbol}:
\[
\Pcal_\eK(\xi,\Xi)\assign \left(\int_{\gamma} \omega_{\eM,\xi}\right)\cdot \bar{k}^\times \in \CC_\infty^\times/\bar{k}^\times.
\]
As mentioned in Remark~\ref{rem: ps} (1), we use the notation $p_\eK(\xi,\Xi)$ for an arbitrary representative of the coset $\Pcal_\eK(\xi,\Xi)$, and call it a period symbol as well if there is no risk of confusion.
\end{definition}

Recall the fundamental period $\tilde{\pi}$ of the Carlitz $\FF_{q}[t]$-module $\mathbf{C}$. When $\eK = \FF_q(t)$, we note that every generalized CM type must be of the form $\Xi= n\xi_\theta$ for some positive integer $n$, where $\xi_\theta$ is the point $t=\theta$ on $\PP^1_{/\bar{k}}$. In the following proposition, we calculate the period symbol in this case.

\begin{proposition}\label{prop: carlitz-ps}
Let $n$ be a positive integer and let $\Xi\assign n\xi_{\theta}$. Then we have that
\[
p_{\FF_q(t)}(\xi_\theta,\Xi) \sim \tilde{\pi}^{n}.
\]
\end{proposition}

\begin{proof}
To calculate the period symbol $p_{\FF_q(t)}(\xi_\theta,\Xi)$, we note that in this case  $\eI_{\Xi}=n\cdot \infty$, and we can take $h\assign (t-\theta)^{n}$ and $W=0$ in Theorem~\ref{thm: CM dual}. So the associated dual $t$-motive is given by $\eC^{\otimes n}=\Gamma\left(\bU, \mathcal{O}_{\bX} \right)=\ok[t]$, on which the $\sigma$-action is given by
\[
\sigma \cdot x=(t-\theta)^{n}x^{(-1)}, \quad x\in \eC^{\otimes n}.
\]
We note that $\eC^{\otimes n}$ is the dual $t$-motive associated to the $n$-th tensor power ${\mathbf{C}}^{\otimes n}$ of the Carlitz module (see~\cite[p.~428]{CPY19}).

Define
\[
\Omega(t)\assign (-\theta)^{\frac{-q}{q-1}} \prod_{i=1}^{\infty} \biggl(
1-\frac{t}{\theta^{q^{i}}} \biggr)\in \power{\CC_{\infty}}{t},
\]
where $(-\theta)^{\frac{1}{q-1}}$ is a suitable choice of $(q-1)$-st root of $-\theta$ so that $\frac{1}{\Omega(\theta)}=\tilde{\pi}$. We note that $\Omega^{n}$ satisfies the functional equation $(\Omega^{n})^{(-1)}=(t-\theta)^{n}\Omega^{n}$, and so $1/\Omega^{n}$ constitutes an $\FF_{q}[t]$-basis of the Betti-module $H_{\mathrm{Betti}}(\eC^{\otimes n})$. It follows that
\[
p_{\FF_q(t)}(\xi_\theta,\Xi) \sim \int_{\Omega^{-n}}\omega_{\eC^{\otimes n},\xi_{\theta}}\sim\Omega^{-n}(\theta)=\tilde{\pi}^{n}.
\qedhere
\]
\end{proof}

\subsection{Fundamental properties of period symbols}\label{sec: FP PS}

In the following lemma we show that our period symbols satisfy the \emph{additive} property.

\begin{lemma}\label{lem: tensor diff}
Let $\eK/\FF_{q}(t)$ be a CM field and $\Xi_1,\ldots, \Xi_n$ be generalized CM types of $\eK$. Let $\Xi= \Xi_1+\cdots +\Xi_n$. For each $\xi \in J_\eK$, we have that
\[
\prod_{i=1}^n\Pcal_\eK(\xi,\Xi_i) = \Pcal_{\eK}(\xi, \Xi) \quad \in \CC_\infty^\times/\bar{k}^\times.
\]
\end{lemma}

\begin{proof}
Let $\eM_1,\ldots, \eM_n$ be CM dual $t$-motives over $\bar{k}$ with the generalized CM types $(\eK,\Xi_1),\ldots, (\eK,\Xi_n)$, respectively.
By Proposition~\ref{prop: tensor CM type},
\[
\eM\assign  \eM_1 \otimes_{O_{\bK}} \cdots \otimes_{O_{\bK}} \eM_n
\]
is a CM dual $t$-motive with generalized CM type $(\eK,\Xi)$.   Moreover, for $1\leq i \leq n$, let $\omega_{\eM_i,\xi}$ be the differential of $\eM_i$ associated with $\xi$. We then set
\[
\omega_{\eM,\xi} = \omega_{\eM_1,\xi} \otimes \cdots \otimes \omega_{\eM_n,\xi}\assign m_{1}\otimes \cdots\otimes m_{n}\mapsto \omega_{\eM_1,\xi}(m_{1})\cdot \cdots \cdot \omega_{\eM_n,\xi}(m_{n}),
\]
which is a well-defined (from the universal property of tensor products) differential of $\eM$ associated with $\xi$. Notice that
\[
H_{\mathrm{Betti}}(\eM) = H_{\mathrm{Betti}}(\eM_1) \otimes_{O_\eK} \cdots \otimes_{O_\eK} H_{\mathrm{Betti}}(\eM_n).
\]
Therefore taking nonzero cycles $\gamma_1,\ldots, \gamma_n$ in $H_{\mathrm{Betti}}(\eM_1),\ldots, H_{\mathrm{Betti}}(\eM_n)$, respectively, and choosing
$\gamma = \gamma_1\otimes \cdots \otimes \gamma_n \in H_{\mathrm{Betti}}(\eM)$,
we obtain that
\[
\prod_{i=1}^n p_\eK(\xi,\Xi_i) \sim \prod_{i=1}^n \int_{\gamma_i} \omega_{\eM_i,\xi}
= \int_{\gamma} \omega_{\eM,\xi} \sim p_\eK(\xi,\Xi).
\qedhere
\]
\end{proof}

By applying Lemma~\ref{lem: tensor diff}, we now extend the period symbols to a unique bi-additive pairing.

\begin{theorem}\label{thm: Period Symbol}
For each CM field $\eK/\FF_{q}(t)$, there exists a unique bi-additive pairing $$\Pcal_\eK: I_\eK \times I_\eK^0 \rightarrow \CC_\infty^\times/\bar{k}^\times$$ satisfying:
\begin{itemize}
    \item[(1)] $\Pcal_\eK(\xi,\Xi)$ is the period symbol in Definition~\ref{defn: period quantity} for each point $\xi \in J_\eK$ and each generalized CM type $\Xi$ of $\eK$.
    \item[(2)] Let $\eK'/\FF_{q}(t)$ be a CM field containing $\eK$. For $\Phi' \in I_{\eK'}$ and $\Phi^0 \in I_\eK^0$,
    \[
    \Pcal_{\eK}\bigl(\mathrm{Res}_{\eK'/\eK}(\Phi'),\Phi^0\bigr) = \Pcal_{\eK'}\bigl(\Phi',\mathrm{Inf}_{\eK'/\eK}(\Phi^0)\bigr).
    \]
    \item[(3)] Let $\eK'/\FF_{q}(t)$ be a CM field containing $\eK$, and let $\eK'^{+}$ be the maximal totally real subfield of $\eK'$.
    For $\Phi \in I_\eK$ and $\Phi^{\prime,0} \in I_{\eK'}^0$,  we have
    \[
    \Pcal_{\eK'}\bigl(\mathrm{Inf}_{\eK'/\eK}(\Phi),\Phi^{\prime,0} \bigr) = \Pcal_{\eK}\bigl(\Phi,\mathrm{Res}_{\eK'/\eK}(\Phi^{\prime,0})\bigr).
    \]
    \item[(4)] Let $\eK'/\FF_{q}(t)$ be a CM field together with an isomorphism $\varrho: \eK'\rightarrow \eK$ over $\FF_q(t)$, and let $\pi_\varrho : X \rightarrow X'$ be the isomorphism of $\FF_q$-schemes corresponding to $\varrho$.
    Then
    \[
    \Pcal_{\eK'}\bigl((\pi_\varrho)_*(\Phi),(\pi_\varrho)_*(\Phi^0)\bigr) = \Pcal_\eK(\Phi,\Phi^0), \quad \forall\, (\Phi,\Phi^0) \in I_\eK\times I_\eK^0.
    \]
    \item[(5)] Let $\eK^+$ be the maximal totally real subfield of $\eK$.
    Given $\xi^+ \in J_{\eK^+}$ and $\Phi^0 \in I_\eK^0$, one has
    \[
    \tilde{\pi}^{\text{\rm wt}(\Phi^0)} \sim \prod_{\substack{\xi \in J_\eK \\ \pi_{\bX/\bX^+}(\xi) = \xi^+}}
    p_\eK(\xi,\Phi^0).
    \]
\end{itemize}
\end{theorem}

Our strategy matches with Deligne's table in \cite[p.~29]{De80} for the classical case. We include the detailed proofs of each property individually as follows.

\begin{proof}[Proof of Theorem~\ref{thm: Period Symbol} (1)]
As $I_\eK$ is the free abelian group generated by $\xi \in J_\eK$ and $I_\eK^0$ is the subgroup of $I_\eK$ generated by all the generalized CM types of $\eK$ (see Proposition~\ref{prop: CM prop}), the pairing $\Pcal_\eK$ is uniquely determined by property $(1)$. It suffices to prove that it is well-defined.

Given generalized CM types $\Xi_1,\ldots, \Xi_n$ of $\eK$, suppose
\[
m_1\Xi_1+\cdots +m_n \Xi_n = 0, \quad m_1,\ldots, m_n \in \ZZ.
\]
Without loss of generality, assume that $m_1,\ldots, m_\ell >0$ and $m_{\ell+1},\ldots, m_n <0$.
Then
\[
\Xi\assign  m_1\Xi_1+\cdots + m_\ell \Xi_\ell  = (-m_{\ell+1})\Xi_{\ell+1}+\cdots + (-m_n)\Xi_n
\]
is also a generalized CM type of $\eK$. For $\xi \in I_\eK$, by Lemma~\ref{lem: tensor diff} we have
\[
\prod_{i=1}^\ell \Pcal_\eK(\xi,\Xi_i)^{m_i} = \Pcal_\eK(\xi,\Xi) = \prod_{j=\ell+1}^n \Pcal_\eK(\xi,\Xi_j)^{-m_j},
\]
which means that
\[
\prod_{i=1}^n \Pcal_\eK(\xi,\Xi_i)^{m_i} = 1 \quad \in \CC_\infty^\times/\bar{k}^\times.
\]
Thus $\Pcal_\eK$ is well-defined.
\end{proof}

\begin{proof}[Proof of Theorem~\ref{thm: Period Symbol} (2)]
According to the additive properties of $\mathrm{Inf}_{\eK'/\eK}$ and $\mathrm{Res}_{\eK'/\eK}$, we may assume that $\Phi' = \xi' \in J_{\eK'}$ and $\Phi^0 = \Xi \in I_\eK^0$ is a generalized CM type of $\eK$.
Let $\eM$ be a CM dual $t$-motive with the generalized CM type $(\eK,\Xi)$.
Put
\[
\eM' \assign O_{\eK'} \otimes_{O_\eK} \eM \quad (\cong O_{\bK'}\otimes_{O_\bK} \eM),
\]
whose $\bar{k}[t,\sigma]$-module structure is induced from $\eM$.
As $\eM$ is projective of rank one over $O_{\bK}$, we see that $\eM'$ is projective of rank one over $O_{\bK'}$.
Moreover, we get that
\[
\sigma \eM' = O_{\eK'}\otimes_{O_\eK} \sigma \eM
= O_{\bK'} \otimes_{O_\bK} \Ifk_\Xi \eM
= \Ifk'_{\Xi'} \eM',
\]
where $\Xi' = \Inf_{\eK'/\eK}(\Xi)$ and $\Ifk'_{\Xi'}$ is the ideal of $O_{\bK'}$ corresponding to $\Xi'$.
Finally, the natural surjective morphism
\[
\eM^{\oplus [\eK':\FF_q(t)]} \cong O_{\eK'} \otimes_{\FF_q[t]} \eM \twoheadrightarrow \eM',
\]
ensures the purity of $\eM'$ by \cite[Proposition 2.4.10 (c)]{HJ20}.
Therefore Proposition~\ref{prop: being CM-0} implies that $\eM'$ is a CM dual $t$-motive with generalized CM type $(\eK',\Inf_{\eK'/\eK}(\Xi))$,
and we have a natural injective morphism $f: \eM \rightarrow \eM'$ of dual $t$-motives which is $O_\bK$-linear. In particular, the push-forward $f_*: H_{\mathrm{Betti}}(\eM) \rightarrow H_{\mathrm{Betti}}(\eM')$ is also injective.

On the other hand, for a differential $\omega_{\eM',\xi'} \in H_{\mathrm{dR}}(\eM',\bar{k})$ associated with $\xi'$, we note that its pull-back $f^*\omega_{\eM',\xi'} \in H_{\mathrm{dR}}(\eM,\bar{k})$ is a differential associated with $\xi\assign  \pi_{\bX'/\bX}(\xi')= \text{Res}_{\eK'/\eK}(\xi')$.
Hence for a nonzero $\gamma \in H_{\mathrm{Betti}}(\eM)$ we have
\[
p_{\eK'}(\xi',\text{Inf}_{\eK'/\eK}(\Xi))
\sim \int_{f_* \gamma} \omega_{
\eM',\xi'} \sim \int_{\gamma} f^*\omega_{M',\xi'}
\sim
p_\eK(\text{Res}_{\eK'/\eK}(\xi'),\Xi).
\]
Therefore property (2) holds.
\end{proof}

\begin{proof}[Proof of Theorem~\ref{thm: Period Symbol} (3)]
We may assume that $\Phi = \xi \in J_\eK$ and $\Phi^{\prime,0} = \Xi'  \in I_{\eK'}^0$ is a generalized CM type of $\eK'$.
Put $n = [\eK':\eK]$.
Let $\eM'$ be a CM dual $t$-motive with generalized CM type $(\eK',\Xi')$ over $\bar{k}$.
Take
\[
\eM \assign \bigwedge^{n}_{O_\bK} \eM', \quad \text{the $n$-th exterior product of $\eM'$ over $O_\bK$,}
\]
which is equipped with a $\sigma$-action defined by
\[
\sigma \cdot (m_1\wedge \cdots \wedge m_{n}) \assign
(\sigma m_1) \wedge  \cdots \wedge (\sigma m_{n}).
\]
Put $\Xi \assign \Res_{\eK'/\eK}(\Xi')$.
Let $\Ifk'_{\Xi'}$ (resp.\ $\Ifk_\Xi$) be the ideal of $O_{\bK'}$ (resp.\ $O_\bK$) corresponding to $\Xi'$ (resp.\ $\Xi$).
Then $\eM$ is a projective $O_\bK$-module of rank one, and
\[
\sigma \eM = \bigwedge^{n}_{O_\bK} \sigma \eM' = \bigwedge^{n}_{O_\bK} \Ifk_{\Xi'}' \eM' = \Ifk_{\Xi} \bigwedge^{n}_{O_\bK} \eM' = \Ifk_{\Xi} \eM.
\]
Moreover, the surjectivity of the morphisms
\[
\bigotimes_{\ok[t]}^{n} \eM' \twoheadrightarrow \bigotimes_{O_{\bK}}^{n} \eM' \twoheadrightarrow  \bigwedge^{n}_{O_\bK} \eM' = \eM
\]
ensures the purity of $\eM$
by \cite[Proposition 2.4.10 (c)]{HJ20}.
Therefore Proposition~\ref{prop: being CM-0} implies that $\eM'$ is a CM dual $t$-motive with generalized CM type $(\eK,\Res_{\eK'/\eK}(\Xi'))$.

On the other hand,
write $\Inf_{\eK'/\eK}(\xi) = \xi_1'+\cdots + \xi_n'$.
Let $\Pfk'_{\xi_i'}$ be the maximal ideal of $O_{\bK'}$ corresponding to $\xi_i'$ for $i=1,...,n$, and $\Pfk_{\xi}$ be the maximal ideal of $O_{\bK}$ corresponding to $\xi$.
One has $\Pfk_{\xi} \cdot O_{\bK'} = \prod_{i=1}^n \Pfk'_{\xi_i'}$,
and the quotient map $\eM' \twoheadrightarrow \eM'/\Pfk_\xi \eM'$ induces
\[
\frac{\eM}{\Pfk_\xi \eM} \cong \bigwedge^n_{O_{\bK}} \frac{\eM'}{\Pfk_{\xi}\eM'} \cong \bigwedge^n_{\ok} \frac{\eM'}{\Pfk_{\xi}\eM'}.
\]
For $1\leq i \leq n$, let $\omega_{\eM',\xi_i'} \in H_{\rm dR}(\eM',\ok)$ be the differential of $\eM'$ associated with $\xi_i'$.
Then $\omega_{\eM',\xi_1'},...,\omega_{\eM',\xi_n'}$ form a $\bar{k}$-basis of $\Hom_{\ok}(\eM'/\Pfk_\xi \eM',\bar{k})$, which implies that
\[
\omega_{\eM,\xi} \assign \omega_{\eM',\xi_1'}\wedge \cdots \wedge \omega_{\eM',\xi_n'}: \bigwedge^n_{\ok} \frac{\eM'}{\Pfk_{\xi}\eM'} \stackrel{\sim}{\longrightarrow} \ok
\]
can be viewed as a differential of $\eM$ associated with $\xi$.

Finally,
put $\OO_\bK^\dagger \assign \TT^\dagger \otimes_{\ok[t]} O_\bK$.
Observe that
\[
\TT^\dagger \otimes_{\bar{k}[t]} \eM = \TT^\dagger \otimes_{\bar{k}[t]} \bigwedge^n_{O_{\bK}} \eM'
\cong \bigwedge^n_{\OO_\bK^\dagger} (\TT^\dagger \otimes_{\ok[t]} \eM'),
\]
whence we may regard $\bigwedge\limits^n_{O_\eK}H_{\rm Betti}(\eM')$ as an $O_\eK$-submodule of $H_{\rm Betti}(\eM)$.
As $H_{\rm Betti}(\eM')$ is a projective $O_{\eK'}$-module of rank one (i.e.\ a projective $O_\eK$-module of rank $n$),
we can find $\gamma_1',...,\gamma_n' \in H_{\rm Betti}(\eM')$ so that
\[
0 \neq \gamma \assign \gamma_1' \wedge \cdots \wedge \gamma_n' \in \bigwedge^n_{O_\eK}H_{\rm Betti}(\eM') \subseteq H_{\rm Betti}(\eM).
\]
Since $\int_{\gamma_i'} \omega_{\eM',\xi_j'} \sim p_{\eK'}(\xi_{j}', \Xi')$ for all $1\leq i,j \leq n$, it follows that
\begin{align*}
p_{\eK}(\xi,\Res_{\eK'/\eK}(\Xi'))
\sim \int_{\gamma}\omega_{\eM,\xi}
& = \det\big(\int_{\gamma_i'} \omega_{\eM',\xi_j'}\big)_{1\leq i,j\leq n} \\
& \sim \prod_{j=1}^n p_{\eK'}(\xi_j',\Xi')
\sim p_{\eK'}(\Inf_{\eK'/\eK}(\xi),\Xi'),
\end{align*}
as desired.
\end{proof}

\begin{proof}[Proof of Theorem~\ref{thm: Period Symbol} (4)]
Taking~(1) above and Proposition~\ref{prop: CM prop} into account, we may assume that $\Phi = \xi \in J_\eK$ and $\Phi^0 = \Xi$ is a CM type of $\eK$.
Note that $\varrho$ induces a $\bar{k}(t)$-isomorphism (still denoted by $\varrho$) from $\bK'\assign \bar{k}(t)\otimes_{\FF_q(t)} \eK'$ to $\bK\assign \bar{k}(t)\otimes_{\FF_q(t)}\eK$.
Let $\eM= \eM_{(W,h)}$ be a CM dual $t$-motive with the generalized CM type $(\eK,\Xi)$ over $\bar{k}$.
Let $\pi_\varrho : \bX \stackrel{\sim}{\rightarrow} \bX'$ and ${\pi_{\varrho}}_{*}:\Div (\bX)\stackrel{\sim}{\rightarrow} \Div (\bX')$ be the isomorphisms induced by $\varrho$.
Then
\[
\varrho^{-1}\Bigl(\eM_{(W,h)}\Bigr) = \eM_{({\pi_\varrho}_*W,\varrho^{-1}(h))} \rassign \eM'
\]
is a CM dual $t$-motive with generalized CM type $(\eK',{\pi_\varrho}_{*}\Xi)$ over $\bar{k}$.
In other words, $f\assign \varrho|_{\eM'}: \eM' \rightarrow \eM$ is an isomorphism of dual $t$-motives.
Moreover, let $\omega_{\eM,\xi}$ be the differential in $H_{\mathrm{dR}}(\eM,\bar{k})$ associated with $\xi$.
Then $f^* \omega_{\eM,\xi}$ is the differential in $H_{\mathrm{dR}}(\eM',\bar{k})$ associated with $\xi'\assign  \pi_{\varrho}(\xi)$.
For a nonzero cycle $\gamma \in H_{\mathrm{Betti}}(\eM')$, we obtain that
\[
p_{\eK'}\bigl(\pi_\varrho(\xi), {\pi_\varrho}_*\Xi\bigr) \sim \int_{\gamma} f^* \omega_{\eM,\xi} \sim \int_{f_* \gamma} \omega_{\eM,\xi}
\sim p_\eK(\xi,\Xi),
\]
whence proving (4).
\end{proof}

\begin{proof}[Proof of Theorem~\ref{thm: Period Symbol} (5)]
Since
\begin{align*}
\sum_{\substack{\xi \in J_\eK \\ \pi_{\bX/\bX^+}(\xi) = \xi^+}} \xi = \Inf_{\eK/\eK^+}(\xi^+)
\quad \text{ and } \quad
\Res_{\eK/\eK^+}(\Phi^0) &= \mathrm{wt}(\Phi^0) \sum_{\xi_0^+ \in J_{\eK^+}} \xi_0^+ \\
& = \mathrm{wt}(\Phi^0) \Inf_{\eK^+/\FF_q(t)}(\xi_\theta),
\end{align*}
we get
\begin{align*}
\prod_{\substack{\xi \in J_\eK,~\pi_{\bX/\bX^+}(\xi) = \xi^+}}
p_\eK(\xi,\Phi^0)
&\sim p_\eK\big(\Inf_{\eK/\eK^+}(\xi^+),\Phi^0\big) \\
&\sim p_{\eK^+}\big(\xi^+,\Res_{\eK/\eK^+}(\Phi^0)\big) \hspace{2.6cm} \text{(by (3))}\\
&\sim p_{\eK^+}\big(\xi^+,\Inf_{\eK^+/\FF_q(t)}(\xi_\theta)\big)^{\mathrm{wt}(\Phi^0)} \\
&\sim p_{\FF_q(t)}\big(\Res_{\eK^+/\FF_q(t)}(\xi^+),\xi_\theta\big)^{\mathrm{wt}(\Phi^0)} \hspace{1cm}\text{(by (4))}\\
&\sim \tilde{\pi}^{\mathrm{wt}(\Phi^0)} \hspace{2.7cm} \text{(by Proposition~\ref{prop: carlitz-ps}).}
\end{align*}
\end{proof}

\begin{Subsubsec}{Extension of period symbols}\label{sec: EPS}
We extend $\Pcal_\eK$ to a pairing on $I_\eK\times I_\eK$ as follows:
for $\xi_1,\xi_2 \in J_\eK$, we let
\[
\Phi^0_2\assign  [\eK:\eK^+]\cdot \xi_2 - \sum_{\substack{\xi \in J_\eK \\ \pi_{\bX/\bX^+}(\xi) =\pi_{\bX/\bX^+}(\xi_2)}}\xi \in I_\eK^0,
\]
and define
\begin{equation}\label{E:Def of P}
\Pcal_\eK(\xi_1,\xi_2)\assign  \tilde{\pi}^{1/[\eK:\FF_q(t)]} \cdot \Pcal_\eK(\xi_1,\Phi_2^0)^{1/[\eK:\eK^+]}  \quad \in \CC_\infty^\times/\bar{k}^\times,
\end{equation}
Extending $\Pcal_\eK$ bi-additively to $I_\eK\times I_\eK$, we obtain an analogue of Shimura's period symbols for each CM field $\eK$ satisfying the properties (1)--(5) in Theorem~\ref{thm: Period Symbol}.
For every $\Phi_1,\Phi_2 \in I_\eK$, we also keep using the notation $p_\eK(\Phi_1,\Phi_2)$ for an arbitrary representative of the coset $\Pcal_\eK(\Phi_1,\Phi_2)$.
\end{Subsubsec}

\section{Algebraic relations among period symbols}

In this section, we determine all the algebraic relations among the period symbols defined in~\eqref{E:Def of P}, and as such prove our analogue of Shimura's conjecture in this setting.

\subsection{Analogue of Shimura's conjecture}\label{sec: A-SC}
Let $\eK$ be a CM field over $\FF_q(t)$ and $\eK^+$ be the maximal totally real subfield of $\eK$. We start with the following ``Legendre relation.''

\begin{lemma}\label{L:Algebraic relation}
Fix any $\xi_{0}\in J_{\eK}$. For each $\xi^+ \in J_{\eK^+}$, we have
\begin{equation}\label{eqn: Algebraic Relation}
\prod_{\substack{\xi \in J_\eK \\ \pi_{\bX/\bX^+}(\xi) = \xi^+}}
p_\eK (\xi,\xi_0) \sim \tilde{\pi}^{1/[\eK^+:\FF_q(t)]}.
\end{equation}
\end{lemma}

\begin{proof}
This relation follows directly from the definition in \eqref{E:Def of P} and Theorem~\ref{thm: Period Symbol}~(5).
\end{proof}

The main theorem of this paper is the following analogue of Shimura's conjecture (the proof will be given in Section~\ref{sec: Proof MT}).

\begin{theorem}\label{T:Shimura's conjecture}
Let $\eK$ be a CM field over $\FF_q(t)$ and $\eK^+$ its maximal totally real subfield. For each $\xi_0 \in J_\eK$, all algebraic relations among the period symbols $p_\eK(\xi,\xi_0)$ for $\xi \in J_\eK$ are generated by~\eqref{eqn: Algebraic Relation}. Equivalently, we have
\[
\trdeg_{\bar{k}}\, \bar{k}(p_\eK(\xi,\xi_0)\mid \xi \in J_\eK) = 1 + \frac{([\eK:\eK^+]-1)}{[\eK:\eK^+]} \cdot [\eK:\FF_q(t)].
\]
\end{theorem}

As a consequence, Theorem~\ref{T:Shimura's conjecture} leads to the following result that matches with the spirit of Conjectures~\ref{C:Shimura} and~\ref{C:Yoshida}.

\begin{corollary}\label{thm: Shimura}
Let $\eK$ be a CM field over $\FF_q(t)$ and $\eK^+$ be the maximal totally real subfield of $\eK$.
\begin{itemize}
    \item[(a)] Suppose that $\eK\neq \eK^+$, and
    let $\Xi = \xi_1+\cdots+\xi_r$ be a  CM type of $\eK$.
    Given $\xi_0 \in J_\eK$, the period symbols
    $p_\eK(\xi_1,\xi_0), \ldots,p_\eK(\xi_r,\xi_0)$
are algebraically independent over $\ok$.
    \item[(b)] Suppose that $\eK$ is Galois over $\FF_{q}(t)$. Then
\[
\trdeg_{\bar{k}} \bar{k}(\tilde{\pi}, p_\eK(\xi_0,\xi)\mid \xi \in J_\eK) = 1+ \frac{([\eK:\eK^+]-1)}{[\eK:\eK^+]}\cdot [\eK:\FF_{q}(t)].
\]
\end{itemize}
\end{corollary}

\begin{proof}
(a) follows directly from Theorem~\ref{T:Shimura's conjecture}.
To prove (b),
the assumption says in particular that every $\xi$ must be of the form $\xi_0^\varrho$ for a unique $\varrho \in \Gal(\eK/\FF_q(t))$.
By Theorem~\ref{thm: Period Symbol} (4) one has that
\[
p_\eK(\xi,\xi_0) \sim
p_\eK(\xi_0^\varrho,\xi_0) \sim
p_\eK(\xi_0,\xi_0^{\varrho^{-1}}),
\]
whence
\begin{align*}
\bar{k}(p_\eK(\xi,\xi_0)\mid \xi \in J_\eK)
&= \bar{k}\bigl(p_\eK(\xi_0^\varrho,\xi_0)\mid \varrho \in \Gal(\eK/\FF_q(t))\bigr) \\
&=
\bar{k}\bigl(p_\eK(\xi_0,\xi_0^\varrho)\mid \varrho \in \Gal(\eK/\FF_q(t))\bigr) \\
&=
\bar{k}(p_\eK(\xi_0,\xi)\mid \xi \in J_\eK).
\end{align*}
Therefore  the result follows by Theorem~\ref{T:Shimura's conjecture}.
\end{proof}

\subsection{Proof of the main theorem}\label{sec: Proof MT}

To derive Theorem~\ref{T:Shimura's conjecture}, we remark that the Legendre relation in Lemma~\ref{L:Algebraic relation} assures the one-sided inequality:
\begin{align*}
\trdeg_{\bar{k}}\, \bar{k}(\Pcal_\eK(\xi,\xi_0)\mid\xi \in J_\eK)&\leq
1+\left([\eK:\FF_{q}(t)]- [\eK^{+}:\FF_{q}(t)]\right) \\
&= 1 + \frac{([\eK:\eK^+]-1)}{[\eK:\eK^+]} \cdot [\eK:\FF_q(t)].
\end{align*}
For the opposite inequality,
we introduce the notion of ``non-degeneracy'' of a generalized CM type as follows.

\begin{definition}
Let $\eK/\FF_{q}(t)$ be a CM field, and let $\Xi$ be a generalized CM type of $\eK$. We denote by $I_{\Xi}^0$ the subgroup of $I_\eK^0$ generated by $\varsigma(\Xi)$ for all $\varsigma \in \Gal(k^{\rm sep}/k)$. We say that $\Xi$ is \emph{non-degenerate} if
\[
\QQ\otimes_\ZZ I_{\Xi}^0 = \QQ\otimes_\ZZ I_\eK^0.
\]
\end{definition}

\begin{remark}\label{rem: rank}
For a non-degenerate generalized CM type $\Xi$, one has that
\[
\rank_\ZZ(I_\Xi^0) = \rank_\ZZ(I_\eK^0) = 1 + \frac{([\eK:\eK^+]-1)}{[\eK:\eK^+]} \cdot [\eK:\FF_q(t)].
\]
\end{remark}

Given $\xi_0 \in J_\eK$, we put
\begin{equation}\label{E:Xi_0}
\Xi_0 \assign [\eK:\eK^+]\xi_0 + \sum_{\substack{\xi\in J_\eK \\ \pi_{\bX/\bX^+}(\xi) \neq \pi_{\bX/\bX^+}(\xi_0)}} \xi \quad  \in I_\eK^0,
\end{equation}
and derive the following lemma.

\begin{lemma}\label{lem: non-degenerate}
The generalized CM type $\Xi_{0}$ of $\eK$ defined in~\eqref{E:Xi_0} is non-degenerate,
and we have
\[
\bar{k}\bigl(p_\eK(\xi,\Xi_0)\mid \xi \in J_\eK\bigr) \subseteq \bar{k}\bigl(p_\eK(\xi,\xi_0)\mid \xi \in J_\eK\bigr).
\]
\end{lemma}

\begin{proof}
Let
\[
\Phi_0\assign [\eK:\eK^+]\xi_0 - \sum_{\substack{\xi \in J_\eK \\ \pi_{\bX/\bX^+}(\xi) = \pi_{\bX/\bX^+}(\xi_0)}}\xi = \Xi_0 - N_\eK \quad \in I_\eK^0,
\]
where $N_\eK = \sum_{\xi \in J_\eK} \xi$.
Then from the definition of $p_\eK(\xi,\xi_0)$ one obtains
\[
p_\eK(\xi,\Xi_0) \sim
\tilde{\pi} \cdot p_\eK(\xi,\Phi_0) \sim
\tilde{\pi}^{1-1/[\eK^+:\FF_q(t)]} \cdot p_\eK(\xi,\xi_0)^{[\eK:\eK^+]}.
\]
Together with the Legendre relation in \eqref{eqn: Algebraic Relation}, we have that
\[
\bar{k}\bigl(p_\eK(\xi, \Xi_0)\mid \xi \in J_\eK\bigr) \subseteq \bar{k}\bigl(p_\eK(\xi,\xi_0)\mid \xi \in J_\eK\bigr).
\]

To show the non-degeneracy of $\Xi_0$, let $G_k\assign  \Gal(k^{\rm sep}/k)$.
Set
\[
K_0 \assign  \nu_{\xi_0}(\eK) \subset k^{\mathrm{sep}}
\quad \text{ and } \quad
H_0\assign  \Gal(k^{\mathrm{sep}}/K_0) \subset G_k.
\]
For each $\xi \in J_\eK$,
there exists a unique coset $\varsigma_\xi H_0 \in G_k$ such that $\nu_\xi = \varsigma_\xi \circ \nu_{\xi_0}$, which says that $\varsigma_\xi( \xi_0) = \xi$.
In particular, given $\xi^+ \in J_{\eK^+}$, every $\varsigma \in G_k$ sends
$\pi_{\bX/\bX^+}^{-1}(\xi^+)$ bijectively to
$\pi_{\bX/\bX^+}^{-1}(\varsigma (\xi^+))$.
Observe that
\[
\sum_{\varsigma \in G_k/H_0}\varsigma (\Xi_0) = [\eK:\FF_q(t)] \cdot N_\eK,
\]
i.e.,
\[
N_\eK = \frac{1}{[\eK:\FF_q(t)]} \otimes \sum_{\varsigma \in G_k/H_0}\varsigma (\Xi_0)  \quad \in \QQ \otimes_\ZZ I_\Xi^0.
\]
Moreover, for each CM type $\Xi$ of $\eK$,
write $\Xi = \xi_1+\cdots +\xi_d$, where $d = [\eK^+:\FF_q(t)]$. We let $\xi_i^+\assign  \pi_{\bX/\bX^+}(\xi_i)$ for $i = 1,\dots, d$,  and take
 $\varsigma_i \in G_k$ so that $\varsigma_i (\xi_0) = \xi_i$.
Then
\[
\varsigma_i(\Xi_0) = [\eK:\eK^+]\xi_i + \sum_{\substack{\xi \in J_\eK \\ \pi_{\bX/\bX^+}(\xi) \neq \xi_i^+}} \xi,
\]
which implies that
\begin{align*}
\sum_{i=1}^d \varsigma_i(\Xi_0)
&= \sum_{i=1}^d\left(
[\eK:\eK^+]\xi_i + \sum_{\substack{\xi \in J_\eK \\ \pi_{\bX/\bX^+}(\xi) \neq \xi_i^+}} \xi
\right) \\
&= [\eK:\eK^+] \Xi + \bigl([\eK^+:\FF_q(t)]-1\bigr) N_\eK.
\end{align*}
Hence $\Xi \in \QQ \otimes_\ZZ I_{\Xi_0}^0$.
Finally, since $I_\eK^0$ is generated by all CM types of $\eK$,
we obtain the desired identity
\[
\QQ \otimes_\ZZ I_{\Xi_0}^0
= \QQ \otimes_\ZZ I_\eK^0.
\qedhere
\]
\end{proof}

Finally, by Lemma~\ref{lem: non-degenerate} and Remark~\ref{rem: rank}, the proof of Theorem~\ref{T:Shimura's conjecture} is accomplished by the following key result.

\begin{theorem}\label{thm: lower-bound}
Let $\eK$ be a CM field over $\FF_q(t)$ and $\eK^+$ be the maximal totally real subfield of $\eK$.
For each generalized CM type $\Xi$ of $\eK$, we have
\[
\trdeg_{\bar{k}}\,\bar{k}(p_\eK(\xi,\Xi)\mid \xi \in J_\eK) \geq \rank_\ZZ I_\Xi^0.
\]
\end{theorem}

\begin{remark}\label{rem: Legendre}
Let $\eK$ be a CM field over $\FF_q(t)$ and $\Xi$ be a generalized CM type of $\eK$.
Suppose that $\Xi$ is non-degenerate.
Then Lemma~\ref{L:Algebraic relation} and Theorem~\ref{thm: lower-bound} ensure that
\[
\trdeg_{\bar{k}}\, \bar{k}\bigl(p_\eK(\xi,\Xi) \bigm| \xi \in J_\eK\bigr) = 1 + \frac{([\eK:\eK^+]-1)}{[\eK:\eK^+]} \cdot [\eK:\FF_q(t)],
\]
and that all algebraic relations among period symbols come from the  Legendre relation,
\[
\prod_{\substack{\xi \in J_\eK \\ \pi_{\bX/\bX^+}(\xi) = \xi^+}} p_\eK(\xi,\Xi) \sim \tilde{\pi}^{\text{\rm wt}(\Xi)}, \quad \forall\,\xi^+ \in J_{\eK^+}.
\]
\end{remark}

In order to get the lower bound in Theorem~\ref{thm: lower-bound}, we need the following explicit connection between period symbols and the periods of CM dual  $t$-motives:

\begin{proposition}\label{P:fields equality}
Let $\eM$ be a dual $t$-motive with generalized CM type $(\eK,\Xi)$ over $\bar{k}$ and put $r\assign[\eK:\FF_q(t)]$.
Let $\Upphi\in \Mat_{r}(\ok[t])$ represent multiplication by $\sigma$ on a $\ok[t]$-basis $\left\{m_{1},\ldots,m_{r} \right\}$ of $\eM$, and $\Uppsi\in \Mat_{r}(\TT^{\dagger})\cap \GL_{r}(\TT)$ be a rigid analytic trivialization for $\Upphi$. Then we have
\[
\ok\bigl(p_{\eK}(\xi,\Xi)|\xi\in J_{\eK} \bigr)=\ok\bigl(\Uppsi^{-1}(\theta) \bigr) = \bar{k}\bigl(\Uppsi(\theta)\bigr).
\]
\end{proposition}

\begin{proof}
For each $\xi \in J_\eK$, let $\Pfk_\xi$ be the maximal ideal of $O_\bK$ associated to $\xi$.
We then have the isomorphism
\[
\eM/(t-\theta)\eM\cong
\bigoplus_{\xi\in J_{\eK}} \eM/\Pfk_{\xi}\eM.
\]
As $\eM$ is a projective $O_{\bK}$-module of rank one, $\eM/\Pfk_\xi \eM$ is one dimensional over $\bar{k}$. Since the nontrivial differential $\omega_{\eM,\xi} \in H_{\mathrm{dR}}(\eM,\bar{k})$  given in Proposition~\ref{Prop: delta_M} factors through the quotient $\eM/\Pfk_{\xi}\eM$,
 the differentials
\[
\left\{\omega_{\eM,\xi}\mid \xi\in J_{\eK} \right\}
\]
form a $\ok$-basis of $H_{\mathrm{dR}}(\eM,\bar{k})$. On the other hand, denote by $\bar{m}_{i}$ the image of $m_{i}$ in $\eM/(t-\theta)\eM$ for $i=1,\ldots,r$. For each $1\leq j \leq r$,  we define the $\ok$-linear functional $\omega_j:\eM/(t-\theta)\eM \rightarrow \ok$ satisfying $\omega_j(\bar{m}_{i}) = 1$ if $i = j$, and $0$ otherwise.
It follows that
\[
\left\{\omega_1,\ldots,\omega_r \right\}
\]
forms another $\ok$-basis of $H_{\mathrm{dR}}(\eM,\bar{k})$.

Take an arbitrary nonzero $\gamma_0 \in H_{\mathrm{Betti}}(\eM)$. We have
\begin{align}\label{E:field equality}
 \ok\bigl(p_{\eK}(\xi,\Xi)\mid \xi\in J_{\eK}\bigr)
&= \ok\left(\int_{\gamma_0} \omega_{\eM,\xi} \ \bigg|\ \xi\in J_{\eK}\right) \\
&= \ok\bigl(\omega_i(\gamma)  \mid 1\leq i\leq r,\ \gamma \in H_{\mathrm{Betti}}(\eM)\bigr). \notag
\end{align}
Note that the last equality holds as $H_{\mathrm{Betti}}(\eM)$ is a projective $O_\eK$-module of rank one.

Put ${\mathbf{m}}\assign (m_{1},\ldots,m_{r})\in \Mat_{r\times 1}(\eM)$.
Since the entries of $\Psi^{-1}\mathbf{m}$ form an $\FF_{q}[t]$-basis of the Betti module $H_{\mathrm{Betti}}(\eM)$, the value of  $\omega_j$ specialized at the $i$-th component of $\Psi^{-1}\mathbf{m}$ is given by
\[
\omega_{\bar{m}_{j}}\left( \sum_{j=1}^{r}(\Psi^{-1})_{ij} m_{j}  \right) =(\Psi^{-1})_{ij}(\theta), \]
whence
\[
\ok\bigl(\omega_i(\gamma) \mid 1\leq i\leq r,\ \gamma\in H_{\rm{Betti}}(\eM)\bigr)=\ok(\Psi^{-1}(\theta))
\]
and the desired identity follows from~\eqref{E:field equality}.
\end{proof}

Based on Proposition~\ref{P:fields equality} and Theorem~\ref{T:dim=trdeg}, showing Theorem~\ref{thm: lower-bound} is then translated into finding a lower bound of the dimension of the corresponding $t$-motivic Galois group in question.
We shall use Hodge-Pink theory in the next section and prove Theorem~\ref{thm: lower-bound} at the end of Section~\ref{Sub:The case of CM dual t-motives}.

\section{\texorpdfstring{Hodge-Pink theory of dual $t$-motives}{Hodge-Pink theory on CM dual t-motives}}\label{sec: H-P theory}

In order to apply Hodge-Pink theory, we first review the properties of the categories of Hartl and Juschka that we need.
Then we describe explicitly the Hodge-Pink filtration arising from a CM dual $t$-motive with generalized CM type $(\eK,\Xi)$, which enables us to describe the Hodge-Pink cocharacters in question.
This characterization plays a crucial role in giving the precise connection between the lower bound in Theorem~\ref{thm: lower-bound} on the one hand and the rank of the Galois module generated by $\Xi$ in $I_{\eK}^{0}$ on the other.

\subsection{Hartl-Juschka Categories}
We recall the needed properties of the Hartl-Juschka Category $\Hcal\Jcal$ introduced in \cite[Definition 2.4.1]{HJ20} as follows:
an object of $\Hcal\Jcal$ is a free $\bar{k}[t]$-module $\eM$ of finite rank together with a $\sigma$-action on $M \assign  \bar{k}(t)\otimes_{\bar{k}[t]} \eM$ satisfying that
\begin{enumerate}
    \item $M$ is a pre-$t$-motive over $\bar{k}$;
    \item there exists a sufficiently large integer $n$ so that
\[
    (t-\theta)^n \eM \subseteq \sigma \eM \subseteq (t-\theta)^{-n} \eM.
\]
\end{enumerate}
A morphism $f : \eM \rightarrow \eM'$ in $\Hcal\Jcal$ is a $\bar{k}[t]$-module homomorphism so that the induced $\bar{k}(t)$-linear homomorphism $f: M = \bar{k}(t)\otimes_{\bar{k}[t]} \eM  \rightarrow M' = \bar{k}(t)\otimes_{\bar{k}[t]} \eM'$ satisfies
\[
f (\sigma m) = \sigma f(m), \quad \forall\, m \in M.
\]
Given two objects $\eM$ and $\eM'$ of $\Hcal\Jcal$, the set $\Hom(\eM,\eM')$ of morphisms from $\eM$ to $\eM'$ has a natural $\FF_q[t]$-module structure, and elements in
\[
\Hom^{0}(\eM,\eM')\assign  \FF_q(t)\otimes_{\FF_q[t]} \Hom(\eM,\eM')
\]
are called \textit{quasi-morphisms from $\eM$ to $\eM'$}.
The \textit{Hartl-Juschka Category $\Hcal\Jcal^I$ up to isogeny} has the same objects in $\Hcal\Jcal$ and the morphisms in $\Hcal\Jcal^I$ are given by quasi-morphisms.
In particular, the functor $\eM \mapsto  \bar{k}(t)\otimes_{\bar{k}[t]} \eM$ from $\Hcal\Jcal^I$ to the category  of pre-$t$-motives is faithful.

\begin{remark}
An object $\eM\in \Hcal\Jcal$ is called \emph{effective} if $\sigma \eM \subseteq \eM$.
We note that a dual $t$-motive is an effective object $\eM$ of $\Hcal\Jcal$, which is also  finitely generated over $\bar{k}[\sigma]$.
Thus the category of dual $t$-motives (resp.\ up to isogeny) is a strictly full subcategory of $\Hcal\Jcal$ (resp.\ $\Hcal\Jcal^I$).
\end{remark}

We let $\TT[\sigma]$ be the non-commutative algebra generated by $\sigma$ over $\TT$ subject to the relation
\[
\sigma f=f^{(-1)} \sigma, \quad \forall\, f\in \TT.
\]
Given an object $\eM \in \Hcal\Jcal^I$, we set
$\MM \assign  \TT \otimes_{\bar{k}[t]}\eM$, on which $\sigma$ acts diagonally, and so $\MM$ becomes a left $\TT[\sigma]$-module (as $t-\theta$ is invertible in $\TT$).
Put
\[
H_{\mathrm{Betti}}(\eM)\assign  \{ m \in \MM \mid \sigma m = m\}.
\]
We call $\eM$ \emph{uniformizable} (or \emph{rigid analytically trivial}) if the natural homomorphism
\[
\TT \otimes_{\FF_q[t]} H_{\mathrm{Betti}}(\eM) \longrightarrow \MM
\]
is an isomorphism (as $\TT$-modules). It is known that (see \cite[Proposition 2.4.27]{HJ20}) if $\eM$ is uniformizable, then $H_{\mathrm{Betti}}(\eM)$ is contained in $\MM^\dagger = \TT^\dagger \otimes_{\bar{k}[t]}\eM$, and the natural homomorphism
\[
\TT^\dagger \otimes_{\FF_q[t]} H_{\mathrm{Betti}}(\eM) \longrightarrow \MM^\dagger
\]
is an isomorphism (as $\TT^\dagger$-modules).

\begin{Subsubsec}{$t$-motivic Galois groups}\label{sec: t-motivic}
The full subcategory of $\Hcal\Jcal^I$ consisting of all uniformizable objects is denoted by $\Hcal\Jcal^I_U$.
By \cite[Remark 2.4.15]{HJ20}, the functor $\eM \mapsto \bar{k}(t)\otimes_{\bar{k}[t]} \eM$ from $\Hcal\Jcal^I_U$ to the category $\cR$ of uniformizable pre-$t$-motives is fully faithful.
Now we let $\mathbf{Vect}_{/\FF_q(t)}$ be the category of finite dimensional vector spaces over $\FF_q(t)$ and note that by \cite[Theorem 2.4.23]{HJ20} $\Hcal\Jcal^I_U$ is a neutral Tannakian category over $\FF_q(t)$ with fiber functor $\bomega: \Hcal\Jcal^I_U \rightarrow \mathbf{Vect}_{/\FF_q(t)}$ defined by:
\[
\eM \longmapsto H_\eM \assign  \FF_q(t)\otimes_{\FF_q[t]} H_{\mathrm{Betti}}(\eM).
\]

Given a uniformizable dual $t$-motive $\eM$, let $M \assign  \bar{k}(t)\otimes_{\bar{k}[t]} \eM$.
Viewing $\eM$ (resp.\ $M$) as an object in $\Hcal\Jcal^I_U$ (resp.\ $\cR$),
let $\Tcal_\eM$ (resp.\ $\cT_M$) be the strictly full Tannakian subcategory of $\Hcal\Jcal^I_U$ (resp.\ $\cR$) generated by $\eM$ (resp.\ $M$).
By Tannakian duality, we have an affine algebraic group scheme $\Gamma_\eM$ defined over $\FF_q(t)$ so that $\Tcal_\eM$ is equivalent to the category $\mathbf{Rep}(\Gamma_\eM)$ of the $\FF_q(t)$-finite dimensional representations of the algebraic group $\Gamma_\eM$.
We call $\Gamma_\eM$ the \emph{$t$-motivic Galois group of $\eM$}.
The fully faithful functor from $\Hcal\Jcal^I_U$ to $\cR$ gives an equivalence between $\Tcal_\eM$ and $\cT_M$, whence $\Gamma_\eM$ is isomorphic to $\Gamma_M$  defined in Section~\ref{Seubset: Tannakian category of t-motives}.
\end{Subsubsec}

\begin{remark}\label{rem: base change}
We recall the functorial property of the base change $\Gamma_{\eM,\CC_{\infty}}\assign \CC_{\infty}\times_{\FF_{q}(t)} \Gamma_{\eM}$ (via the embedding $\FF_{q}(t)\hookrightarrow \CC_{\infty}$ by sending $t$ to $\theta$) as follows.
Let $\bomega_\eM$ be the restriction of the fiber functor $\bomega$ to $\Tcal_\eM$.
For any object $\Xcal \in\Tcal_\eM $ we let $A_{\Xcal}\subset \Hom(\bomega_{\eM} (\Xcal),\bomega_{\eM} (\Xcal))$ be defined as in~\cite[p.~149]{DMOS} and $\underline{A}_{\Xcal}\in \Tcal_\eM$ be defined as in~\cite[Remark~3.3]{DMOS},
from which $\underline{A}_{\Xcal}$ is a ring in $\Tcal_\eM$ such that $\bomega_\eM(\underline{A}_{\Xcal})=A_{\Xcal}$ as $\FF_{q}(t)$-algebras.
Define
\[
\underline{B}\assign \varinjlim_{\Xcal} \underline{A}_{\Xcal}^{\vee}
\quad \text{ and } \quad
B\assign \varinjlim_\Xcal A_{\Xcal}^{\vee}.
\]
Then by~\cite[Remark~3.3]{DMOS} we have that
\[
\bomega_\eM(\underline{B})=B,\quad \underline{\End}^{\otimes}(\bomega_\eM)=\Spec(\bomega_\eM(\underline{B}))=\Gamma_{\eM},
\]
and for any fiber functor $\eta: \Tcal_{\eM}\rightarrow {\mathbf{Proj}_{R}}$ where $R$ is an $\FF_{q}(t)$-algebra and ${\mathbf{Proj}_{R}}$ is the category of finitely generated projective $R$-modules,
the following equality holds:
\[
\underline{\End}^{\otimes}(\eta)=\Spec(\eta(\underline{B})).
\]

In particular, let $R = \CC_\infty$ (i.e., $\mathbf{Proj}_R = \mathbf{Vect}_{/\CC_\infty}$) and take $\eta \assign \omega_{\eM,\CC_\infty}$, the base change of $\bomega_\eM$ from $\FF_q(t)$ to $\CC_\infty$, i.e.,
\[
\bomega_{\eM,\CC_\infty}(\Xcal)\assign \CC_\infty \otimes_{\FF_q(t)}\bomega_\eM(\Xcal).
\]
Then $\bomega_{\eM,\CC_\infty}$ is still a tensor functor, and we have that
\begin{align*}
\underline{\mathrm{Aut}}^{\otimes }(\bomega_{\eM,\CC_\infty})=\underline{\End}^{\otimes}(\bomega_{\eM,\CC_\infty})& =\Spec(\bomega_{\eM,\CC_\infty}(\underline{B})) \\
& =\Spec(\CC_{\infty}\otimes_{\FF_{q}(t)} \bomega_\eM(\underline{B}))= \Gamma_{\eM,\CC_{\infty}},
\end{align*}
where the first identity comes from~\cite[Prop.~1.13]{DMOS}.
\end{remark}

\subsection{Hodge-Pink filtration}
\label{sec: HP-fil}

Let $\eM$ be an object in $\Hcal\Jcal^I_U$.
Recall that $H_{\mathrm{Betti}}(\eM)$ actually lies in $\MM^\dagger$, and the natural map
\[
\TT^\dagger \otimes_{\FF_q[t]}
H_{\mathrm{Betti}}(\eM) \longrightarrow \MM^\dagger
\]
is an isomorphism as $\TT^\dagger$-modules.
Put
\[
\eM_{\CC_\infty}\assign  \CC_\infty[t] \otimes_{\bar{k}[t]}\eM
\quad \text{ and } \quad
H_{\eM,\CC_\infty} \assign \bomega_{\eM,\CC_\infty}(\eM) = \CC_\infty \otimes_{\FF_q(t)} H_\eM.
\]
The above isomorphism induces the chain of isomorphisms
\begin{multline}\label{eqn: H-iso}
H_{\eM,\CC_\infty} \cong \CC_\infty \otimes_{\FF_q[t]} H_{\mathrm{Betti}}(\eM) \cong \frac{\TT^\dagger}{(t-\theta) \TT^\dagger} \otimes_{\FF_q[t]} H_{\mathrm{Betti}}(\eM) \\
\cong \frac{\TT^\dagger}{(t-\theta) \TT^\dagger} \otimes_{\bar{k}[t]} \eM
\cong  \frac{\eM_{\CC_\infty}}{(t-\theta)\eM_{\CC_\infty}}.
\end{multline}

\begin{definition}\label{defn: HPF}
(cf.\ \cite[Section 2.4.5]{HJ20}) For each object $\eM$ in the category $\Hcal\Jcal^I_U$,
the \emph{Hodge-Pink filtration of $\eM$} is $F^\bullet H_{\eM,\CC_\infty} \assign (F^i H_{\eM,\CC_\infty})_{i\in\ZZ}$, where $F^i H_{\eM,\CC_\infty}$ is the image of $\eM_{\CC_\infty} \cap (t-\theta)^i (\sigma \eM_{\CC_\infty})$ in $H_{\eM,\CC_\infty}$, through the identification \eqref{eqn: H-iso} for each $i \in \ZZ$. That is,
\[
H_{\eM,\CC_\infty} \supseteq F^i H_{\eM,\CC_\infty} \cong  \frac{\bigl(\eM \cap (t-\theta)^i (\sigma \eM_{\CC_\infty})\bigr) + (t-\theta) \eM_{\CC_\infty}}{(t-\theta)\eM_{\CC_\infty}}.
\]
\end{definition}

\begin{remark}
Given an object $\eM$ of $\Hcal\Jcal^I_U$, there exists a sufficiently large integer $n$ such that
\begin{equation}\label{eqn: Fil}
(t-\theta)^n \eM \subseteq \sigma \eM \subseteq (t-\theta)^{-n} \eM.
\end{equation}
Let $r$ be the rank of $\eM$ over $\bar{k}[t]$.
There exists a unique (non-ordered) $r$-tuple of integers $(w_1,\ldots, w_r)$ so that
\[
\frac{(t-\theta)^{-n} \eM_{\CC_\infty}}{\sigma \eM_{\CC_\infty}} \cong \bigoplus_{i=1}^r \frac{\CC_\infty[t]}{(t-\theta)^{n-w_i} \CC_\infty[t]}.
\]
We call $(w_1,\dots, w_r)$ the \emph{Hodge-Pink weights of~$\eM$}.
In other words, $w_1,\dots, w_r$ are the ``jumps'' of the Hodge-Pink filtration of $\eM$.
In particular, $\eM$ is effective if and only if $w_1,\dots, w_r$ are all non-positive.
\end{remark}

We fix an object $\eM\in \Hcal\Jcal^I_U$ and let $\{m_1,\ldots, m_r\}$ be a basis of the free $\CC_\infty[t]$-module $\eM_{\CC_\infty}$ so that
\[
\sigma\eM_{\CC_\infty} = \bigoplus_{i=1}^r(t-\theta)^{-w_i}\CC_\infty[t] \cdot m_i \quad \subset \CC_\infty(t)\otimes_{\CC_\infty[t]} \eM_{\CC_\infty}.
\]
Identifying $H_{\eM,\CC_\infty}$ with $\eM_{\CC_\infty}/(t-\theta)\eM_{\CC_\infty}$ via \eqref{eqn: H-iso}, let $v_1,\ldots, v_r \in H_{\eM,\CC_\infty}$ be the images of $m_1,\dots, m_r$ through the identification \eqref{eqn: H-iso}.
It is straightforward to see that
\[
F^{i_0} H_{\eM,\CC_\infty}
= \bigoplus_{\substack{1\leq i\leq r \\ w_{i}\geq i_0}}
\CC_\infty v_{i} \quad \forall\,i_0 \in \ZZ.
\]
In particular, we have the associated graded spaces
\begin{equation}\label{eqn: Grading}
\mathrm{Gr}^{i_0} (H_{\eM,\CC_\infty})
\assign \bigoplus_{\substack{1\leq i \leq r \\ w_{i} = i_0}} \CC_\infty v_{i} \cong
\frac{F^{i_0}H_{\eM,\CC_\infty}}{F^{i_0+1}H_{\eM,\CC_\infty}}, \quad \forall\,i_0 \in \ZZ,
\end{equation}
and
\begin{equation}\label{E:HM-Gr}
H_{\eM,\CC_\infty} = \bigoplus_{i_0 \in \ZZ} {\rm Gr}^{i_0}(H_{\eM,\CC_\infty}).
\end{equation}

\begin{lemma}\label{lem: Grading}
Given two objects $\eM$ and $\eM'$ in $\Hcal\Jcal^I_U$, we
let $\eM_0 \assign  \eM \otimes_{\bar{k}[t]} \eM'\in \Hcal\Jcal^I_U$, on which $\sigma$ acts diagonally. The natural isomorphism $H_{\eM_0,\CC_\infty} \cong H_{\eM,\CC_\infty} \otimes_{\CC_\infty} H_{\eM',\CC_\infty}$
induces
\[
\Gr^{i_0}(H_{\eM_0,\CC_\infty}) \cong \bigoplus_{j_0+j_0' = i_0} \mathrm{Gr}^{j_0} (H_{\eM,\CC_\infty})\otimes_{\CC_\infty} \mathrm{Gr}^{j_0'} (H_{\eM',\CC_\infty}), \quad \forall i_0 \in \ZZ.
\]
\end{lemma}

\begin{proof}
Let $(w_1,\ldots, w_r)$ (resp.\ $(w_1',\ldots, w_r')$) be the Hodge-Pink weights of $\eM$ (resp.\ $\eM'$).
Take $\{m_1,\ldots, m_r\}$ (resp.\ $\{m_1',\ldots, m_r'\}$) to be a basis of the free $\CC_\infty[t]$-module $\eM_{\CC_\infty}$ (resp.\ $\eM_{\CC_\infty}'$) so that
$$
\sigma\eM_{\CC_\infty} = \bigoplus_{j=1}^r(t-\theta)^{-w_j}\CC_\infty[t] \cdot m_j \quad \text{ and } \quad
\sigma\eM_{\CC_\infty}' = \bigoplus_{j'=1}^r(t-\theta)^{-w_{j'}'}\CC_\infty[t] \cdot m_{j'}'.
$$
Then $\{m_j\otimes m'_{j'}:1\leq j,j' \leq r\}$ is a basis of $\eM_{0,\CC_\infty}$ so that
$$
\sigma \eM_{0,\CC_\infty} =  \sigma\eM_{\CC_\infty}\otimes_{\CC_\infty[t]} \sigma\eM_{\CC_\infty}'= \bigoplus_{1\leq j,j'\leq r} (t-\theta)^{-(w_j+w_{j'}')}\CC_\infty[t] \cdot m_j\otimes m'_{j'}.
$$
This means that $w_j+w_{j'}'$, for $1\leq j,j' \leq r$, are the Hodge-Pink weights of $\eM_0$.

For $1\leq j,j' \leq r$, let $v_j$, $v'_{j'}$, and $v_{j,j'}$ be the image of $m_j$, $m'_{j'}$, and $m_j\otimes m'_{j'}$ in $H_{\eM,\CC_\infty}$, $H_{\eM',\CC_\infty}$, and $H_{\eM_0,\CC_\infty}$, respectively (through the identification \eqref{eqn: H-iso}).
The canonical isomorphism $H_{\eM,\CC_\infty}\otimes_{\CC_\infty} H_{\eM',\CC_\infty} \cong H_{\eM_0,\CC_\infty}$
sending $v_j\otimes v'_{j'}$ to $v_{j,j'}$ and the isomorphism in \eqref{eqn: Grading} implies that for each $i_0 \in \ZZ$,
\begin{align*}
{\Gr}^{i_0}(H_{\eM_0,\CC_\infty})
&\cong \bigoplus_{\substack{1\leq j,j'\leq r \\ w_{j}+w_{j'}' = i_0}} \CC_\infty v_{j,j'} \\
&=
\bigoplus_{j_0+j_0'=i_0}\Biggl(
\bigoplus_{\substack{1\leq j,j'\leq r \\ w_{j} = j_0,\ w_{j'}'=j_0'}}\CC_\infty  v_{j,j'} \Biggr) \\
&\cong
\bigoplus_{j_0+j_0'=i_0}\Biggl[\biggl(
\bigoplus_{\substack{1\leq j\leq r \\ w_{j} = j_0}}\CC_\infty v_{j}
\biggr) \otimes_{\CC_\infty}
\biggl(
\bigoplus_{\substack{1\leq j'\leq r \\ w_{j'}' = j_0'}}\CC_\infty v'_{j'}
\biggr)\Biggr] \\
&\cong
\bigoplus_{j_0+j_0' = i_0} {\Gr}^{j_0} (H_{\eM,\CC_\infty})\otimes_{\CC_\infty} {\Gr}^{j_0'} (H_{\eM',\CC_\infty}).
\end{align*}
\end{proof}

In the next subsection we shall use the Hodge-Pink filtration of the given object $\eM$ in $\Hcal\Jcal^I_U$ to construct a special family of cocharacters of the $t$-motivic Galois group $\Gamma_\eM$.

\subsection{Hodge-Pink cocharacters}
\label{sec: HP coc}

Given an object $\eM$ in $\Hcal\Jcal^I_U$,
let $\cT_{M}$ be the neutral Tannikian subcategory of $\cT$ over $\FF_{q}(t)$ generated by the pre-$t$-motive $M = \bar{k}(t)\otimes_{\bar{k}[t]}\eM$ (see Definition~\ref{Def: Tannakian category of t-motives} (2)).
Recall in \ref{sec: t-motivic} that the Tannakian subcategory $\Tcal_\eM$ of $\Hcal\Jcal^I_U$ generated by $\eM$ is equivalent to $\cT_{M}$, and so we have the algebraic group isomorphism (over $\FF_q(t)$)
\[
\Gamma_\eM \cong \Gamma_M.
\]
Note that we have a faithful representation $$\Gamma_M \hookrightarrow \GL(H_{\eM}),$$
which is functorial in $M$ (see~\cite[Thm.~4.5.3]{Papanikolas}). This implies that
\begin{equation}\label{eqn: CI}
\Gamma_{M}\hookrightarrow {\Cent}_{\GL(H_{\eM})}\left( \End_{\cT}(M)\right)\quad  \textup{(cf.~\cite[p.~1448, (6)]{CPY10})}
\end{equation}
as we have the natural embedding $\End_{\cT}(M)\hookrightarrow \End(H_{\eM})$.

We let $\GG_{m/\CC_\infty}$ be the multiplicative group over $\CC_{\infty}$, and we denote by $\mathbf{Rep}(\GG_{m/\CC_\infty})$ the Tannakian category of finite dimensional algebraic representations of $\GG_{m/\CC_\infty}$ over $\CC_{\infty}$. Note that by~\cite[p.~144, (2.30)]{DMOS}, $\mathbf{Rep}(\GG_{m/\CC_\infty})$ is equivalent to the category ${\mathbf{Grad.Vect}}_{/\CC_{\infty}}$ of $\ZZ$-graded vector spaces of finite dimensions over $\CC_{\infty}$.
The fiber functor, which is also the forgetful functor, is denoted by
\[
\bomega^{\GG_{m/\CC_\infty}}:\mathbf{Rep}(\GG_{m/\CC_\infty})\rightarrow {\mathbf{Vect}}_{/\CC_{\infty}}.
\]

For any object $\eM\in \Hcal\Jcal^I_U$, we have the following functor (induced by the Hodge-Pink filtrations),
\[
\Fcal:\Tcal_\eM \rightarrow {\mathbf{Grad.Vect}}_{/\CC_{\infty}}, \quad \Fcal(\eM') \assign {\oplus_{i\in \ZZ}\rm{Gr}}^{i}H_{M',\CC_{\infty}}.
\]
By Lemma~\ref{lem: Grading}, one checks that $\Fcal$ is a tensor functor.
On the other hand, recall from Remark~\ref{rem: base change} that $\bomega_{\eM,\CC_\infty}$ denotes the base change of the fiber functor $\bomega_\eM$ from $\FF_q(t)$ to $\CC_\infty$, i.e.
\[
\bomega_{\eM,\CC_\infty}: \Tcal_\eM \rightarrow \mathbf{Vect}_{/\CC_\infty}, \quad \bomega_{\eM,\CC_\infty}(\eM') = \CC_\infty \otimes_{\FF_q(t)} H_{\eM'}.
\]
Let $\bomega^{\rm GV} : \mathbf{Grad.Vect}_{/\CC_\infty} \rightarrow \mathbf{Vect}_{/\CC_\infty}$ be the forgetful functor.
By~\eqref{E:HM-Gr} we then have the following commutative diagram
\[
\SelectTips{cm}{}
\xymatrix{
({\mathbf{Rep}}(\Gamma_{\eM}) & \hspace{-0.9cm})\hspace{0.7cm}  \Tcal_\eM \ar[d]_{\bomega_{\eM,\CC_\infty}} \ar[r]^<<<<<<{\Fcal} \ar@{}[l]|*=0[@]{\approx}
&{\mathbf{Grad.Vect}}_{/\CC_{\infty}} \hspace{0.3cm} ( \ar[ld]^{\bomega^{\rm GV}} \ar@{}[r]|*=0[@]{\approx}
& \hspace{-0.3cm}\mathbf{Rep}(\GG_{m/\CC_\infty})) \\
& {\mathbf{Vect}}_{/\CC_{\infty}} &
},
\]
By the arguments of~\cite[p.~130, Cor.~2.9]{DMOS} we obtain a cocharacter of $\Gamma_{\eM,\CC_{\infty}}$ (induced from $\Fcal$),
\[
\chi_{\eM}^\vee:\GG_{m/\CC_{\infty}}={\rm{Aut}}^{\otimes}(\bomega^{\rm GV})\rightarrow {\rm{Aut}}^{\otimes}(\bomega_{\eM,\CC_{\infty}})=\Gamma_{\eM,\CC_{\infty}}\hookrightarrow \GL(H_{\eM,\CC_{\infty}}).
\]
Note that the image of $\chi_{\eM}^{\vee}$ is uniquely determined by the action of $\GG_{m/\CC_{\infty}}$ on \[H_{\eM,\CC_\infty} = \bigoplus_{i \in \ZZ} {\rm Gr}^{i}(H_{\eM,\CC_\infty}),\] i.e., for each $i\in \ZZ$ and $x\in \GG_{m/\CC_{\infty}}(\CC_{\infty})$ we have
\begin{equation}\label{E:chiMvee}
\chi_{\eM}^\vee(x) v = x^i \cdot v, \quad \forall v \in \text{Gr}^i(H_{\eM,\CC_\infty}).
\end{equation}

Since $\GL(H_{\eM,\CC_\infty})$ (resp.~$\Gamma_{\eM,\CC_{\infty}}$) is the base change $\CC_{\infty}\times_{\FF_{q}(t)} \GL(H_{\eM})$ (resp.~$\CC_{\infty}\times_{\FF_{q}(t)}\Gamma_{\eM}$), we have a natural action of $\Aut_{k}(\CC_{\infty})$ on $\GL(H_{\eM,\CC_\infty})$, which leaves $\Gamma_{\eM,\CC_{\infty}}$ invariant. Thus any $\Aut_k(\CC_\infty)$-conjugate of $\chi_\eM^\vee$ is still a cocharacter of $\Gamma_{\eM,\CC_\infty}$.

\begin{definition}\label{defn: HPC}
(cf.\ \cite[Definition 7.10]{Pink}) Given an object $\eM$ in $\Hcal\Jcal^I_U$, a \emph{Hodge-Pink cocharacter} of $\Gamma_{\eM,\CC_{\infty}}$ is a cocharacter which lies in the $\Gamma_{\eM,\CC_\infty} \rtimes \Aut_k(\CC_\infty)$-conjugate class of $\chi_\eM^\vee$.
\end{definition}

We shall describe the images of the Hodge-Pink cocharacters of $\Gamma_{\eM,\CC_\infty}$ in $\GL(H_{\eM,\CC_\infty})$ via \eqref{E:chiMvee} when $\eM$ is a CM dual $t$-motive, and provide a lower bound for the dimension of $\Gamma_{\eM}$ in the next subsection.

\subsection{\texorpdfstring{The case of CM dual $t$-motives}{The case of CM dual t-motives}}\label{Sub:The case of CM dual t-motives}

Let $\eM$ be a CM dual $t$-motive with generalized CM type $(\eK,\Xi)$, where $\eK$ is a CM field with $[\eK:\FF_q(t)] = r$.
For $\xi \in J_\eK$, let $\Pfk_\xi$ be the maximal ideal of $O_\bK$ associated to $\xi$, i.e.\ $\Pfk_\xi$ is generated by $\alpha - \nu_\xi(\alpha)$ for all $\alpha \in O_\eK$.
Write $\Xi = \sum_{\xi \in J_\eK} m_\xi \xi$ and put $\Ifk_\Xi \assign  \prod_{\xi \in J_\eK} \Pfk_\xi^{m_\xi}$.
As $\eM$ is projective of rank one over $O_\bK$ and $\sigma \eM = \Ifk_\Xi \eM$ (see the equality~\eqref{eqn: sigma-M iso}), we have that
\[
\frac{\eM}{\sigma \eM} \cong \frac{O_\bK}{\Ifk_\Xi} \cong \prod_{\xi \in J_\eK}\frac{O_\bK}{\Pfk_\xi^{m_\xi}}.
\]
Since $(t-\theta)O_\bK = \prod_{\xi \in J_\eK} \Pfk_\xi$, i.e.,\ the prime ideal $(t-\theta)\bar{k}[t]$ of $\bar{k}[t]$ splits completely in $O_\bK$,
for each $\xi \in J_\eK$ one has that
\[
\bar{k}[t] \cap \Pfk_{\xi}^{m_\xi} = (t-\theta)^{m_\xi}\cdot \bar{k}[t] \quad \text{ and } \quad \dim_{\bar{k}}(\frac{O_\bK}{\Pfk_\xi^{m_\xi}}) = m_\xi.
\]
Thus the natural inclusion $\bar{k}[t]\hookrightarrow O_{\bK}$
induces the following isomorphism (as $\bar{k}[t]$-modules)
\[
\frac{O_\bK}{\Pfk_\xi^{m_\xi}} \cong \frac{\bar{k}[t]}{(t-\theta)^{m_\xi}\bar{k}[t]} \quad \textup{ for every $\xi \in J_\eK$.}
\]
To conclude, we obtain the following description of Hodge-Pink weights of $\eM$.

\begin{proposition}\label{prop: HPW-CM}
Let $\eK/\FF_{q}(t)$ be a CM field and $\eM$ be a CM dual $t$-motive with generalized CM type $(\eK, \Xi)$. Write $\Xi = \sum_{\xi \in J_\eK} m_\xi \xi \in I_\eK^0$.
Then
\[
\frac{\eM_{\CC_\infty}}{\sigma \eM_{\CC_\infty}} \cong \prod_{\xi \in J_\eK} \frac{\CC_\infty[t]}{(t-\theta)^{m_\xi}\CC_\infty[t]}.
\]
Consequently, the Hodge-Pink weights of $\eM$ are
$(-m_\xi)_{\xi \in J_\eK}$.
\end{proposition}

Let $\eM$ be a CM dual $t$-motive with generalized CM type $(\eK,\Xi)$ over $\bar{k}$,
and let $M = \bar{k}(t)\otimes_{\bar{k}[t]} \eM$ be the pre-$t$-motive associated to $\eM$. Since $O_{\eK}\subset \End_{\ok[t,\sigma]}(\eM)$,
we can embed $\eK$ into the endomorphism ring $\End_{\cT}(M)$.
From \eqref{eqn: CI}, the closed immersion $\Gamma_\eM\hookrightarrow \GL_{\FF_q(t)}(H_\eM)$ actually factors through $T_\eK \assign  {\rm Res}_{\eK/\FF_q(t)} \GG_m$, the Weil restriction of $\GG_m$ from $\eK$ to $\FF_q(t)$.
Thus the base change $\Gamma_{\eM,\CC_\infty}$ can be viewed as a subgroup of
\[
T_{\eK,\CC_\infty} \assign  \CC_\infty \underset{\FF_q(t)}{\times} {\rm Res}_{\eK/\FF_q(t)} \GG_{m}
\cong \prod_{\xi \in J_\eK} \CC_\infty \underset{\nu_\xi, \eK}{\times} \GG_{m/\eK}.
\]
Moreover, for each $\xi \in J_\eK$, let $H_\xi$ be the unique one-dimensional subspace of $H_{\eM,\CC_\infty}$ satisfying that
\[
\alpha \cdot v = \nu_\xi(\alpha) \cdot v, \quad \forall\, \alpha \in O_\eK,\  v \in H_\xi.
\]
The decomposition
\[
H_{\eM,\CC_\infty} = \bigoplus_{\xi \in J_\eK} H_\xi
\]
gives the embedding
\[
\prod_{\xi \in J_\eK} \CC_\infty \underset{\nu_\xi, \eK}{\times} \GG_{m/\eK} \hookrightarrow \GL_{\CC_\infty}(H_{\eM,\CC_\infty}), \quad (x_\xi)_{\xi \in J_\eK} \longmapsto \begin{pmatrix} \ddots & & \\ & x_\xi & \\ & & \ddots \end{pmatrix}.
\]

Write $\Xi = \sum_\xi m_\xi \xi \in I_\eK^0$. Then for $i \in \ZZ$, we have
\[
{\Gr}^i(H_{\eM,\CC_\infty}) = \bigoplus_{\substack{\xi \in J_\eK \\ -m_\xi = i}} H_\xi \quad \textup{(see~\eqref{eqn: Grading}).}
\]
Therefore by~\eqref{E:chiMvee} the cocharacter
\[
\chi_\eM^\vee: \GG_{m/\CC_\infty} \longrightarrow \Gamma_{\eM,\CC_\infty} \subset
T_{\eK}
\cong \prod_{\xi \in J_\eK} \CC_\infty \underset{\nu_\xi, \eK}{\times} \GG_{m/\eK}
\]
is realized as
\[
\chi_\eM^\vee(x) = (x^{-m_\xi})_{\xi \in J_\eK} \quad \in \prod_{\xi \in J_\eK} \CC_\infty \underset{\nu_\xi, \eK}{\times} \GG_{m/\eK}.
\]

On the other hand, for each cocharacter $\chi^\vee:  \GG_{m/\CC_\infty} \rightarrow T_{\eK,\CC_\infty}$,
there exists a unique set of integers $\{w_\xi: \xi \in J_\eK\}$ (i.e., the coweights of $\chi^\vee$) so that
\[
\chi^\vee(x) = (x^{w_\xi})_{ \xi \in J_\eK} \quad  \in \prod_{\xi \in J_\eK} \CC_\infty \underset{\nu_\xi, \eK}{\times} \GG_{m/\eK}.
\]
We remark that the map
\[
\chi^\vee \longmapsto
\Phi_{\chi^\vee} \assign  \sum_{\xi \in J_\eK}(-w_\xi) \xi \in I_\eK
\]
gives an isomorphism between the group of cocharacters of $T_{\eK,\CC_\infty}$ and $I_\eK$, which is equivariant under the action of $\Aut_k(\CC_\infty)$, i.e.\
\[
\Phi_{({}^\varsigma \chi^\vee)} = \varsigma (\Phi_{\chi^\vee}),
\quad
\forall\,\chi^\vee: \GG_{m/\CC_\infty} \rightarrow T_{\eK,\CC_\infty} \ \text{ and }\ \varsigma \in \Aut_k(\CC_\infty).
\]
In particular, we have
\begin{equation}\label{eqn: HPc-CMt}
\Phi_{\chi^\vee_\eM} = \Xi \in I_\eK^0.
\end{equation}
As $\Gamma_\eM$ is commutative, we obtain the following correspondence.

\begin{proposition}\label{prop: HP-wt and CM type}
Under the above identification between the group of cocharacters of $T_{\eK,\CC_\infty}$ and $I_\eK$, the set of the Hodge-Pink cocharacters of $\Gamma_\eM(\CC_{\infty})$ corresponds bijectively to
\[
\{\varsigma (\Xi) :  \varsigma \in \Aut_k(\CC_\infty)\}.
\]
\end{proposition}

\begin{remark}\label{rem: TM}
(1) Since every point $\xi \in J_\eK$ given in~\eqref{E:JK} is defined over $k^{\text{sep}}$ (as $\eK/\FF_q(t)$ is separable),
one has
\[
\{\varsigma (\Xi) :  \varsigma \in \Aut_k(\CC_\infty)\} =
\{\varsigma (\Xi) :  \varsigma \in \Gal(k^{\text{sep}}/k)\}.
\]
(2) Let
$I^0_\Xi$ be the subgroup of $I_\eK^0$ generated by $\varsigma (\Xi)$ for all $\varsigma \in \Gal(k^{\text{sep}}/k)$,
and let
$T_\eM$ be the subgroup of $\Gamma_\eM$ generated by
the image of all the Hodge-Pink cocharacters of $\Gamma_\eM$.
Then
\[
\dim \Gamma_\eM \geq \dim T_\eM = \rank_\ZZ I^0_\Xi.
\]
\end{remark}

\begin{proof}[Proof of Theorem~\ref{thm: lower-bound}]
Given a non-degenerate CM type $\Xi$ of $\eK$, let $\eM$ be a CM dual $t$-motive with generalized CM type $(\eK,\Xi)$.
From Proposition~\ref{P:fields equality} and Theorem~\ref{T:dim=trdeg}, we obtain that
\[
\trdeg_{\bar{k}} \bar{k}\bigl(p_\eK(\xi,\Xi)\mid \xi \in J_\eK\bigr)
= \dim \Gamma_\eM \geq \rank_\ZZ I_\Xi^0
\]
as desired.
\end{proof}

\section{\texorpdfstring{Application to periods of abelian $t$-modules}{Application to periods of abelian t-modules}}
\label{sec: App}

Let $E_\rho$ be an abelian $t$-module over $\bar{k}$ and $\eM(\rho)$ be the Hartl-Juschka dual $t$-motive associated with $E_\rho$ (see\ Remark~\ref{rem: M(rho)}).
Based on Remark~\ref{Rmk:Compactiblity of CM type}, we say that $E_\rho$ is \emph{CM with generalized CM type $(\eK,\Xi)$} if $\eM(\rho)$ is a CM dual $t$-motive with generalized CM type $(\eK,\Xi)$.
The aim of this section is to apply Theorem~\ref{thm: lower-bound} to prove  the algebraic independence of the coordinates of a nonzero period vector of $E_{\rho}$ when $E_{\rho}$ has a non-degenerate CM type.
To do so, we first establish the needed comparisons for the Lie spaces, Betti modules and de Rham modules in terms of $t$-motives in the following subsections.

\subsection{Comparisons between Lie spaces}\label{sec: Com Lie}

Let $E_\rho$ be an abelian $t$-module over $\ok$, and let $\Mscr(\rho)$ be the $\ok[t,\tau]$-module defined in Definition~\ref{defn: t-module} (2).
There is a canonical $\bar{k}$-vector space isomorphism  (see~\cite[Lemma 1.3.4]{Anderson86})
\begin{equation}\label{eqn: Lie1}
\Lie(E_\rho) \cong \Hom_{\bar{k}}(\Mscr(\rho)/\tau \Mscr(\rho), \bar{k}).
\end{equation}
On the other hand, as $(t-\theta)^N\eM(\rho) \subset \sigma \eM(\rho)$ for a sufficiently large integer $N$, we may identify $\eM(\rho) \assign \Hom_{\bar{k}[t]}(\tau \Mscr(\rho),\bar{k}[t]dt)$ with
\[
\biggl\{ \eum \in \Hom_{\ok[t]}\Bigl(\Mscr(\rho), \bar{k}[t][\frac{1}{t-\theta}] dt\Bigr) \biggm| \eum(\tau \Mscr(\rho)) \subset \ok[t]dt\biggr\},
\]
and $\sigma \eM(\rho)$ corresponds to  $\Hom_{\ok[t]}(\Mscr(\rho),\ok[t]dt)$.
Thus we obtain a $\ok$-bilinear map
\[
\Mscr(\rho) \times \eM(\rho) \longrightarrow \ok[t][\frac{1}{t-\theta}] dt \xrightarrow{\Res_{t=\theta}} \bar{k},
\]
where $(m,\eum)\mapsto \Res_{t=\theta}\eum(m)$ for every $m \in \Mscr(\rho)$ and $\eum \in \eM(\rho)$.
This induces a perfect pairing
\begin{equation}\label{eqn: perfect pairing}
\frac{\Mscr(\rho)}{\tau \Mscr(\rho)} \times \frac{\eM(\rho)}{\sigma \eM(\rho)} \longrightarrow \bar{k}.
\end{equation}

\begin{lemma}\label{lem: comp Lie}
The above perfect pairing and \eqref{eqn: Lie1} lead to an isomorphism of $\bar{k}$-vector spaces,
\[
\Lie(E_\rho) \cong \eM(\rho)/\sigma \eM(\rho).
\]
\end{lemma}

\begin{remark}
Viewing $\Lie(E_\rho)$ as a $\bar{k}[t] \ (=\!\!\bar{k}\otimes_{\FF_q}\FF_q[t])$-module via $\text{id}_{\bar{k}} \otimes \partial \rho$, the above isomorphism is actually $\bar{k}[t]$-linear.
\end{remark}

\subsection{Comparisons between period lattices and Betti modules}\label{sec: Lambda-Betti}
Suppose $E_\rho$ is a uniformizable abelian $t$-module over $\bar{k}$, i.e.\
the corresponding exponential map $\exp_\rho : \Lie(E_\rho)(\CC_\infty) \rightarrow E_\rho(\CC_\infty)$ is surjective.
Let $\Lambda_\rho$ be the period lattice of $E_\rho$, i.e.
\[
\Lambda_\rho \assign  \{\lambda \in \Lie(E)(\CC_\infty)\mid \exp_\rho(\lambda) = 0\},
\]
which is a discrete $\FF_q[t]$-submodule with
\[
\rank_{\FF_q[t]}(\Lambda_\rho) = \rank_{\bar{k}[t]} \eM(\rho).
\]
For each $\lambda \in \Lambda_\rho$, let $\Gcal_\lambda: \Mscr(\rho)\rightarrow \TT $ be the following \emph{Anderson generating function}
\begin{equation}\label{eqn: AGF}
\langle m \mid \Gcal_\lambda\rangle \assign  \sum_{n=0}^\infty m\Bigl(\exp_\rho(\partial \rho_t^{-n-1}\lambda)\Bigr) t^n, \quad \forall\,m \in \Mscr(\rho).
\end{equation}
Following~\cite{NP21}, we take
\[
\gamma_\lambda\assign  -\Gcal_\lambda\big|_{\tau \Mscr(\rho)} dt\in \Hom_{\bar{k}[t]}(\tau \Mscr(\rho), \TT\, dt) \cong \TT \otimes_{\bar{k}[t]} \eM(\rho).
\]
It is straightforward to check that $\sigma \cdot \gamma_\lambda = \gamma_\lambda$, which means that
$\gamma_\lambda \in H_{\mathrm{Betti}}(\eM(\rho))$.

\begin{lemma}\label{lem: L(rho)-H(rho)}
The map $(\lambda\longmapsto \gamma_\lambda)$ gives an $\FF_{q}[t]$-module isomorphism from $\Lambda_\rho$ onto $H_{\mathrm{Betti}}(\eM(\rho))$.
\end{lemma}

\begin{proof}
Note that $\rank_{\FF_q[t]}\Lambda_\rho = \rank_{\bar{k}[t]} \eM(\rho) = \rank_{\FF_q[t]} H_{\mathrm{Betti}}(\eM(\rho))$. Let $\{\lambda_1,\ldots, \lambda_r\}$ be an $\FF_q[t]$-base of $\Lambda_\rho$ and $\{\tau m_1, \ldots,  \tau m_r\}$ be a $\bar{k}[t]$-base of $\tau \Mscr(\rho)$.
The result follows immediately from the fact (cf.\ \cite[Proposition 4.3.10]{NP21}, see also \cite[Proposition 6.2.4]{GP19} and \cite[Proposition 3.3.9]{Papanikolas}) that
\[
(\langle \tau m_i \mid \Gcal_{\lambda_j}\rangle)_{1\leq i,j\leq n} \in \GL_r(\TT).
\qedhere
\]
\end{proof}

Recall that $\TT^\dagger$ is the ring of rigid analytic functions on $\CC_\infty\setminus \{\theta^{q^i}\mid i \in \NN\}$.
As $E_\rho$ is uniformizable, so is $\eM(\rho)$.
Thus by~\cite[Prop.~2.3.30]{HJ20}  we have that
\[
H_{\mathrm{Betti}}(\eM(\rho)) \subseteq \TT^\dagger \otimes_{\bar{k}[t]}\eM(\rho) \subseteq
\CC_\infty[\![t-\theta]\!] \otimes_{\bar{k}[t]}\eM(\rho).
\]
Let $\bar{\gamma}_\lambda$ be the image of $\gamma_\lambda$ in $\Hom_{\bar{k}}(\Mscr(\rho)/\tau\Mscr(\rho),\CC_\infty)$ under the following maps
\begin{align*}
\CC_\infty[\![t-\theta]\!] \otimes_{\bar{k}[t]}\eM(\rho)\twoheadrightarrow \CC_\infty[\![t-\theta]\!]\otimes_{\bar{k}[t]}\frac{\eM(\rho)}{\sigma \eM(\rho)} & \cong \CC_\infty \otimes_{\bar{k}} \frac{\eM(\rho)}{\sigma \eM(\rho)} \\
& \cong \Hom_{\bar{k}}(\frac{\Mscr(\rho)}{\tau\Mscr(\rho)},\CC_\infty),
\end{align*}
where the last isomorphism comes from \eqref{eqn: perfect pairing}. We obtain
\[
\bar{\gamma}_\lambda(m) = -\Res_{t=\theta}\Bigl( \langle m |\Gcal_\lambda\rangle dt\Bigr),
\]
which coincides with the image of $\lambda$ in $\Hom_{\bar{k}}(\Mscr(\rho)/\tau \Mscr(\rho),\CC_\infty)$ under the isomorphism \eqref{eqn: Lie1} by \cite[Lemma 1.3.4 and Proposition 3.3.2]{Anderson86}.
In conclusion we obtain the following.

\begin{proposition}\label{P:PeriodsDiagComm}
The following diagram commutes:
\[
\begin{tabular}{ccccc}
$\Lie(E)(\CC_\infty)$ & $\stackrel{\sim}{\longrightarrow}$ &
$\displaystyle\Hom_{\bar{k}}(\frac{\Mscr(\rho)}{\tau \Mscr(\rho)},\CC_\infty)$ & $\stackrel[\eqref{eqn: perfect pairing}]{\sim}{\longleftarrow}$ &
$\hspace{-0.3cm}\displaystyle \CC_\infty \otimes_{\bar{k}} \frac{\eM(\rho)}{\sigma \eM(\rho)}$ \\
$\cup$ & & $\uparrow$ & & \rotatebox{90}{\scalebox{1}[1]{$\cong$}} \\
$\Lambda_\rho$ & $\stackrel{\sim}{\longrightarrow}$ &
$H_{\mathrm{Betti}}(\eM(\rho))$ & $\longrightarrow$ & $\displaystyle \CC_\infty[\![t-\theta]\!] \otimes_{\bar{k}[t]} \frac{\eM(\rho)}{\sigma \eM(\rho)}$. \\
\rotatebox{90}{\scalebox{1}[1]{$\in$}} & & \rotatebox{90}{\scalebox{1}[1]{$\in$}} & & \\
$\lambda$ & $\longmapsto$ & $\gamma_\lambda$ & &
\end{tabular}
\]
\end{proposition}

\subsection{Comparisons between de Rham modules and pairings}\label{sec: comparison DR}
Let $E_\rho$ be an abelian $t$-module over $\bar{k}$.
The de Rham module of $E_\rho$ was first introduced via \lq\lq bi-derivations\rq\rq\ in~\cite{Ge89, Yu90, BP02}. When $E_\rho$ is uniformizable, the de Rham pairing of $E_\rho$ was constructed using the Anderson generating function (cf.\ \cite{BP02}, see also \cite[(4.3.3) and (4.3.4)]{NP21}). Here we shall identify canonically the de Rham modules of $E_\rho$ with the one of the Hartl-Juschka dual $t$-motive $\eM(\rho)$ associated with $E_\rho$, and make a natural comparison between their de Rham pairings.

Recall that the de Rham module of $E_\rho$ is isomorphic to (see~\cite[\S 3.1]{BP02} and \cite[Proposition 4.1.3]{NP21})
\begin{equation}\label{E:Isom HdR and QtauM}
H_{\mathrm{dR}}(E_\rho,\bar{k}) \cong \frac{\tau \Mscr(\rho)}{(t-\theta) \cdot \tau \Mscr(\rho)}.
\end{equation}
From the canonical isomorphism $\tau \Mscr(\rho) \cong \Hom_{\bar{k}[t]}(\eM(\rho), \Omega_{\bar{k}[t]/\bar{k}})$,
we may identify $(t-\theta)\cdot \tau \Mscr(\rho)$ with $\Hom_{\bar{k}[t]}(\eM(\rho), (t-\theta)\Omega_{\bar{k}[t]/\bar{k}})$.
This induces the following isomorphism.

\begin{lemma}\label{lem: dR-iso} In the setting above, we have the following identifications:
\begin{multline*}
\hspace{-0.4cm} H_{\mathrm{dR}}(E_\rho,\bar{k}) \cong \frac{ \Hom_{\bar{k}[t]}(\eM(\rho), \Omega_{\bar{k}[t]/\bar{k}})}{\Hom_{\bar{k}[t]}(\eM(\rho), (t-\theta)\Omega_{\bar{k}[t]/\bar{k}})}
\cong \Hom_{\bar{k}[t]}\left(\eM(\rho),\frac{\Omega_{\bar{k}[t]/\bar{k}}}{(t-\theta)\Omega_{\bar{k}[t]/\bar{k}}}\right) \\
\cong \Hom_{\bar{k}}\left(\frac{\eM(\rho)}{(t-\theta)\eM(\rho)},\bar{k}\right)
= H_{\mathrm{dR}}(\eM(\rho),\bar{k}).
\end{multline*}
\end{lemma}

Assume that $E_\rho$ is uniformizable.
Let $\Lambda_\rho \subset \Lie(E_\rho)(\CC_\infty)$ be the period lattice of $E_\rho$.
For each $\lambda \in \Lambda_\rho$,
let $\Gcal_\lambda: \Mscr(\rho)\rightarrow \TT $ be the Anderson generating function in \eqref{eqn: AGF}. Recall that the de Rham pairing of $E_\rho$ is defined by (cf.\ \cite[(4.3.3) and (4.3.4)]{NP21})
\begin{equation}\label{eqn: dR-P t-module}
\begin{tabular}{ccl}
$H_{\mathrm{dR}}(E_\rho,\bar{k})\times\Lambda_\rho$ & $\longrightarrow$ & $\CC_\infty$, \\
\quad \quad \quad $(\delta, \lambda)$ & $\longmapsto$ & $[\delta,\lambda]\assign  \langle \tau m_\delta | \Gcal_\lambda \rangle\big|_{t=\theta}$,
\end{tabular}
\end{equation}
where $\tau m_\delta  \in \tau \Mscr(\rho)$, whose image in $H_{\mathrm{dR}}(E_\rho,\bar{k})$ is $\delta$ via~\eqref{E:Isom HdR and QtauM}.
On the other hand, for each $\delta \in H_{\mathrm{dR}}(E_\rho,\bar{k})$ and $\lambda \in \Lambda_\rho$,
let $\omega_\delta \in H_{\mathrm{dR}}(\eM(\rho),\bar{k})$ be the differential corresponding to $\delta \in H_{\mathrm{dR}}(E_\rho,\bar{k})$ under the isomorphism in Lemma~\ref{lem: dR-iso}, and
let $\gamma_\lambda \in H_{\mathrm{Betti}}(\eM(\rho))$ be the cycle
corresponding to $\lambda$ under the isomorphism in Lemma~\ref{lem: L(rho)-H(rho)}.
Identifying $\tau m_\delta$ as an element of $\Hom_{\TT^\dagger}(\TT^\dagger \otimes_{\bar{k}[t]}\eM(\rho), \TT^\dagger dt)$,
we have the following commutative diagram:
\[
\SelectTips{cm}{}
\xymatrix{
\TT^\dagger \otimes_{\bar{k}[t]} \eM(\rho) \ar[r]^>>>>>{\omega_\delta} \ar[d]_{\tau m_\delta} & \CC_\infty
\\
\TT^\dagger\, dt \ar[r]^{\sim}
& \TT^\dagger \ar[u]_{(\cdot)|_{t=\theta}}
}
\]
This gives us the comparison,
\[
\int_{\gamma_\lambda}\omega_\delta = \omega_{\delta}(\gamma_\lambda) = \langle \tau m_\delta | -\Gcal_\lambda\rangle|_{t=\theta} = -
\langle \tau m_\delta | \Gcal_\lambda \rangle\big|_{t=\theta} = -[\delta,\lambda],
\]
from which we conclude the following.

\begin{proposition}\label{prop: comp-dR-pairing}
Let $E_\rho$ be a uniformizable abelian $t$-module over $\bar{k}$.
We have the following comparison between the de Rham pairings of $E_\rho$ and $\eM(\rho)$:
\[
\SelectTips{cm}{}
\xymatrixrowsep{0.4cm}
\xymatrixcolsep{0.05cm}
\xymatrix{
[\ ,\ ]: & H_{\mathrm{dR}}(E_\rho,\bar{k})\ar@{<->}[dd]^{\rotatebox{90}{\scalebox{1}[1]{$\sim$}}}_{\text{\rm Lem.\ \ref{lem: dR-iso}}} & \times & \Lambda_\rho\ar@{<->}[dd]^{\rotatebox{90}{\scalebox{1}[1]{$\sim$}}}_{\text{\rm Lem.\ \ref{lem: L(rho)-H(rho)}}} \ar[rrr] & & & \CC_\infty \ar@{<->}[dd]^{\rotatebox{90}{\scalebox{1}[1]{$\sim$}}} & z \ar@{<->}[dd] \\
&&&& \\
\int : & H_{\mathrm{dR}}(\eM(\rho),\bar{k}) & \times & H_{\mathrm{Betti}}(\eM(\rho)) \ar[rrr] & & & \CC_\infty & -z\\
}
\]
\end{proposition}

\subsection{\texorpdfstring{Quasi-periods of CM abelian $t$-modules}{Quasi-periods of CM abelian t-modules}}
\label{sec: app CM t-mod}

Let $E_\rho$ be a CM abelian $t$-module with generalized CM type $(\eK,\Xi)$ over $\bar{k}$.
By Theorem~\ref{thm: Uniform}, we know that $E_\rho$ is always uniformizable.
Let $\Lambda_\rho$ be the period lattice of $E_\rho$.
For each $\delta \in H_{\mathrm{dR}}(E_\rho,\bar{k})$ and $\lambda \in \Lambda_\rho$, we call $[\delta,\lambda]$ the \emph{quasi-period with respect to $\delta$ associated to $\lambda$} (see~\cite[p.~113]{BP02}).
From the natural comparison between the de Rham pairings of $E_\rho$ and $\eM(\rho)$ in Proposition~\ref{prop: comp-dR-pairing}, we obtain the following:

\begin{theorem}\label{thm: CM t-mod}
Let $E_\rho$ be a CM abelian $t$-module with generalized CM type $(\eK,\Xi)$ over~$\ok$.
Let $\eK^+$ be the maximal totally real subfield of $\eK$.
Then
\begin{align}\label{eqn: trdeg}
1+\frac{([\eK:\eK^+]-1)}{[\eK:\eK^+]}\cdot [\eK:\FF_q(t)] &\geq
\trdeg_{\bar{k}}\, \bar{k}\bigl([\delta,\lambda] \bigm| \delta \in H_{\mathrm{dR}}(E_\rho,\bar{k}),\ \lambda \in \Lambda_\rho \bigr) \\
&\geq \rank_\ZZ I_\Xi^0. \notag
\end{align}
In particular, if the generalized CM type $\Xi$ of $\eK$ is non-degenerate, then
\[
\trdeg_{\bar{k}} \bar{k}\bigl([\delta,\lambda] \bigm|  \delta \in H_{\mathrm{dR}}(E_\rho,\bar{k}),\ \lambda \in \Lambda_\rho \bigr)
= 1+\frac{[\eK:\eK^+]-1}{[\eK:\eK^+]}\cdot [\eK:\FF_q(t)].
\]
\end{theorem}

\begin{proof}
From Proposition~\ref{prop: comp-dR-pairing},
we have that
\[
\bar{k}\bigl([\delta,\lambda] \bigm|  \delta \in H_{\mathrm{dR}}(E_\rho,\bar{k}),\ \lambda \in \Lambda_\rho \bigr)
= \bar{k}\bigl(p_\eK(\xi,\Xi) \bigm| \xi \in J_\eK\bigr).
\]
Hence \eqref{eqn: trdeg} follows immediately from Lemma~\ref{L:Algebraic relation} and Theorem~\ref{thm: lower-bound}.
When $\Xi$ is non-degenerate, one obtains
\[
\rank_\ZZ I_\Xi^0 = \rank_\ZZ I_\eK^0 = 1+\frac{[\eK:\eK^+]-1}{[\eK:\eK^+]}\cdot [\eK:\FF_q(t)].
\]
This completes the proof.
\end{proof}

\begin{remark}
(1) Let $E_\rho$ be a CM abelian $t$-module with a  generalized CM type $(\eK,\Xi)$.
Suppose that $\Xi$ is non-degenerate.
Then all of the algebraic relations among the quasi-periods of $E_\rho$ come from the following Legendre relation of the period symbols (cf.\ Remark~\ref{rem: Legendre}):
\begin{equation}\label{eqn: Legendre relation}
\prod_{\substack{\xi \in J_\eK \\ \pi_{\bX/\bX^+(\xi) = \xi^+}}} p_\eK(\xi,\Xi) \sim \tilde{\pi}^{\text{\rm wt}(\Xi)}, \quad \forall \xi^+ \in J_{\eK^+}.
\end{equation}

(2) Let $E_\rho = (\GG_{a/\ok},\rho)$ be a Drinfeld $\FF_q[t]$-module of rank $r$ over $\bar{k}$.
Suppose $E_\rho$ has CM by $O_\eK$, where $O_\eK$ is the integral closure of $\FF_q[t]$ in a separable imaginary field $\eK$ (see Remark~\ref{rem: CM fields}~(2)) and $[\eK:\FF_q(t)] = r$.
Extending $\rho$ to an $\FF_q$-algebra homomorphism $\rho: O_\eK \rightarrow \End_{\FF_q}(\GG_{a/\ok})$, its derivation $\partial \rho : O_\eK \rightarrow \Lie(\GG_{a/\ok})$ induces an embedding $\nu_\rho: \eK \hookrightarrow \CC_\infty$.
Let $\eM(\rho)$ be the Hartl-Juschka dual $t$-motive associated to $E_\rho$. Then $\eM(\rho)$ is a CM dual $t$-motive over $\ok$ with CM type $(\eK,\xi_\rho)$,
where $\xi_\rho \in J_\eK$ is the point corresponding to $\nu_\rho$.
As $\xi_\rho$ is non-degenerate, Theorem~\ref{thm: app CM t-mod} implies that
\[
\trdeg_{\bar{k}} \bar{k}\bigl([\delta,\lambda] \bigm|  \delta \in H_{\mathrm{dR}}(E_\rho,\bar{k}),\ \lambda \in \Lambda_\rho \bigr) = r.
\]
This coincides with \cite[Theorem~1.2.2]{CP12}.
\end{remark}

\subsection{\texorpdfstring{Periods of CM Hilbert-Blumenthal $t$-modules}{Periods of CM Hilbert-Blumenthal t-modules}}\label{sec: CM HB t-mod}
Now, let $\eK$ be a CM field and $\Xi$ be a CM type of $\eK$.
Suppose $E_\rho$ is a CM abelian $t$-module with CM type $(\eK,\Xi)$ over $\bar{k}$.
Let $\eK^+$ be the maximal totally real subfield of $\eK$, and write $\Xi = \sum_{i=1}^d \xi_i$, where $d = [\eK^+:\FF_q(t)]$.
For $1\leq i \leq d$, recall that $\Pfk_i$ is the maximal ideal of $O_{\bK}$ corresponding to $\xi_i$.
By Lemma~\ref{lem: comp Lie} and the equality~\eqref{eqn: sigma-M iso} in the appendix, we have the following identification:
\begin{equation}\label{eqn: Lie id}
\Lie(E_\rho) \cong \frac{\eM(\rho)}{\sigma \eM(\rho)} = \frac{\eM(\rho)}{\Ifk_\Xi \eM(\rho)} \cong \prod_{i=1}^d \frac{\eM(\rho)}{\Pfk_i \eM(\rho)} \cong \ok^d.
\end{equation}
Let $\xi_i^+\assign \pi_{\bX/\bX^+}(\xi_i)$ and $\nu_i^+ = \nu_{\xi_i^+} \in \Emb(\eK^+,\CC_\infty)$.
With respect to \eqref{eqn: Lie id}, the multiplication of $O_\eK^+$ on $\eM(\rho)$ induces a structure homomorphism (still denoted by) $\rho: O_\eK^+\rightarrow \End_{\rm ab.\ t-mod.}(E_\rho)$ such that the induced $\CC_\infty$-linear homomorphism $\partial \rho: O_{\eK^+}\rightarrow \End_{\CC_\infty}(\Lie(E_\rho)(\CC_\infty))$ is isomorphic to $\nu_1^+\oplus\cdots \oplus \nu_d^+$.
In other words, $E_\rho$ is actually a Hilbert-Blumenthal $O_\eK^+$-module (introduced in the beginning of Section~\ref{subsec: U-CM}).

Conversely, let $E_\rho$ be a Hilbert-Blumenthal $O_\eK^+$-module satisfying the property that $\rank_A(\Lambda_\rho) = [\eK:\FF_q(t)]$, and suppose that the structure homomorphism $\rho$ can be extended to a homomorphism (still denoted by $\rho$) $O_\eK \rightarrow \End_{\rm ab.\ t-mod.}(E_\rho)$.
Then $\eM(\rho)$ is pure, $\rank_{\bar{k}[t]}\eM(\rho) = [\eK:\FF_q(t)]$, and $\rho$ induces an $\FF_q[t]$-algebra embedding $O_\eK \hookrightarrow \End_{\bar{k}[t,\sigma]}(\eM(\rho))$.
Moreover, one checks that
\[
\sigma \eM(\rho) = \Ifk_{\Xi}\eM(\rho),
\]
where $\Xi$ is the CM type of $\eK$ fixed above. Therefore by Proposition~\ref{prop: being CM} we have that $E_\rho$ is a CM abelian $t$-module with CM type $(\eK,\Xi)$.

In~\cite{Yu89}, Yu showed that for any Hilbert-Blumenthal $t$-module $E_\rho$ defined over $\ok$, every coordinate of a nonzero period vector of $E_\rho$ is transcendental over $k$.
When $E_\rho$ is in particular a CM abelian $t$-module with non-degenerate CM type,
we extend Yu's result as follows.

\begin{theorem}\label{thm: CM HB t-mod}
Let $E_\rho$ be a CM abelian $t$-module of dimension $d$ with a non-degenerate CM type $(\eK,\Xi)$ over $\ok$.
Let $\eK^+$ be the maximal totally real subfield of $\eK$ and suppose $\eK \neq \eK^+$.
Fix an identification $\Lie(E_\rho) \cong \bar{k}^d$, where $d = [\eK^+:\FF_q(t)]$, and extend it to $\Lie(E_\rho)(\CC_\infty)\cong \CC_\infty^d$.
Given a nonzero period vector  $\lambda \in \Lambda_\rho \subseteq \Lie(E_\rho)(\CC_\infty)$, let $(\lambda_1,\ldots,\lambda_d)$ be the corresponding vector in $\CC_\infty^d$. Then the coordinates $\lambda_{1},\ldots,\lambda_{d}$ are algebraically independent over $\ok$.
\end{theorem}

\begin{proof}
Let $\Xi\assign \xi_{1}+\cdots+\xi_{d}$ be the given  CM type of $\eK$.
Without loss of generality, we may assume that the fixed identification $\Lie(E_\rho) \cong \bar{k}^d$ coincides with \eqref{eqn: Lie id}.
Then for a nonzero period vector $\lambda \in \Lambda_\rho$  with  corresponding coordinate vector  $(\lambda_1,\dots ,\lambda_d)\in \CC_\infty^d$,
we claim that
\[
\lambda_i \sim p_\eK(\xi_i,\Xi), \quad i=1,\dots ,d.
\]
For $1\leq i \leq d$, let $\xi_i^+\assign \pi_{\bX/\bX^+}(\xi_i)$.
Then $J_{\eK^+} = \{\xi_1^+,\dots ,\xi_d^+\}$, and the result follows from Remark~\ref{rem: Legendre}.

To prove the claim above, let $\eM = \eM(\rho)$.
Take the differential $\omega_{\eM,\xi_i} \in H_{\mathrm{dR}}(\eM,\bar{k})$ associated with $\xi_i$ for $1\leq i \leq d$,  each of which extends  to a nonzero $\CC_\infty$-linear functional on $\CC_\infty \otimes_{\bar{k}} (\eM/(t-\theta) \eM)$ factoring through
$\CC_\infty \otimes_{\bar{k}} (\eM/\Pfk_i \eM)$.
Let $\gamma_\lambda \in H_{\mathrm{Betti}}(\eM)$ be the cycle corresponding to $\lambda$ in Lemma~\ref{lem: L(rho)-H(rho)}.
We use the commutative diagram in Proposition~\ref{P:PeriodsDiagComm} together with the identification~\eqref{eqn: Lie id}  extended to $\Lie(E_\rho)(\CC_\infty) \cong \CC_\infty^d$  to conclude that
\[
\lambda_i \sim \int_{\gamma_\lambda} \omega_{\eM,\xi_i} \sim p_\eK(\xi_i,\Xi), \quad 1\leq i \leq d,
\]
as desired.
\end{proof}

\appendix

\section{Non-geometric CM fields and CM types}\label{secA: CM}

\subsection{Notation on curves}\label{secA: Notation X}

Let $\eK$ be a CM field over $\FF_q(t)$.
Suppose the constant field of $\eK$ is $\FF_{q^\ell}$ with $\ell>1$.
Then $\eK$ is geometric over $\FF_{q^\ell}(t)$.
Let $X$ be the smooth, projective, geometrically connected algebraic curve defined over $\FF_{q^\ell}$ with function field $\eK$.
We call $X$ the curve associated with $\eK$.
Fix an algebraically closed field $\KK$ with $k\subset \KK \subset \CC_\infty$.
Viewing $X$ as a scheme over $\FF_q$, let $\bX$ be the base change of $X$ from $\FF_q$ to $\KK$, i.e.,
\[
\bX = \KK \times_{\FF_q} X = \coprod_{i=0}^{\ell-1}\bX_{(i)},
\quad \text{ where } \quad
\bX_{(i)}\assign  \KK \underset{\Frob_q^{-i},\FF_{q^\ell}}{\times} X, \quad i \in \ZZ,
\]
i.e.\ the base change of $X$ from $\FF_{q^\ell}$ to $\KK$ with respect to the Frobenius twist
\[
\Frob_q^{-i}: \FF_{q^\ell}\hookrightarrow \KK,
\quad \varepsilon \longmapsto  \varepsilon^{1/q^i}, \quad \forall \varepsilon \in \FF_{q^\ell}.
\]
In particular, we have $\bX_{(i+\ell)} = \bX_{(i)}$.
Moreover, let $\bK_{(i)}\assign \KK(\bX_{(i)})$ be the function field of $\bX_{(i)}$.
The Frobenius twist  $(\alpha \mapsto \alpha^{q^{-1}})$ on the scalar field $\KK$ induces a field isomorphism $\bK_{(i)}\stackrel{\sim}{\rightarrow} \bK_{(i+1)}$ corresponding to the (geometric) Frobenius twist
\[
\bX_{(i+1)}\longrightarrow \bX_{(i)}, \quad x \longmapsto x^{(1)}, \quad \forall x \in \bX_{(i+1)},\ i \in \ZZ.
\]

Let $\eK^+$ be the maximal totally real subfield of $\eK$ (which is always geometric over $\FF_q(t)$), and let $X^+$ be the curve over $\FF_q$ with function field $\eK^+$ as before.
The embeddings $$\FF_q(t)\hookrightarrow \eK^+\hookrightarrow \eK$$ induce the $\FF_q$-morphisms $\pi_{X/X^+}:X\rightarrow X^+$, $\pi_{X^+/\PP^1}:X^+\rightarrow \PP^1$, and $\pi_{X/\PP^1}:X\rightarrow \PP^1$
so that
$$
\pi_{X/\PP^1} = \pi_{X^+/\PP^1}\circ \pi_{X/X^+}.
$$
Let $\bX^+$ be the base change of $X^+$ from $\FF_q$ to $\KK$, and denote by $\pi_{\bX/\PP^1}$, $\pi_{\bX^+/\PP^1}$, and $\pi_{\bX/\bX^+}$ the base change of respective morphisms to $\KK$.

\subsection{CM types}
We continue in the setting above with a CM field $\eK$ over $\FF_{q}(t)$.
Denote by
\[J_\eK\assign  \pi_{\bX/\PP^{1}}^{-1}(\theta) = \{\xi \in \bX(\CC_{\infty})\mid \pi_{\bX/\PP^{1}}(\xi) = \theta \},\]
the set of points of $\bX$ lying above $\theta$.
Let $\Emb(\eK,\CC_\infty)$ be the set of all $\FF_q$-algebra embeddings from $\eK$ into $\CC_\infty$ which send $t$ to $\theta$.
As in the geometric case, we have a bijection between $J_\eK$ and
$\Emb(\eK,\CC_\infty)$
given as follows:
Take $\xi \in J_\eK$, corresponding to an $\FF_q$-morphism
\[
\Spec(\CC_\infty)\rightarrow X.
\]
As $\xi$ is not $\overline{\FF}_q$-valued (where $\overline{\FF}_q$ is the algebraic closure of $\FF_q$ in $\CC_\infty$),
this morphism must factor through the generic point
$\Spec(\eK)\hookrightarrow X$.
Thus we obtain an $\FF_q$-morphism $\Spec(\CC_\infty)\rightarrow \Spec(\eK)$, which corresponds to an $\FF_q$-algebra embedding $\nu_\xi:\eK\hookrightarrow \CC_\infty$ sending $t$ to $\theta$ (as $\pi_{\bX/\PP^1}(\xi) = \theta$).
Conversely, for $\nu \in \Emb(\eK,\CC_\infty)$, the induced $\FF_q$-morphism
\[
\Spec(\CC_\infty)\stackrel{\nu^*}{\longrightarrow}
\Spec(\eK)\hookrightarrow X
\]
gives a $\CC_\infty$-valued point $\xi_\nu$ of $\bX$ so that $\pi_{\bX/\PP^1}(\xi_\nu) = \theta$ (as $\nu(t) = \theta$).
It is clear that
\[
\xi_{\nu_\xi} = \xi, \quad \forall \xi \in J_\eK, \quad \text{ and } \quad \nu_{\xi_\nu} = \nu, \quad \forall \nu \in \Emb(\eK,\CC_\infty).
\]
In particular, for $\xi \in J_\eK$, one has that
\[
\xi \in \bX_{(i)} \quad \text{ if and only if } \quad
\nu_\xi \big|_{\FF_{q^\ell}} = \text{Frob}_q^{-i}, \quad 0\leq i < \ell.
\]
The generalized CM types of a given CM field are defined as follows.

\begin{definition}\label{defn: A.CM type}
Let $\eK$ be a CM field over $\FF_q(t)$ with the maximal totally real subfield $\eK^+$.
Let $I_\eK$ be the free abelian group generated by elements in $J_\eK$.
Let $I_\eK^0$ be the subgroup of $I_\eK$ consisting of elements
$$ \Phi= \sum_{\xi \in J_\eK} m_\xi \xi, \quad m_\xi \in \ZZ,$$
satisfying
\[
 \sum_{\xi\in \pi_{\bX/\bX^+}^{-1}(\xi_1^+)} m_\xi= \sum_{\xi'\in \pi_{\bX/\bX^+}^{-1}(\xi_2^+)} m_{\xi'},
\quad \forall \xi_1^+, \xi_2^+ \in J_{\eK^+}.
\]
We put \[\wt(\Phi)\assign \sum_{\xi\in \pi_{\bX/\bX^+}^{-1}(\xi^+)} m_\xi \quad \text{ for a } \xi^+ \in J_{\eK^+}.\]
A  \textit{generalized CM type} of $\eK$ is a nonzero effective divisor in $I_\eK^0$. A \textit{CM type} of $\eK$ is a generalized CM type $\Xi$ of $\eK$ for which $\text{wt}(\Xi) = 1$.
\end{definition}

The following properties and definitions extend naturally to the case of non-geometric CM fields without difficulties.
\begin{itemize}
    \item $I_\eK^0$ is generated by all the CM types of $\eK$  (Proposition~\ref{prop: CM prop}).
    \item $I_\eK^0$ is invariant under the right action of $\Aut_{\FF_q(t)}(\eK)$ and the left action of $\Gal(k^{\rm sep}/k)$ (Lemma \ref{lem: right-action} and {Remark}~\ref{rem: left-action}).
    \item The restriction and inflation maps on CM types (Section~\ref{sub:Res and Inf}).
    \item The reduction points at infinity (Proposition~\ref{prop: reduction at infinity} and Definition~\ref{Def: IXi}).
\end{itemize}

Let $\eK$ be a CM field over $\FF_q(t)$ with the constant field $\FF_{q^\ell}$, and let $X$ be the curve over $\FF_{q^\ell}$ associated with $\eK$.
Given a CM type $\Xi = \sum_{\xi \in J_\eK}m_\xi \xi$ of $\eK$, we may also define the $\overline{\FF}_q$-rational divisor $\eI_\Xi$ as in Definition~\ref{Def: IXi}.
Recall that the Frobenius twist $(\alpha \mapsto \alpha^{q^{-1}})$ on the scalar field $\KK$ corresponds to the geometric Frobenius morphism on $\bX = \coprod_{i=0}^{\ell-1} \bX_{(i)}$,
which maps $\bX_{(i+1)}$ to $\bX_{(i)}$ by sending $x$ to $x^{(1)}$ for every $x \in \bX_{(i+1)}$.
Write
\[
\Xi = \sum_{i=0}^{\ell-1} \Xi_{(i)}, \quad \text{ where } \quad \Xi_{(i)} = \sum_{\xi \in J_\eK\cap \bX_{(i)}} m_\xi \xi.
\]
We may also decompose $\eI_\Xi$ into
\[
\eI_\Xi = \sum_{i=0}^{\ell-1} \eI_{\Xi_{(i)}}, \quad \text{ where } \quad \eI_{\Xi_{(i)}} \assign  \sum_{\xi \in J_\eK \cap \bX_{(i)}} m_\xi \overline{\infty}_\xi.
\]
Set
\begin{align}\label{eqn: Sharp}
 \Xi^{\#} & \assign  (\Xi_{(1)}^{(1-\ell)}+\cdots + \Xi_{(\ell-1)}^{(-1)}) + \Xi_{(0)} \\
\text{and}\quad
\eI^\#_\Xi\  & \assign  (\eI_{\Xi_{(1)}}^{(1-\ell)}+\cdots + \eI_{\Xi_{(\ell-1)}}^{(-1)}) + \eI_{\Xi_{(0)}}. \nonumber
\end{align}
Then $\Xi^\#$ and $\eI^\#_\Xi$ are divisors on $\bX_{(0)}$ and $\Xi^\# - \eI^\#_\Xi$ has degree zero.
Next, we shall apply Lang's isogeny theorem on the Jacobian of $\bX_{(0)}$ to introduce shtuka functions for a given CM type $\Xi$ and define the CM dual $t$-motives in the case of non-geometric CM fields.

\section{\texorpdfstring{CM dual $t$-motives for non-geometric case}{CM dual t-motives for non-geometric case}}
\label{secB: CM motive}

\subsection{Shtuka functions of CM types}
We retain the notations of Appendix~\ref{secA: CM} with a CM field $\eK$ over $\FF_q(t)$ whose constant field is $\FF_{q^\ell}$.

\begin{lemma}\label{lem: B-shtuka function}
Let $\Xi$ be a generalized CM type of $\eK$.
There exists a divisor $W \in \Div(\bX_{(0)})$ and a function $h \in \bK_{(0)} = \KK(\bX_{(0)})$ so that
\begin{equation}\label{eqn: shtuka}
\divv(h) = W^{(\ell)} - W + \Xi^\# - \eI_{\Xi}^\# \quad \in \Div(\bX_{(0)}).
\end{equation}
We call $h$ a \textit{shtuka function} associated to the divisor $W$ for the generalized CM type $(\eK,\Xi)$.
\end{lemma}

\begin{proof}
As the divisor $\Xi^\#-\eI_\Xi^\#$ is of degree zero on the curve $\bX_{(0)}$ defined over $\FF_{q^\ell}$,
the proof is the same as the geometric case in Lemma~\ref{L:Diff,W,h}.
\end{proof}

\subsection{\texorpdfstring{Anderson dual $t$-motives with CM types}{Anderson dual t-motives with generalized CM types}}

Let $O_\eK$ be the integral closure of $\FF_q[t]$ in $\eK$, and put
\[
O_\bK \assign  \KK\otimes_{\FF_q} O_\eK = \prod_{i=0}^{\ell-1} O_{\bK,{(i)}}, \text{ where } O_{\bK,{(i)}}\assign  \KK \underset{\Frob_q^{-i},\FF_{q^\ell}}{\otimes} O_\eK \subset \bK_{(i)},\ i \in \ZZ.
\]
For each $n \in \ZZ$, the $n$-th Frobenius twist on $O_\bK$ is given by:
\[
(c\otimes a)^{(n)}\assign  c^{q^n} \otimes a, \quad \forall c \in \KK \text{ and } a \in O_\eK.
\]
Hence for $\alpha_{(i)} \in O_{\bK,(i)}$ one has
\[
\alpha_{(i)}^{(-n)} \in O_{\bK,(i+n)}, \quad \forall n \in \ZZ.
\]
In particular, for each $\alpha \in O_\bK$,  writing $\alpha = (\alpha_{(0)},\ldots, \alpha_{(\ell-1)}) \in O_{\bK,{(0)}}\times\ldots \times O_{\bK,{(\ell-1)}}$,
we have the following expression:
\[
\alpha^{(-1)} = (\alpha_{(\ell-1)}^{(-1)},\alpha_{(0)}^{(-1)},\alpha_{(1)}^{(-1)},\ldots, \alpha_{(\ell-2)}^{(-1)}).
\]
Set
\[
\bU \assign  \bX - \pi_{\bX/\PP^1}^{-1}(\infty) = \Spec(O_\bK) = \coprod_{i=0}^{\ell-1} \bU_{(i)},
\]
where
$\bU_{(i)} = \Spec(O_{\bK,(i)}) = \bX_{(i)} \cap \bU$.
We may identify $\bK_{(i)} = \KK(\bX_{(i)})$ with the field of fractions of $O_{\bK,(i)}$ for $0\leq i < \ell$.
Let $\Xi$ be a generalized  CM type of $\eK$, and take $h \in \bK_{(0)}$ and $W \in \Div(\bX_{(0)})$ satisfying \eqref{eqn: shtuka}.
Define
\begin{equation}\label{e:Appendix M(W,h)}
\eM_{(W,h)}\assign  \prod_{i=0}^{\ell-1}\eM_{(W,h),(i)},
\end{equation}
where
\[
\eM_{(W,h),(0)}\assign  \Gamma\bigl(\bU_{(0)}, \Ocal_{\bX_{(0)}}(-W^{(\ell)})\bigr) \quad \subset \bK_{(0)},
\]
and
\[
\eM_{(W,h),(i)}\assign  \Gamma\bigl(\bU_{(i)}, \Ocal_{\bX_{(i)}}(-W^{(\ell-i)} + \sum_{j=1}^i \Xi_{(j)}^{(j-i)})\bigr) \quad \subset \bK_{(i)}, \quad 1\leq i < \ell.
\]
The $\sigma$-action on $\eM_{(W,h)}$ is given by: for $(m_{(0)},\ldots, m_{(\ell-1)}) \in \eM_{(W,h)}$,
\[
\sigma \cdot (m_{(0)},\ldots, m_{(\ell-1)}) \assign  \Bigl(h \cdot m_{(\ell-1)}^{(-1)}, m_{(0)}^{(-1)},\ldots, m_{(\ell-2)}^{(-1)}\Bigr).
\]
Then we have
\[
\sigma \cdot ( \alpha \cdot m) = \alpha^{(-1)}\cdot  (\sigma \cdot m), \quad \forall \alpha \in O_\bK \quad \text{ and } m \in \eM_{(W,h)}.
\]
Moreover,
\begin{multline}\label{eqn: sigma-M}
\sigma \eM_{(W,h)} = \Gamma\bigl(\bU_{(0)},\Ocal_{\bX_{(0)}}(-W^{(\ell)}-\Xi_{(0)})\bigr)  \\
 {}\times \prod_{i=1}^{\ell-1}
\Gamma\Bigl(\bU_{(i)},\Ocal_{\bX{(i)}}\bigl(-W^{(\ell-i)} + (\sum_{j=1}^{i}\Xi_{(j)}^{(j-i)})-\Xi_{(i)}\bigr)\Bigr).
\end{multline}
In particular, we point out that for $(m_{(0)},\ldots, m_{(\ell-1)}) \in \eM_{(W,h)}$,
\[
\sigma^\ell \cdot (m_{(0)},\ldots, m_{(\ell-1)})
= \Bigl(h \cdot m_{(0)}^{(-\ell)}, h^{(-1)}\cdot  m_{(1)}^{(-\ell)},\ldots, h^{(1-\ell)}\cdot m_{(\ell-1)}^{(-\ell)}\Bigr).
\]
Adapting the proof of Theorem~\ref{thm: CM dual}, we have  the following:

\begin{theorem}
The $\KK[t,\sigma]$-module $\eM = \eM_{(W,h)}$ constructed above is a dual $t$-motive over $\KK$ equipped with an $O_\eK$-module structure commuting with the $\sigma$-action.
In particular,
$$\rank_{\KK[t]}\eM = [\eK:\FF_q(t)] \quad \text{ and } \quad \rank_{\KK[\sigma]}\eM = \deg \Xi.
$$
\end{theorem}

\begin{proof}
As $\eM$ is a projective $O_\bK$-module of rank one, we have that $\eM$ is free over $\KK[t]$ with
$$
\rank_{\KK[t]}\eM = \rank_{\KK[t]} O_\bK = [\eK:\FF_q(t)].
$$
From the above description of $\sigma \eM$, we get that
$$
(t-\theta)^{\text{wt}(\Xi)} \cdot \eM \subset \sigma \eM.
$$
It suffices to show that $\eM$ is finitely generated as a $\KK[\sigma]$-module, and so $\eM$ is free over $\KK[\sigma]$ with
$$
\rank_{\KK[\sigma]}\eM = \dim_{\KK}(\eM/\sigma \eM) = \sum_{i=0}^{\ell-1}\deg \Xi_{(i)} = \deg \Xi \quad \quad \text{(by Lemma~\ref{L:degE})}.
$$
To prove this, take $\mu \in \NN$ sufficiently large so that $(\eI_{\Xi}^{\#})^{(\ell \mu)} = \eI_\Xi^\#$.
Put
$$
\bI\assign  \eI_\Xi^\#+ (\eI_{\Xi}^{\#})^{(-\ell)}+\cdots +
(\eI_{\Xi}^{\#})^{(\ell-\ell \mu)}.
$$
One has that $\bI^{(-\ell)} = \bI$, and the support of $\bI$ covers $\bX_{(0)} - \bU_{(0)} = \pi_{\bX/\PP^1}^{-1}(\infty)\cap \bX_{(0)}$.
For each $n \in \ZZ$, let
\begin{align*}
\eM_n \assign\ &  \Gamma(\bX_{(0)}, \Ocal_{\bX_{(0)}}(-W^{(\ell)} + n \bI) \\
&\times \prod_{i=1}^{\ell-1}
\Gamma\Bigl(\bX_{(i)}, \Ocal_{\bX_{(i)}}\bigl(-W^{(\ell-i)} + (\sum_{j=1}^i \Xi_{(j)}^{(j-i)})+n \bI^{(-i)}\bigr)\Bigr).
\end{align*}
Thus there exists $n_0 \in \ZZ$ so that
$\eM = \cup_{n\geq n_0} \eM_n$.

Put $\mathbf{\Xi}\assign  \Xi^\# + (\Xi^{\#})^{(-\ell)}+\cdots + (\Xi^{\#})^{(\ell-\ell \mu)}$.
Multiplication by $\sigma^{\ell \mu}$ on $\eM_n$ for each $n> n_0$ induces a $\KK$-vector space isomorphism
\begin{multline*}
\sigma^{\ell \mu} : \frac{\eM_n}{\eM_{n-1}} \stackrel{\sim}{\longrightarrow} {} \frac{\Gamma\bigl(\bX_{(0)}, \Ocal_{\bX_{(0)}}(-W^{(\ell)} - \mathbf{\Xi} + (n+1) \bI)\bigr)}{\Gamma\bigl(\bX_{(0)}, \Ocal_{\bX_{(0)}}(-W^{(\ell)} - \mathbf{\Xi} + n \bI)\bigr)} \\
\times \prod_{i=1}^{\ell-1}
\frac{\Gamma\Bigl(\bX_{(i)}, \Ocal_{\bX_{(i)}}\bigl(-W^{(\ell-i)} + (\sum_{j=1}^i \Xi_{(j)}^{(j-i)})-\mathbf{\Xi}^{(-i)}+ (n+1) \bI^{(-i)}\bigr)\Bigr)}{\Gamma\Bigl(\bX_{(i)}, \Ocal_{\bX_{(i)}}\bigl(-W^{(\ell-i)} + (\sum_{j=1}^i \Xi_{(j)}^{(j-i)})-\mathbf{\Xi}^{(-i)}+ n \bI^{(-i)}\bigr)\Bigr)}.
\end{multline*}
By the Riemann-Roch theorem (cf.\ \cite[p.\ 80, Cor.\ 2]{Mumford}), the right hand side in the above is isomorphic to $\eM_{n+1}/\eM_n$ (as $\KK$-vector spaces) under the natural inclusion map when $n$ is sufficiently large.
In other words, there exists a sufficiently large $n_1 \in \ZZ$ so that the multiplication by $\sigma^{\ell \mu}$ induces a $\KK$-vector space isomorphism
$$
\sigma^{\ell \mu} : \frac{\eM_n}{\eM_{n-1}} \stackrel{\sim}{\longrightarrow} \frac{\eM_{n+1}}{\eM_{n}} \quad \text{ for all $n\geq n_1$}.
$$
This implies that any $\KK$-basis of $\eM_{n_1}$ generates $\eM$ over $\KK[\sigma^{\ell \mu}]$, and so over $\KK[\sigma]$.
As $\eM_{n_1}$ is finite dimensional over $\KK$, the result follows.
\end{proof}

\begin{definition}\label{defnB: CM dual t-motives}
Let $\eK$ be a CM field over $\FF_q(t)$ and $\Xi$ be a (resp.\ generalized) CM type of $\eK$.
A dual $t$-motive (over $\KK$) which is isomorphic to $\eM_{(W,h)}$ for some shtuka function $h$ associated to $W$ for $(\eK,\Xi)$ is called a
\textit{CM dual $t$-motive with the $($resp.\ generalized$)$  CM type $(\eK,\Xi)$}.
\end{definition}

\begin{remark}\label{rem: sigma-M}
Let $\eK$ be a CM field over $\FF_q(t)$.
For each $\xi \in J_\eK$, let $\Pfk_\xi$ be the ideal of $O_\bK$ associated with $\xi$.
Given a generalized  CM type $\Xi = \sum_{\xi \in J_\eK}m_\xi \xi$ of $\eK$,
let $\Ifk_\Xi \assign  \prod_{\xi \in J_\eK}\Pfk_\xi^{m_\xi}$.
Let $\eM$ be a CM dual $t$-motive with generalized  CM type $(\eK,\Xi)$.
From the description of $\sigma \eM$ in \eqref{eqn: sigma-M} (when identifying $\eM \cong \eM_{(W,h)}$ for some pair $(W,h)$), we have that
\begin{equation}\label{eqn: sigma-M iso}
\sigma \eM = \Ifk_\Xi \cdot \eM \subset \eM.
\end{equation}
\end{remark}

Similarly we obtain the extension of Proposition~\ref{prop: tensor CM type}.

\begin{proposition}\label{propB: tensor}
Let $\eK$ be a CM field.
Given CM dual $t$-motives $\eM_i$ with generalized  CM types $(\eK,\Xi_i)$ for $i=1,...,n$,
the tensor product
\[
\eM_1\otimes_{O_\bK} \cdots \otimes_{O_\bK} \eM_n
\]
with $\sigma$-action defined by
\[
\sigma (m_1\otimes \cdots \otimes m_n) \assign  \sigma  m_1 \otimes \cdots \otimes \sigma  m_n, \quad \forall m_i \in \eM_i,\ i = 1,\ldots, n,
\]
is a CM dual $t$-motive with generalized  CM type $(\eK,\Xi_1+\cdots +\Xi_n)$.
\end{proposition}

\begin{proof}
Without loss of generality, we may assume that $n=2$ and
$\eM_i = \eM_{(W_i,h_i)}$, where
$h_i \in \KK(\bX_{(0)})$ is a shtuka function associated to $W_i \in \Div(\bX_{(0)})$ for $(\eK,\Xi_i)$ for $i=1,2$.
Put
$$
W \assign  W_1+W_2, \quad h \assign  h_1 h_2, \quad \text{ and } \quad \Xi \assign  \Xi_1+\Xi_2.
$$
One has
$$
\divv(h) = W^{(\ell)}-W+\Xi^\#-\eI_\Xi^\# \quad \in \Div(\bX_{(0)}).
$$
Let $\eM = \eM_{(W,h)}$.
For $0\leq i \leq \ell-1$,
the $\Ocal_{\bX_{(i)}}$-module isomorphism
\begin{multline*}
\Ocal_{\bX_{(i)}}\bigl(-W_1^{(\ell-i)}+(\sum_{j=1}^i (\Xi_{1})_{(j)}^{(j-i)})\bigr) \otimes_{\Ocal_{\bX_{(i)}}}  \Ocal_{\bX_{(i)}}\bigl(-W_2^{(\ell-i)}+(\sum_{j=1}^i (\Xi_{2})_{(j)}^{(j-i)})\bigr) \\
\stackrel{\sim}{\longrightarrow} \Ocal_{\bX_{(i)}}\bigl(-W^{(\ell-i)}+(\sum_{j=1}^i \Xi_{(j)}^{(j-i)})\bigr)
\end{multline*}
induces the following $O_\bK$-module isomorphism
\[
\SelectTips{cm}{}
\xymatrix{
\eM_1\otimes_{O_\bK} \eM_2 \ar[d]_{\Theta}^{\rotatebox{90}{\scalebox{1}[1]{$\sim$}}} & ((m_1)_{(1)},\ldots, (m_1)_{(\ell-1)})\otimes ((m_2)_{(1)},\ldots, (m_2)_{(\ell-1)}) \ar@{|->}[d] \\
\eM &
((m_1)_{(1)} (m_2)_{(1)},\ldots, (m_1)_{(\ell-1)}(m_2)_{(\ell-1)}).
}
\]
Moreover, it is straightforward to see that $\Theta$ is invariant under the action of $\sigma$.
Hence $\eM_1\otimes_{O_\bK}\eM_2$ is isomorphic to the CM dual $t$-motive $\eM$ with generalized  CM type $(\eK,\Xi)$, and the proof is complete.
\end{proof}

To end this section, we generalize Proposition~\ref{L:IsogW1W2} to arbitrary CM fields.

\begin{proposition}\label{propB: isogeny}
Any CM dual $t$-motives with the same generalized  CM type must be isogeneous.
\end{proposition}

\begin{proof}
Let $\eK$ be a CM field over $\FF_q(t)$ with constant field $\FF_{q^\ell}$.
Let $\eM_{(W_1,h_1)}$ and $\eM_{(W_2,h_2)}$ be two CM dual $t$-motives with the same generalized  CM type $(\eK,\Xi)$.
Take an $\FF_{q^\ell}$-rational divisor $D_0$ of the curve $X$ (associated with $\eK$) over $\FF_{q^\ell}$ with degree one (see~\cite[Cor.~VII.~5.5]{Weil}).
Viewing $D_0$ as a divisor of $\bX_{(0)}$, let
$$W = W_2 - W_1 - (\deg W_2-\deg W_1)\cdot D_0.$$
Then
$$
W^{(\ell)} - W = \divv(h_2/h_1) \quad \in \Div(\bX_{(0)}).
$$
Thus as in the geometric case, there exist an $\FF_{q^\ell}$-rational divisor $D_1$ of $\bX_{(0)}$ with degree $0$ and $f \in \KK(\bX_{(0)})^\times = \bK_{(0)}^\times$ such that
$W = D_1 + \divv(f)$, i.e.,
\begin{equation}\label{eqn: isogeny f}
\hspace{1.5cm} W_2-W_1 +\bigl(-D_1 -(\deg W_2-\deg W_1) \cdot D_0\bigr)= \divv(f) \quad \in \Div(\bX_{(0)}).
\end{equation}
Applying the Riemann-Roch Theorem on the curve $X$ over $\FF_{q^\ell}$ as in the geometric case, we may assume that $D_1$ has been chosen so that
\[
D_3 \assign -D_1-(\deg W_2-\deg W_1)\cdot D_0 \quad
\text{is effective on $\bU_{(0)}$.}
\]
Then $W+D_3 = \divv(f)$ and
\[
\divv(h_2/h_1) = W^{(\ell)} - W = \divv(f^{(\ell)}/f) \quad \in \Div(\bX_{(0)}).
\]
Up to a $\KK$-scalar multiple, we may choose $f$ so that
\[
h_2 \cdot f =  f^{(\ell)} \cdot h_1.
\]

Define $\Theta : \eM_{(W_1,h_1)} \rightarrow \eM_{(W_2,h_2)}$ as follows:
for $((m_1)_{(0)},\ldots, (m_1)_{(\ell-1)}) \in \eM_{(W_1,h_1)}$, set
\begin{align*}
 \Theta & \bigl((m_1)_{(0)},(m_1)_{(1)},\ldots, (m_1)_{(\ell-1)}\bigr) \\
& \assign
\bigl(f^{(\ell)}\cdot (m_1)_{(0)}, f^{(\ell-1)} \cdot (m_1)_{(1)},\ldots,  f^{(1)} \cdot (m_1)_{(\ell-1)}\bigr).
\end{align*}
From the equality \eqref{eqn: isogeny f}, the map  $\Theta$ is an injective $O_\bK$-module homomorphism and the cokernel is of finite dimension over $\KK$ (by Lemma~\ref{L:degE}).
Moreover,
\begin{align*}
\Theta {}&\bigl(\sigma \cdot ((m_1)_{(0)},(m_1)_{(1)},\ldots, (m_1)_{(\ell-1)})\bigr) \\
&= \Theta\bigl( h_1\cdot (m_1)_{(\ell-1)}^{(-1)}, (m_1)_{(0)}^{(-1)},\ldots, (m_1)_{(\ell-2)}^{(-1)}\bigr) \\
&= \bigl(f^{(\ell)} \cdot h_1
\cdot (m_1)_{(\ell-1)}^{(-1)}, f^{(\ell-1)} \cdot (m_1)_{(0)}^{(-1)},\ldots, f^{(1)}\cdot  (m_1)_{(\ell-2)}^{(-1)}\bigr) \\
&= \bigl(h_2 \cdot (f^{(1)}\cdot (m_1)_{(\ell-1)})^{(-1)}, (f^{(\ell)}\cdot  (m_1)_{(0)})^{(-1)},\ldots, (
f^{(2)} \cdot (m_1)_{(\ell-2)})^{(-1)}\bigr) \\
&= \sigma \cdot \bigl(f^{(\ell)}\cdot  (m_1)_{(0)},\ldots,  f^{(1)}\cdot (m_1)_{(\ell-1)}\bigr) \\
&= \sigma \cdot \Theta\bigl((m_1)_{(0)},(m_1)_{(1)},\ldots, (m_1)_{(\ell-1)})\bigr).
\end{align*}
Hence $\Theta:\eM_{(W_1,h_1)}\rightarrow \eM_{(W_2,h_2)}$ is an isogeny, and the statement holds.
\end{proof}

\section{\texorpdfstring{Uniformizability of CM dual $t$-motives}{Uniformizability of CM dual t-motives}}\label{AppendixSubsec: Unif. of CM dual t-motives}

\subsection{\texorpdfstring{Characterization of CM dual $t$-motives}{Characterization of CM dual t-motives}}

Fix an algebraically closed field $\KK$ with $k\subset \KK \subset \CC_\infty$.
We first characterize when a given dual $t$-motive is ``CM''.

\begin{lemma}\label{lemB: projectivity}
Let $\eK$ be a CM field over $\FF_q(t)$, and $O_\eK$ be the integral closure of $\FF_q[t]$ in $\eK$.
Given a dual $t$-motive $\eM$, suppose that $\rank_{\KK[t]}(\eM) = [\eK:\FF_q(t)]$ and that there is an $\FF_q[t]$-algebra embedding $O_\eK \hookrightarrow \End_{\KK[t,\sigma]}(\eM)$.
Then $\eM$ is projective of rank one over $O_\bK \assign  \KK\otimes_{\FF_q} O_\eK$.
\end{lemma}

\begin{proof}
Suppose that the constant field of $\eK$ is $\FF_{q^\ell}$. Let $e_0,\ldots, e_{\ell-1}$ be the idempotents in
$$O_\bK = \prod_{i=0}^{\ell-1}O_{\bK,{(i)}}, \quad \text{ where } \quad
O_{\bK,{(i)}}
\assign  \CC_\infty \underset{\Frob_q^{-i}, \FF_{q^\ell}}{\otimes} O_\eK,
$$
so that $O_\bK \cdot e_i = O_{\bK,(i)}$ for $0\leq i <\ell$.
Put $\eM_{(i)}\assign  e_i \cdot \eM$ and then we may write
$$\eM = \prod_{i=0}^{f-1}\eM_{(i)}.
$$
The Frobenius twisting permutes $e_0,\ldots, e_{\ell-1}$ by
\[
e_{i}^{(-1)} = e_{i+1}, \quad 0\leq i \leq \ell-2, \quad \text{ and } \quad e_{\ell-1}^{(-1)} = e_0.
\]
It follows that
\[
\sigma \cdot \eM_{(i)} \subset \eM_{(i+1)} \quad \text{ for } 0\leq i<\ell-1, \quad
\sigma \cdot \eM_{(\ell-1)} \subset \eM_{(0)}.
\]
As $\dim_{\KK} (\eM/\sigma \eM)$ is finite, we obtain the equalities
\[
\rank_{\KK[t]} \eM_{(0)} = \rank_{\KK[t]} \eM_{(1)} = \cdots = \rank_{\KK[t]} \eM_{(\ell-1)},
\]
whence
\[
\rank_{\KK[t]} M_{(i)} = \frac{1}{\ell} \cdot \rank_{\KK[t]}(\eM) = [\eK:\FF_{q^\ell}(t)], \quad 0\leq i < \ell.
\]
Therefore
$\eM_{(i)}$ is a projective $O_{\bK,(i)}$-module of rank one for $0\leq i \leq \ell-1$,
showing that $\eM$ is a projective $O_\bK$-module of rank one.
\end{proof}

\begin{remark}\label{rem: I-M and Xi-M}
Let $\eM$ be a dual $t$-motive satisfing the conditions in the above lemma.
As $\eM$ is projective of rank one over $O_\bK$ and $(t-\theta)^N \cdot \eM \subseteq \sigma \eM$ when $N$ is sufficiently large,
there exists a unique ideal $\Ifk_\eM$ of $O_\eK$ which is of the form
$$
\Ifk_\eM = \prod_{\xi \in J_\eK}\Pfk_\xi^{m_\xi}, \quad m_\xi \geq 0,
$$
where $\Pfk_\xi$ is the maximal ideal of $O_\bK$ corresponding to $\xi$ for each $\xi \in J_\eK$, so that $\sigma \eM = \Ifk_\eM\cdot \eM$.
In this case, we set
$$\Xi_\eM \assign  \sum_{\xi \in J_\eK}m_\xi \xi \quad \in I_\eK.$$
\end{remark}

\begin{proposition}\label{prop: being CM}
Let $\eK$ be a CM field over $\FF_q(t)$,
and $O_\eK$ be the integral closure of $\FF_q[t]$ in $\eK$.
Given a dual $t$-motive $\eM$, suppose the following conditions hold:
\begin{itemize}
    \item[(1)] $\rank_{\KK[t]}\eM = [\eK:\FF_q(t)]$;
    \item[(2)] there is an $\FF_q[t]$-algebra embedding $O_\eK \hookrightarrow \End_{\KK[t,\sigma]} \eM$;
    \item[(3)] $\Xi_\eM$ is a generalized CM type of $\eK$;
    \item[(4)] $\eM$ is pure (see \cite[Definition~2.4.6~(a)]{HJ20}).
\end{itemize}
Then $\eM$ is a CM dual $t$-motive with  generalized CM type $(\eK,\Xi_\eM)$.
\end{proposition}

\begin{proof}
As in Lemma~\ref{lemB: projectivity}, we identify
$$ \eM = \prod_{i=0}^{\ell-1} \eM_{(i)},$$
where $\eM_{(i)} = e_i \cdot \eM$ is projective of rank one over $O_{\bK,(i)}$ for $0\leq i < \ell$.
Write $$
\Xi \assign  \Xi_\eM = \sum_{\xi \in J_\eK} m_\xi \xi = \sum_{i=0}^{\ell -1} \Xi_{(i)}, \quad
\text{ where } \Xi_{(i)} = \sum_{\xi \in J_\eK\cap \bX_{(i)}} m_\xi \xi, \quad 0\leq i < \ell.
$$
Let $\Ifk_{\eM,(i)} \assign  (\prod_{\xi \in J_\eK \cap \bX_{(i)}} \Pfk_\xi^{m_\xi}) \cdot O_{\bK,(i)}$ for $0\leq i < \ell$.
As $\sigma \eM = \Ifk_\eM \cdot \eM$, we have that \begin{equation}\label{eqn: M(i)}
\sigma \eM_{(i)} = \Ifk_{\eM,(i+1)} \cdot \eM_{(i+1)}, \quad \textup{for $0\leq i \leq \ell-2$,} \quad \textup{and $\sigma \eM_{(\ell-1)} = \Ifk_{\eM,(0)} \cdot \eM_{(0)}$.}
\end{equation}

Take a suitable divisor $W_0$ of $\bX_{(0)}$ and $\bm \in \eM_{(0)}$ so that
\begin{equation}\label{eqn: M0-1}
\eM_{(0)} = \Gamma\bigl(\bU_{(0)}, \Ocal_{\bX_{(0)}}(-W_0^{(\ell)})\bigr) \cdot \bm.
\end{equation}
It follows from \eqref{eqn: M(i)} that for $1\leq i \leq \ell-1$,
\[
\eM_{(i)} = \Gamma\bigl(\bU_{(i)},\Ocal_{\bX_{(i)}}(-W_0^{(\ell-i)}+ \sum_{j=1}^i \Xi_{(j)}^{(j-i)})\bigr) \cdot (\sigma^i \cdot \bm).
\]
In particular,
\[
\sigma \cdot  \eM_{(\ell-1)}
= \Gamma
\bigl(\bU_{(0)},\Ocal_{\bX_{(0)}}(-W_0+ \sum_{j=1}^{\ell-1} \Xi_{(j)}^{(j-\ell)})\bigr)
\cdot (\sigma^\ell \cdot \bm),
\]
and so
\begin{align}\label{eqn: M0-2}
\eM_{(0)} &= \Ifk_{\eM,(0)}^{-1} \cdot \bigl(\sigma \cdot \eM_{(\ell-1)}\bigr) \\
&=
\Gamma
\bigl(\bU_{(0)},\Ocal_{\bX_{(0)}}(-W_0+ (\sum_{j=1}^{\ell-1} \Xi_{(j)}^{(j-\ell)})+\Xi_{(0)})\bigr)
\cdot (\sigma^\ell \cdot \bm) \notag \\
&=  \Gamma\bigl(\bU_{(0)},\Ocal_{\bX_{(0)}}(-W_0+\Xi^\#)\bigr)\cdot (\sigma^\ell \cdot \bm), \notag
\end{align}
where $\Xi^\#$ is defined in \eqref{eqn: Sharp}.

From the two expressions for $\eM_{(0)}$ in \eqref{eqn: M0-1} and \eqref{eqn: M0-2}, we have an $h_\eM \in \bK_{(0)}^\times$ such that
\[
\sigma^\ell \cdot \bm = h_\eM \cdot \bm,
\]
and
\[
\Gamma\bigl(\bU_{(0)},\Ocal_{\bX_{(0)}}(-W_0+\Xi^\#)\bigr)\cdot h_\eM = \Gamma\bigl(\bU_{(0)},\Ocal_{\bX_{(0)}}(-W_0^{(\ell)})\bigr).
\]
Therefore
\[
\divv(h_\eM) = W_0^{(\ell)}-W_0 + \Xi^\# - \eI^\#_{h_\eM} \quad \in \Div(\bX_{(0)}),
\]
where $\eI_{h_\eM}^\# \in \Div(\bX_{(0)})$ has support at the points of $\bX_{(0)}$ lying above $\infty$.
Finally, to apply the conditions (3) and (4), we invoke the following technical lemma (whose proof will then be given):

\begin{lemma}\label{lemB: Infinity part}
There exists $W_\infty \in \Div(\bX_{(0)})$ so that
$
W_\infty^{(\ell)} - W_\infty = \eI_\Xi^\# - \eI_{h_\eM}^\#
$.
\end{lemma}

Put $W\assign  W_0+W_\infty$.
We then have
$$
\divv(h_\eM) = W^{(\ell)}-W + \Xi^\# - \eI_\Xi^\# \quad \in \Div(\bX_{(0)}).
$$
In conclusion, we obtain a dual $t$-motive isomorphism
$\eM_{(W,h_\eM)}\stackrel{\sim}{\rightarrow} \eM$ defined by sending
$
(m_{(0)},\ldots, m_{(\ell-1)})
$
to $(m_{(0)}\cdot \bm,m_{(1)} \cdot \sigma \bm,\ldots,  m_{(\ell-1)}\cdot \sigma^{\ell-1} \bm)$
for every $(m_{(0)},\ldots, m_{(\ell-1)})$ in $\eM_{(W,h_\eM)}$.
Therefore $\eM$ is a CM dual $t$-motive with generalized CM type $(\eK,\Xi)$.
\end{proof}

\begin{proof}[Proof of Lemma~\ref{lemB: Infinity part}]
Regarding $X$ as a curve over $\FF_{q^\ell}$, we may view $\bX_{(0)}$ as the base change of $X$ from $\FF_{q^\ell}$ to $\KK$, and $\eK$ is a geometric extension over $\FF_{q^\ell}(t)$ (corresponding to an $\FF_{q^\ell}$-morphism from $X$ to $\PP^1$).
Let $\pi_{\bX_{(0)}/\PP^1}: \bX_{(0)}\rightarrow \PP^1$ be the base change from $\FF_{q^\ell}$ to $\KK$, which coincides with the composition of the inclusion $\bX_{(0)}\subset \bX$ and $\pi_{\bX/\PP^1}: \bX \rightarrow \PP^1$.
Let $\Ocal_{\eK,\infty}$ be the integral closure of $\FF_{q^\ell}[1/t]$ in $\eK$.
Since the smoothness of $X$ (over $\FF_{q^\ell}$) is preserved under the base change from $\FF_{q^\ell}$ to $\KK$ by \cite[Chapter III Proposition 10.1(b)]{Hartshorne},
the integral closure of $\KK[1/t]$ in $\bK_{(0)}$ $(=\KK(t)\otimes_{\FF_{q^\ell}(t)} \eK = \KK(t) \otimes_{\FF_{q^\ell}[1/t]}\Ocal_{\eK,\infty})$ can be identified with $\KK\otimes_{\FF_{q^\ell}} \Ocal_{\eK,\infty} = \KK[1/t]\otimes_{\FF_{q^\ell}[1/t]} \Ocal_{\eK,\infty}$.
For each $\overline{\infty} \in \pi^{-1}_{\bX_{(0)}/\PP^1}(\infty)$,
let
$\hat{\bK}_{\overline{\infty}}$ be the completion of $\bK_{(0)}$ ($=\KK(t)\otimes_{\FF_{q^\ell}(t)}\eK$) with respect to $\overline{\infty}$.
We have the following identifications,
\[
\hat{\bK}_{(0),\infty}\assign  \prod_{\overline{\infty} \in \pi_{\bX_{(0)}/\PP^1}^{-1}(\infty)}\hat{\bK}_{\overline{\infty}} = \KK(\!(1/t)\!)\otimes_{\KK(t)}\bK_{(0)}
= \KK(\!(1/t)\!)\otimes_{\FF_{q^\ell}[1/t]} \Ocal_{\eK,\infty}.
\]
Moreover, put
\[
\hat{\Ocal}_{\bX,\overline{\infty}}
\assign  \{\alpha \in \hat{\bK}_{\overline{\infty}}\mid \ord_{\overline{\infty}}(\alpha) \geq 0 \} \ \text{ and }\  \Mfk_{\overline{\infty}}\assign  \{\alpha \in \hat{\bK}_{\overline{\infty}}\mid \ord_{\overline{\infty}}(\alpha) > 0 \}.
\]
By \cite[Exercise 3.17 in \S 4.3]{Liu}) we obtain
\begin{align}\label{E:Ocal_{(0),infty}}
\hat{\Ocal}_{\bX_{(0)},\infty}\assign  \prod_{\overline{\infty} \in \pi_{\bX_{(0)}/\PP^1}^{-1}(\infty)} \hat{\Ocal}_{\bX,\overline{\infty}}
&=
\KK[\![1/t]\!] \otimes_{\KK[1/t]}(\KK[1/t]\otimes_{\FF_{q^\ell}[1/t]}\Ocal_{\eK,\infty}) \\
&= \KK[\![1/t]\!] \otimes_{\FF_{q^\ell}[1/t]} \Ocal_{\eK,\infty}. \notag
\end{align}

Now, the purity of $\eM$ (from the condition (4) in Proposition~\ref{prop: being CM}) says that there exists a  $\KK[\![1/t]\!]$-lattice $L$ of full rank in
$\KK(\!(1/t)\!)\otimes_{\KK[t]}\eM$
and positive integers $u,w$
satisfying that
\[
t^u L = \sigma^w L.
\]
After iterating the action of $t^u$ on the equation above $\ell$ times, we may assume that $\ell$ divides $w$.
Since the idempotents $e_0,\ldots, e_{\ell-1} \in O_\bK = \prod_{i=0}^{\ell-1} O_{\bK,(i)}$
acting on the module $\KK(\!(1/t)\!)\otimes_{\KK[t]}\eM$ actually commute with $\sigma^\ell$,
we get that $L_0\assign  e_0 \cdot L$ is a $\KK[\![1/t]\!]$-lattice of full rank in $\eM_{(0)}(\!(1/t)\!) \assign \KK(\!(1/t)\!)\otimes_{\KK[t]} \eM_{(0)}$ satisfying
\[
t^u L_0 = \sigma^w L_0.
\]
Notice that the action of $\sigma^\ell$ on $\eM_{(0)}(\!(1/t)\!)$ commutes with multiplication by elements in $\hat{\eK}_{(0),\infty}$.
After replacing $L_0$ by $\Ocal_{\eK,\infty} \cdot L_0$, we may assume that $L_0$ is an $\hat{\Ocal}_{\bX_{(0)},\infty}$-module via~\eqref{E:Ocal_{(0),infty}}.
Identifying $\eM_{(0)}(\!(1/t)\!)$ as a free $\hat{\bK}_{(0),\infty}$-module of rank one generated by $\bm$, we
can write $L_0 = \Lcal_0 \cdot \bm$, where
\[
\Lcal_0 = \prod_{\overline{\infty} \in \pi_{\bX_{(0)}/\PP^1}^{-1}(\infty)} \Mfk_{\overline{\infty}}^{\ell_{\overline{\infty}}} \subset \prod_{\overline{\infty} \in \pi_{\bX_{(0)}/\PP^1}^{-1}(\infty)} \hat{\bK}_{\overline{\infty}} \quad
\text{ for $\ell_{\overline{\infty}} \in \ZZ$.}
\]
Then
\[
(t^u \Lcal_0) \cdot \bm = t^u L_0 = \sigma^w L_0 = \Bigl(\prod_{i=0}^{(w/\ell)-1}h_\eM^{(-\ell i)}\Bigr) \cdot \Lcal_0^{(-w)} \cdot
\bm.
\]
Since $\pi_{\bX_{(0)}/\PP^1}^{-1}(\infty)$ is invariant under the $\ell$-th Frobenius twisting (as the maximal ideals $\Mfk_{\overline{\infty}}$ with $\overline{\infty} \in \pi_{\bX_{(0)}/\PP^1}^{-1}(\infty)$ are permuted by the $\ell$-th Frobenius twisting), there exists a positive integer $\ell_\infty$ so that
\[
\overline{\infty}^{(\ell_\infty)} = \overline{\infty}, \quad \forall \overline{\infty} \in \pi_{\bX_{(0)}/\PP^1}^{-1}(\infty).
\]
we may increase $u$ and $w$ to be multiples of $\ell_\infty$, and get $\Lcal_0^{(-w)} = \Lcal_0$.
Hence
\[
t^u \cdot \Lcal_0 = \Bigl(\prod_{i=0}^{(w/\ell)-1} h_\eM^{(-\ell i)}\Bigr) \cdot \Lcal_0 \quad \quad \subset \prod_{\overline{\infty} \in \pi_{\bX_{(0)}/\PP^1}^{-1}(\infty)}  \hat{\bK}_{\overline{\infty}}.
\]
The above equality implies that
\begin{equation}\label{eqn: purity}
u \cdot \pi_{\bX_{(0)}/\PP^1}^*(\infty)=  \sum_{i=0}^{(w/\ell)-1} (\eI_{h_\eM}^\#)^{(-\ell i)},
\end{equation}
where $\pi_{\bX_{(0)}/\PP^1}^*(\infty) \in \Div(\bX_{(0)})$ is the pull-back of the divisor $\infty$.
In particular, let $X^+$ be the curve over $\FF_q$ associated with the maximal totally real subfield $\eK^+$ of $\eK$, and let $\pi_{\bX^+/\PP^1}^{-1}(\infty) = \{\infty_1^+,\ldots, \infty_d^+\}$ where $d \assign [\eK^+:\FF_q(t)]$.
Write
\[
\eI_{h_\eM}^\# = \sum_{\overline{\infty} \in \pi_{\bX/\PP^1}^{-1}(\infty) \cap \bX_{(0)}} n_{\overline{\infty}} \overline{\infty} = \sum_{i=1}^d\left( \sum_{\substack{\overline{\infty} \in \bX_{(0)} \\ \pi_{\bX/\bX^+}(\overline{\infty}) =   \infty_i^+}}n_{\overline{\infty}} \overline{\infty}\right).
\]
As $\deg \eI_{h_\eM}^\# = \deg \Xi^\# = d \cdot {\rm wt}(\Xi)$, equality~\eqref{eqn: purity} assures us that
\[
\sum_{\substack{\overline{\infty} \in \bX_{(0)} \\ \pi_{\bX/\bX^+}(\overline{\infty}) =   \infty_i^+}}n_{\overline{\infty}} = {\rm wt}(\Xi), \quad \forall i = 1,\ldots, d.
\]

On the other hand,
we may write
\[
\eI_\Xi^\# = \sum_{\overline{\infty} \in \pi_{\bX/\PP^1}^{-1}(\infty)\cap \bX_{(0)}} m_{\overline{\infty}} \overline{\infty}
\ \text{ with}\!
\sum_{\substack{\overline{\infty} \in \bX_{(0)} \\ \pi_{\bX/\bX^+}(\overline{\infty}) =   \infty_i^+}}m_{\overline{\infty}} = {\rm wt}(\Xi), \ \forall i = 1,\ldots, d.
\]
For $1\leq i \leq d$, we take a point $\overline{\infty}_i \in \pi_{\bX/\bX^+}^{-1}(\infty_i^+) \cap \bX_{(0)}$.
Since $\infty_i^+$, as a place of $\eK^+$, is non-split in $\eK$, every point $\overline{\infty} \in  \pi_{\bX/\bX^+}^{-1}(\infty_i^+)\cap \bX_{(0)}$ must be equal to the Frobenius twist $(\overline{\infty}_i)^{(\ell \mu_{\overline{\infty}})}$ for some non-negative integer $\mu_{\overline{\infty}}$.
Therefore,
\begin{align*}
\eI_\Xi^\# - \eI_{h_\eM}^\# &=  \sum_{\overline{\infty} \in  \pi_{\bX/\PP^1}^{-1}(\infty) \cap \bX_{(0)}} (m_{\overline{\infty}}-n_{\overline{\infty}}) \overline{\infty} \\
&= \sum_{i=1}^d \left( \sum_{\substack{\overline{\infty} \in \bX_{(0)} \\ \pi_{\bX/\bX^+}(\overline{\infty}) =   \infty_i^+}} (m_{\overline{\infty}}-n_{\overline{\infty}}) ( \overline{\infty}_i^{(\ell \mu_{\overline{\infty}})} - \overline{\infty}_i)\right).
\end{align*}
Consequently we can take
\[
W_\infty \assign  \sum_{i=1}^d \left( \sum_{\substack{\overline{\infty} \in \bX_{(0)} \\ \pi_{\bX/\bX^+}(\overline{\infty}) =   \infty_i^+}} (m_{\overline{\infty}}-n_{\overline{\infty}}) \sum_{j=0}^{\mu_{\overline{\infty}}-1} \overline{\infty}_i^{(\ell j)}\right)
\]
and obtain
\[
W_\infty^{(\ell)} - W_\infty = \eI_\Xi^\# - \eI_{h_\eM}^\#.
\qedhere
\]
\end{proof}

\subsection{\texorpdfstring{Dual $t$-motives arising from Hilbert-Blumenthal $t$-modules}{Dual t-motives arising from Hilbert-Blumenthal t-modules}}

Throughout this section,
we take $\KK = \CC_\infty$.
Let $\eK/\FF_{q}(t)$ be a CM field with maximal totally real subfield $\eK^+$.
Put $d = [\eK^+:\FF_q(t)]$.
Let $\Xi = \xi_1+\cdots + \xi_d$ be a  CM type of $\eK$.
Notice that
$\nu_{\xi_1}\big|_{\eK^+},\ldots, \nu_{\xi_d}\big|_{\eK^+}$ give all the $\FF_q$-algebra embeddings from $\eK^+$ into $\CC_\infty$ sending $t$ to $\theta$.
Set
$$
\nu_\Xi:\eK \rightarrow \CC_\infty^d, \quad \nu_\Xi(\lambda) \assign  (\nu_{\xi_1}(\lambda),\ldots, \nu_{\xi_d}(\lambda)), \quad \forall \lambda \in \eK.
$$
Let $O_\eK$ be the integral closure of $\FF_q[t]$ in $\eK$. Then $\nu_\Xi(O_\eK)$ becomes a discrete $A$-lattice in $\CC_\infty^d$.

\begin{proposition}
There exists a Hilbert-Blumenthal $O_{\eK^+}$-module $E_{(\eK,\Xi)} = (\GG_a^d,\rho^\Xi)$ whose period lattice is $\nu_\Xi(O_\eK)$.
\end{proposition}

\begin{proof}
This is directly from \cite[Theorem~7~(II)]{Anderson86}.
\end{proof}

From \cite[Theorem 7 I]{Anderson86} we can actually extend $\rho^\Xi$ to an $\FF_q$-algebra homomorphism from $O_\eK$ into $\End_{\FF_q}(\GG^d_{a/\CC_\infty})$ so that $$\partial \rho^\Xi = \nu_{\xi_1}\oplus \cdots \oplus \nu_{\xi_d} \quad \text{ on $\Lie(\GG_a^d)$}.$$

Recall in Definition~\ref{defn: t-module} that we set
\[\Mscr(\rho^\Xi)\assign  \Hom_{\FF_q}({\GG_a^d}_{/\CC_\infty}, {\GG_a}_{/\CC_\infty}),
\]
together with the $O_\bK$-module structure given by
$$
(c \otimes a) \cdot m \assign  c \cdot m \circ \rho_\eK(a), \quad \forall c \in \CC_\infty,\ a \in O_\eK,\ m \in \Mscr(\rho_\eK),
$$
where $O_{\bK}=\CC_{\infty}\otimes_{\FF_{q}}O_{\eK}$.
As $E_{\rho^\Xi}$ is pure and uniformizable (from the definition of a Hilbert-Blumenthal $O_{\eK^+}$-module), the $t$-motive $\Mscr(\rho^\Xi)$ is pure and uniformizable
with
\begin{align*}
& \rank_{\CC_\infty[t]}\Mscr(\rho^\Xi) = [\eK:\FF_q(t)] \rassign  r \\
\text{and} \quad &
\rank_{\CC_\infty[\tau]}\Mscr(\rho^\Xi) = \dim_{\CC_\infty}\left(\frac{\Mscr(\rho^\Xi)}{\tau \Mscr(\rho^\Xi)}\right) = d
\end{align*}
(see \cite[Theorem 4 (III)]{Anderson86}).
In particular,
let
$$
\Ifk_\Xi \assign  \prod_{i=1}^d \Pfk_i \subset O_\bK,
$$
where $\Pfk_i$ is the maximal ideal of $O_\bK$ generated by $a - \nu_{\xi_i}(a)$ for $a \in O_\eK$ (i.e.\ the maximal ideal corresponding to $\xi_{i}$).
We have {the following:}
\begin{lemma}
\begin{equation}\label{Ifk Mscr}
\Ifk_\Xi  \cdot \Mscr(\rho^\Xi) \subset \tau  \cdot \Mscr(\rho^\Xi).
\end{equation}
\end{lemma}

\begin{proof}
It suffices to show that for any $a_{1},\ldots,a_{d}\in O_{\eK}$,
\[ (a_{1}-\nu_{\xi_{1}}(a_{1})\cdots (a_{d}-\nu_{\xi_{d}}(a_{d}))\cdot m \in \tau  \cdot \Mscr(\rho^\Xi), \quad \forall m \in  \Mscr(\rho^\Xi).    \]Now, given any $m\in \Mscr(\rho^\Xi)\cong \Mat_{1\times d}(\CC_{\infty}[\tau])$ we write
\[m=\sum_{j=0}^{s}v_{j} \tau^{j}\quad {\rm{for}} \quad v_{o},\ldots,v_{s}\in \Mat_{1\times d}(\CC_{\infty}). \]
Note that for any $a\in O_{\eK}$ and $\xi\in J_{\eK}$, we have
\[(a-\nu_{\xi}(a))\cdot m\assign  m \rho_{a}^{\Xi}- \nu_{\xi}(a)\cdot m=\left(\sum_{j=0}^{s}v_{j}\tau^{j} \right) \rho_{a}^{\Xi}-\nu_{\xi}(a) \left( \sum_{j=0}^{s}v_{j}\tau^{j}\right) \]
and hence the constant vector of $(a-\nu_{\xi}(a))\cdot m$ is given by
\[ v_{0}\cdot \left(\partial(\rho_{a}^{\Xi})-\nu_{\xi}(a)I_{d} \right) .\]
It follows that for any $a_{1},\ldots,a_{d}\in O_{\eK}$, the constant vector of
\[(a_{1}-\nu_{\xi_{1}}(a_{1}))\cdots (a_{d}-\nu_{\xi_{d}}(a_{d}))\cdot m \]
is given by
\[v_{0}\cdot\left(\partial(\rho_{a_{d}}^{\Xi})-\nu_{\xi_{d}}(a_{d})I_{d} \right)\cdots \left(\partial(\rho_{a_{1}}^{\Xi})-\nu_{\xi_{1}}(a_{1})I_{d} \right),\]which is zero because
\[ \partial(\rho_{a}^{\Xi})=\begin{pmatrix}\nu_{\xi_{1}}(a)& &\\
&\ddots&\\
& & \nu_{\xi_{d}}(a)
\end{pmatrix} \quad \forall a\in O_{\eK}.
\]
That is, $(a_{1}-\nu_{\xi_{1}}(a_{1})\cdots (a_{d}-\nu_{\xi_{d}}(a_{d}))\cdot m \in \tau  \cdot \Mscr(\rho^\Xi)$ as claimed.
\end{proof}

In \eqref{eqn: M(rho)} we define the Hartl-Juschka dual $t$-motive associated with $E_{(\eK,\Xi)}$ by
\[
\eM(\rho^\Xi)\assign \Hom_{\CC_\infty[t]}(\tau \Mscr(\rho^\Xi), \Omega_{\CC_\infty[t]/\CC_\infty}),
\]
where $\Omega_{\CC_\infty[t]/\CC_\infty} = \CC_\infty[t]dt$ is the module of K\"ahler differentials of $\CC_\infty[t]$ over $\CC_\infty$, together with the following $\sigma$-action
\[
(\sigma \cdot \phi)(\tau m)
\assign  \phi( \tau^2 m)^{(-1)}, \quad \forall \tau m \in \tau \Mscr(\rho^\Xi).
\]
The following result is known from \cite[Remark 2.4.4 (c), Proposition 2.4.9, and Proposition 2.4.17]{HJ20}:
\begin{proposition}\label{propC: PU-Mrho}
$\eM(\rho^\Xi)$ is pure and uniformizable with
\begin{align*}
& \rank_{\CC_\infty[t]} \eM(\rho^\Xi) = \rank_{\CC_\infty[t]} \Mscr(\rho^\Xi) = r\\
\text{and}\quad  &
\dim_{\CC_\infty} (\frac{\eM(\rho^\Xi)}{\sigma \eM(\rho^\Xi)}) = \dim_{\CC_\infty}(\frac{\Mscr(\rho^\Xi)}{\tau \Mscr(\rho^\Xi)}) = d.
\end{align*}
\end{proposition}
We note that the $O_\bK$-module structure on $\Mscr(\rho^\Xi)$ induces an $O_\bK$-module structure on $\eM(\rho_\eK)$ as follows:
for $\phi \in \eM(\rho^\Xi)$ and $\alpha \in O_\bK$,
\[
(\alpha \cdot \phi)(\tau m) \assign  \phi(\alpha \tau m), \quad \forall \tau m \in \tau \Mscr(\rho^\Xi).
\]
In particular, this induces an $\FF_q(t)$-algebra embedding $O_\eK\hookrightarrow \End_{\CC_\infty[t,\sigma]}(\eM(\rho^\Xi))$, and the following lemma holds:

\begin{lemma}\label{lem: CM type of M(rho)}
$\Ifk_\Xi \cdot \eM(\rho^\Xi) = \sigma \cdot \eM(\rho^\Xi)$.
\end{lemma}

\begin{proof}
By Lemma~\ref{lemB: projectivity}, we know that $\eM(\rho^\Xi)$ is projective of rank one over $O_\bK$.
As
$$
\dim_{\CC_\infty}\left(\frac{\eM(\rho^\Xi)}{\sigma \eM(\rho^\Xi)}\right) = d = \dim_{\CC_\infty}\left(\frac{O_\bK}{\Ifk_\Xi}\right) = \dim_{\CC_\infty}\left(\frac{\eM(\rho^\Xi)}{\Ifk_\Xi \eM(\rho^\Xi)}\right),
$$
it suffices to show that
$\Ifk_\Xi \cdot \eM(\rho^\Xi) \subseteq \sigma \cdot \eM(\rho^\Xi)$.
Given $\alpha \in \Ifk_\Xi$ and $\phi \in \eM(\rho^\Xi)$, define $\phi'_\alpha \in \eM(\rho^\Xi)$ by
$$
\phi'_\alpha (\tau m) \assign  \phi(\alpha m)^{(1)}, \quad \forall \tau m \in \tau   \Mscr(\rho^\Xi).
$$
The map $\phi'_{\alpha}$ is well-defined since $\alpha m \in \alpha \Mscr(\rho^\Xi) \subset \tau \Mscr(\rho^\Xi)$, and for $a \in \CC_\infty[t]$ one has
\begin{align*}
\phi'_\alpha(a \tau m) = \phi'_\alpha(\tau a^{(-1)} m ) & = \phi(\alpha a^{(-1)} m)^{(1)} \\
& = \Bigl(a^{(-1)} \phi(\alpha m)\Bigr)^{(1)} = a  \phi(\alpha m)^{(1)} = a  \phi'_\alpha( \tau m).
\end{align*}
We shall check that $\alpha \cdot \phi = \sigma \cdot \phi_\alpha' \in \sigma \cdot \eM(\rho^\Xi)$ and the result follows.
Indeed, given any $\tau m \in \tau \Mcal(\rho^\Xi)$, we have
$$
(\sigma \cdot \phi'_\alpha)(\tau m)
= \phi'_\alpha(\tau^2 m)^{(-1)}
=\Bigl(\phi(\alpha \tau m)^{(1)}\Bigr)^{(-1)} = \phi(\alpha \tau m) = (\alpha \cdot \phi)(\tau m),
$$whence
\[
\alpha \cdot \phi = \sigma \cdot \phi_\alpha' \subset \sigma \cdot \eM(\rho^\Xi).
\]
\end{proof}

From the characterization in Proposition~\ref{prop: being CM}, we obtain

\begin{corollary}
$\eM(\rho^\Xi)$ is a CM dual $t$-motive with the CM type $(\eK,\Xi)$.
\end{corollary}

\begin{proof}
By Proposition~\ref{propC: PU-Mrho} and Lemma~\ref{lem: CM type of M(rho)}, the dual $t$-motive $\eM(\rho^\Xi)$ satisfies the conditions (1)--(4) in Proposition~\ref{prop: being CM}.
Hence the result follows.
\end{proof}

\bibliographystyle{plain}

\end{document}